\documentclass{amsart}
\usepackage{mathrsfs, amsmath, graphics}
\usepackage{amssymb, comment}

\usepackage{graphicx}
\usepackage{enumerate}
\usepackage{epsfig}

\usepackage{amsfonts}

\usepackage{amssymb,amscd}
\usepackage{tikz}
\usepackage{verbatim}
\usepackage{enumerate}
\usepackage{booktabs}
\usepackage[all]{xy}
\usepackage{subfigure}

\usepackage{caption}
\usepackage{subfigure}
\usepackage{marginnote}

\newtheorem{thm}{Theorem}[section]

\newtheorem{cor}[thm]{Corollary}
\newtheorem{prop}[thm]{Proposition}

\theoremstyle{definition}
\newtheorem{defn}[thm]{Definition}

\newtheorem{question}[thm]{Question}

\theoremstyle{remark}
\newtheorem{remark}[thm]{Remark}
\usepackage{enumerate}
\numberwithin{equation}{section}

\theoremstyle{notation}

\usepackage{floatrow}
\newfloatcommand{capbtabbox}{table}[][\FBwidth]

\usepackage{lineno}

\usepackage{graphicx}
\usepackage{mathrsfs}
\usepackage{epsfig}
\usepackage{amsmath}
\usepackage{amsfonts}
\usepackage{amssymb}
\usepackage{amssymb,amscd}
\usepackage{tikz}
\usepackage{verbatim}
\usepackage{enumerate}
\usepackage{booktabs}
\usepackage[all]{xy}

\theoremstyle{definition}

\theoremstyle{remark}
\newtheorem{rem}[thm]{Remark}
\numberwithin{equation}{section}

\title[The moduli space of the modular group]{The moduli space of the modular group in  three-dimensional complex hyperbolic geometry}

\author{Jiming Ma}
\address{School of Mathematical Sciences, Fudan University, Shanghai, 200433, P. R. China}
\email{majiming@fudan.edu.cn}



\date{June 26, 2023}

\thanks{Jiming Ma was partially supported by  NSFC 12171092. \\
}

\begin{document}



\maketitle

\begin{abstract}
	
We study the  moduli space of discrete, faithful, type-preserving 	representations of the modular group  $\mathbf{PSL}(2,\mathbb{Z})$ into  $\mathbf{PU}(3,1)$. The entire moduli space $\mathcal{M}$ is a union of $\mathcal{M}(0,\frac{2\pi}{3},\frac{4\pi}{3})$, $\mathcal{M}(\frac{2\pi}{3},\frac{4\pi}{3},\frac{4\pi}{3})$ and some isolated points. This is the first Fuchsian group such that its $\mathbf{PU}(3,1)$-representations space  has been entirely constructed. 
Both  $\mathcal{M}(0,\frac{2\pi}{3},\frac{4\pi}{3})$ and  $\mathcal{M}(\frac{2\pi}{3},\frac{4\pi}{3},\frac{4\pi}{3})$ are  parameterized by  a square, where two opposite sides of the square correspond to representations of $\mathbf{PSL}(2,\mathbb{Z})$ into the smaller group $\mathbf{PU}(2,1)$.  
 In particular,  both  sub moduli spaces $\mathcal{M}(0,\frac{2\pi}{3},\frac{4\pi}{3} )$ and  $\mathcal{M}(\frac{2\pi}{3},\frac{4\pi}{3},\frac{4\pi}{3})$ interpolate the geometries studied in \cite{FalbelKoseleff:2002} and  \cite{Falbelparker:2003}.

\end{abstract}

\tableofcontents

\section{Introduction}\label{sec:intro}



\subsection{Motivation}\label{subsec:motivation}
Hyperbolic $n$-space ${\bf H}^{n}_{\mathbb R}$ is the unique complete simply
connected Riemannian $n$-manifold with all sectional curvatures $-1$.  Complex hyperbolic $n$-space ${\bf H}^{n}_{\mathbb C}$ is the unique complete simply
connected K\"ahler  $n$-manifold with all holomorphic  sectional curvatures $-1$.
The   holomorphic isometry group of ${\bf H}^{n}_{\mathbb C}$ is $\mathbf{PU}(n,1)$, the   orientation preserving isometry group of ${\bf H}^{n}_{\mathbb R}$ is $\mathbf{PO}(n,1)$. Besides  ${\bf H}^{n}_{\mathbb R}$ is a totally geodesic submanifold of ${\bf H}^{n}_{\mathbb C}$,  $\mathbf{PO}(n,1)$ is a natural subgroup of  $\mathbf{PU}(n,1)$.



Over the last sixty years the theory of Kleinian groups, that  is,   deformations of groups into $\mathbf{PO}(3,1)$,  has flourished because of its close connections with low dimensional topology and geometry. See Minsky's ICM talk 	\cite{Minsky:2006} for related topics and the reference.
There are also some celebrated works on deformations of groups into $\mathbf{PU}(2,1)$. See also Schwartz's ICM talk \cite{Schwartz-icm} for related topics and the reference.  But to the author's knowledge, there are very few results on deformations of groups into $\mathbf{PU}(3,1)$.

One of the most important questions in complex hyperbolic geometry is
the existence  of (infinitely many commensurable classes of)  non-arithmetic complex hyperbolic lattices  \cite{Margulis, Fisher:2021, DeligneMostow:1986, Deraux:2020, dpp:2016, dpp:2021}. This is notorious difficult comparing to its real hyperbolic counterpart	\cite{Gromov-PS:1992}. All of $\mathbf{PO}(2,1)$, $\mathbf{PO}(3,1)$  and $\mathbf{PU}(2,1)$ are subgroups of $\mathbf{PU}(3,1)$. It is reasonable that deformations of some well understood discrete groups in $\mathbf{PO}(2,1)$, $\mathbf{PO}(3,1)$ and  $\mathbf{PU}(2,1)$  into the larger group $\mathbf{PU}(3,1)$  may give   some discrete, but not faithful representations. Which in turn   have the opportunity to
give some new  ${\bf H}^3_{\mathbb C}$-lattices  as pioneered  by \cite{DeligneMostow:1986, dpp:2016, dpp:2021}.  In this paper, we study  representations of $\mathbf{PSL}(2,\mathbb{Z})$  into $\mathbf{PU}(3,1)$, with the potential motivations to construct non-arithmetic complex hyperbolic lattices and  non-trivial examples of  5-manifolds    admitting uniformizable CR-structures.

\subsection{Main results}\label{subsec:result}

 Let $\mathbf{PSL}(2,\mathbb{Z})$ be the {\it modular group}, which is isomorphic to the free product $\mathbb{Z}_{2} *\mathbb{Z}_{3}$ as an abstract group.  We fix  a presentation of $\mathbf{PSL}(2,\mathbb{Z})$:
	\begin{equation}\label{psl2z}
	\left\langle a_{0}, a_{1},a_{2} ~ \big| \begin{array}  {c} a_{2}^2= a_{1}^3=id, ~~ a_2a_1=a_0\end{array}\right\rangle.
\end{equation}
The modular group $\mathbf{PSL}(2,\mathbb{Z})$ is a discrete subgroup of $\mathbf{PSL}(2,\mathbb{R})$, the later is isomorphic to  $\mathbf{PO}(2,1)$, and $a_0$ is a parabolic element in  $\mathbf{PSL}(2,\mathbb{R})$. The quotient space of the hyperbolic plane ${\bf H}^{2}_{\mathbb R}$ by  $\mathbf{PSL}(2,\mathbb{Z})$ is the \emph{moduli surface}, which has two cone points and a cusp.  

It is well-known that  $\mathbf{PSL}(2,\mathbb{Z})$  is rigid in  $\mathbf{PO}(2,1)$ and in $\mathbf{PO}(3,1)$. But due to the remarkable works of   \cite{FalbelKoseleff:2002, Falbelparker:2003},  there are one-dimensional components of 
 discrete and faithful representations of  $\mathbf{PSL}(2,\mathbb{Z})$  into $\mathbf{PU}(2,1)$.  In this paper, we go on to  show   there are two-dimensional components of 
 discrete and faithful representations of  $\mathbf{PSL}(2,\mathbb{Z})$  into $\mathbf{PU}(3,1)$. 
 The entire moduli space $\mathcal{M}$ of  representations of  $\mathbf{PSL}(2,\mathbb{Z})$  into $\mathbf{PU}(3,1)$ is a union of $\mathcal{M}(0,\frac{2\pi}{3},\frac{4\pi}{3})$, $\mathcal{M}(\frac{2\pi}{3},\frac{4\pi}{3},\frac{4\pi}{3})$ and some isolated points.
 It is prominent that each of the sub  moduli spaces $\mathcal{M}(0,\frac{2\pi}{3},\frac{4\pi}{3})$ and  $\mathcal{M}(\frac{2\pi}{3},\frac{4\pi}{3},\frac{4\pi}{3})$ is  parameterized by a square, with two sides of the square  correspond to exactly the moduli spaces studied \cite{FalbelKoseleff:2002} and  \cite{Falbelparker:2003} respectively. So both sub moduli spaces $\mathcal{M}(0,\frac{2\pi}{3},\frac{4\pi}{3} )$ and  $\mathcal{M}(\frac{2\pi}{3},\frac{4\pi}{3},\frac{4\pi}{3})$ interpolate   the ${\bf H}^2_{\mathbb C}$-geometries of the modular group   in \cite{FalbelKoseleff:2002} and  \cite{Falbelparker:2003}  studied  twenty years ago.


More precisely, let ${\bf H}^{3}_{\mathbb C}$ be three dimensional complex hyperbolic space. For a representation $\rho:\mathbf{PSL}(2,\mathbb{Z}) \longrightarrow  \mathbf{PU}(3,1)$, we denote by $$A_{i}=\rho(a_{i})$$ for $i=0,1,2$.  Let $A_2$ be a $\pi$-rotation about a $\mathbb C$-line $\mathcal{L}$ in ${\bf H}^{3}_{\mathbb C}$, and $A_1$ be an order three  element in  $\mathbf{PU}(3,1)$ of  $(0, \frac{2\pi}{3}, \frac{4 \pi}{3})$ type or  $(\frac{2\pi}{3}, \frac{4 \pi}{3}, \frac{4 \pi}{3})$ type, see Section \ref{sec:modulispace} for more details. Then  $A_1$  leaves a complex hyperbolic plane $\mathcal{P}$  in ${\bf H}^{3}_{\mathbb C}$ invariant such that $A_1$ acts on  $\mathcal{P}$ as a  $(\frac{2\pi}{3}, \frac{4 \pi}{3})$ type elliptic element. We may consider the group $\langle A_2, A_1\rangle$ with the additional condition that $A_0=A_2A_1$ is parabolic. So we have a type-preserving  representation of $\mathbf{PSL}(2,\mathbb{Z})$ into $\mathbf{PU}(3,1)$. We assume $\mathcal{L} \cap \mathcal{P}$ is non-empty. There are some special cases:
\begin{enumerate} 
	\item If $\mathcal{L} \subset \mathcal{P}$, then $A_2$ also preserves $\mathcal{P}$ invariant. So 
$\langle A_2, A_1\rangle$ can be viewed as a group in $\mathbf{PU}(2,1)$ with one generator $A_2$ a  $\mathbb C$-reflection about a   $\mathbb C$-line in  $\mathcal{P}$, and $A_1$ is a regular element in $\mathbf{PU}(2,1)$. This is exactly the moduli space studied in \cite{Falbelparker:2003}. Where Falbel and Parker gave the complete classification of discrete and faithful representations in this moduli space. 
\item 
 If $\mathcal{L}$ intersects $ \mathcal{P}$ perpendicularly at a point, then $A_2$ also preserves $\mathcal{P}$ invariant. So 
$\langle A_2, A_1\rangle$ can be viewed as a group in $\mathbf{PU}(2,1)$ with one generator $A_2$ a  $\mathbb C$-reflection on a   point in  $\mathcal{P}$, and $A_1$ is a regular element in $\mathbf{PU}(2,1)$. This is exactly the moduli space studied in \cite{FalbelKoseleff:2002}. Where Falbel and Koseleff showed the entire moduli space  consists  of discrete and faithful representations. 
\end{enumerate}

	 We will consider the case  that there is an angle $\beta$ between $\mathcal{L}$ and  $ \mathcal{P}$. 
Our first result is 

\begin{thm}\label{thm:modularfmoduli} Let $\rho:\mathbf{PSL}(2,\mathbb{Z}) \longrightarrow  \mathbf{PU}(3,1)$ be a type-preserving,  discrete and faithful representation. Up to conjugacy in $\mathbf{PU}(3,1)$, $\rho$ lies in one of the fifteen components:
	
	\begin{itemize}
		
		\item  Thirteen of these components are points. Each of which corresponds to a rigid $\mathbb{C}$-Fuchsian representation;
		
		\item  One of the two  2-dimensional  components $\mathcal{M}(0,\frac{2\pi}{3},\frac{4\pi}{3})$ and  $\mathcal{M}(\frac{2\pi}{3},\frac{4\pi}{3},\frac{4\pi}{3})$.
	\end{itemize}
	
\end{thm}

Here the author remarks that we do not claim every $ \rho$ in $ \mathcal{M}(0,\frac{2\pi}{3},\frac{4\pi}{3})$ or in $\mathcal{M}(\frac{2\pi}{3},\frac{4\pi}{3},\frac{4\pi}{3})$ is discrete and faithful. In fact many representations in these two components are not  faithful. To determine the discreteness of a representation  is highly interesting and  difficult. See Theorem \ref{thm:modulardiscrete} for result in this direction.
 
Some representations in  $\mathcal{M}(0,\frac{2\pi}{3},\frac{4\pi}{3})$ and  $\mathcal{M}(\frac{2\pi}{3},\frac{4\pi}{3},\frac{4\pi}{3})$ are well-understood thanks to   anterior noteworthy work. More precisely, 
\begin{thm}\label{thm:modularmoduli2dim} The sub moduli spaces $\mathcal{M}(0,\frac{2\pi}{3},\frac{4\pi}{3})$ and $\mathcal{M}(\frac{2\pi}{3},\frac{4\pi}{3},\frac{4\pi}{3})$  both are parameterized by $\alpha, \beta \in [0, \frac{\pi}{2}]$. For $(\alpha,\beta) \in [0, \frac{\pi}{2}]^2$, we denote by $\rho=\rho(\alpha,\beta)$ the corresponding representation in any of them, and $\Gamma=\Gamma(\alpha,\beta)=\rho(\mathbf{PSL}(2,\mathbb{Z}))$.
	
	\begin{enumerate}
		
		\item  [(1).] When $\beta=0$,  $\Gamma$ preserves a ${\bf H}^{2}_{\mathbb C} \hookrightarrow {\bf H}^{3}_{\mathbb C}$ invariant. These representations are precisely the ones studied in \cite{Falbelparker:2003};
		
		\item  [(2).] When $\beta=\frac{\pi}{2}$,  $\Gamma$ also preserves a ${\bf H}^{2}_{\mathbb C} \hookrightarrow {\bf H}^{3}_{\mathbb C}$ invariant.  These representations are precisely the ones studied in \cite{FalbelKoseleff:2002};
		
		\item  [(3).] When $\alpha=\frac{\pi}{2}$, $\Gamma$  preserves a ${\bf H}^{1}_{\mathbb C} \hookrightarrow {\bf H}^{3}_{\mathbb C}$  invariant. Moreover, 
		any two  representations  in this arc are the same when restrict to this invariant $\mathbb{C}$-line. In particular, any	 representation   in this arc is discrete and faithful;
		
		\item [(4).]

		\begin{itemize}

			\item   [(4a).] For $\mathcal{M}(0,\frac{2\pi}{3},\frac{4\pi}{3})$,
		when $\alpha=0$,  $\Gamma$  preserves a ${\bf H}^{3}_{\mathbb R} \hookrightarrow {\bf H}^{3}_{\mathbb C}$ invariant. A representation  in this arc is discrete and faithful if and only if $\beta=\frac{\pi}{2}$;
		\item [(4b).] For $\mathcal{M}(\frac{2\pi}{3},\frac{4\pi}{3},\frac{4\pi}{3})$, when $\alpha=0$,  $\Gamma$  does not preserves a ${\bf H}^{3}_{\mathbb R} \hookrightarrow {\bf H}^{3}_{\mathbb C}$ invariant.
		\end{itemize}
		
			\end{enumerate}
			
\end{thm}

This is the first Fuchsian group such that its $\mathbf{PU}(3,1)$-representations  has been entirely constructed.

The author believes that the moduli space  $\mathcal{M}(\frac{2\pi}{3},\frac{4\pi}{3},\frac{4\pi}{3})$ is even more exotic than  $\mathcal{M}(0,\frac{2\pi}{3},\frac{4\pi}{3})$. One 
superficial reason is that the matrices representations are a little complicated than  the case of  $\mathcal{M}(0,\frac{2\pi}{3},\frac{4\pi}{3})$. Another reason is that Item   (4b) in Theorem \ref{thm:modularmoduli2dim}. Even  the non-discreteness of a  representation in $\mathcal{M}(\frac{2\pi}{3},\frac{4\pi}{3},\frac{4\pi}{3})$ when $\alpha=0$ and $\beta \in (0, \frac{\pi}{2})$ is not trivial.  


 For $\mathcal{M}(0,\frac{2\pi}{3},\frac{4\pi}{3})$ or $\mathcal{M}(\frac{2\pi}{3},\frac{4\pi}{3},\frac{4\pi}{3})$, and any $(\alpha,\beta) \in (0, \frac{\pi}{2})^2$, we have a  representation of   $\mathbf{PSL}(2,\mathbb{Z})$  into  $\mathbf{PU}(3,1)$, but not into  the smaller groups  $\mathbf{PO}(3,1)$ or $\mathbf{PU}(2,1)$.   The  discreteness of it is highly interesting, and is difficult to prove in general.
By  Chuckrow's theorem in complex hyperbolic setting \cite{CooperLongThistlethwaite:2007}, for $\mathcal{M}(0,\frac{2\pi}{3},\frac{4\pi}{3})$ or $\mathcal{M}(\frac{2\pi}{3},\frac{4\pi}{3},\frac{4\pi}{3})$, the set of $(\alpha,\beta) \in [0, \frac{\pi}{2}]^2$ which   corresponding to discrete and faithful representations of $\mathbf{PSL}(2,\mathbb{Z})$  is a closed set.  Moreover, from Theorem \ref{thm:modularmoduli2dim}, in the moduli space  $\mathcal{M}(0,\frac{4\pi}{3},\frac{4\pi}{3})$, for any fixed $\beta \in (0, \frac{\pi}{2})$, when $\alpha$ degenerates from $\frac{\pi}{2}$ to $0$, the	behaviors of the corresponding representations change dramatically:  from a discrete and faithful representation tending to a non-discrete or non-faithful representation. The last discrete and faithful representation in this path is very interesting. Since it is believed there should be accidental parabolic or geometrically infinite at this representation.


Our main
 result on discreteness  of representations of   $\mathbf{PSL}(2,\mathbb{Z})$  into  $\mathbf{PU}(3,1)$ is 
\begin{thm}\label{thm:modulardiscrete} For the moduli spaces  $\mathcal{M}(0,\frac{2\pi}{3},\frac{4\pi}{3})$:
	
	\begin{enumerate}
		
		\item  \label{item:modulardiscreterightbottom} There is a neighborhood $\mathcal{N}$ of the point $(\alpha, \beta)=( \frac{\pi}{2},0)$ in the square  $[0, \frac{\pi}{2}]^2$, such that for  any $(\alpha,\beta) \in \mathcal{N}$, the corresponding representation is discrete and  faithful;
			
			\item \label{item:modulardiscreterighttop}  There is a neighborhood $\mathcal{N}$ of the point $(\alpha, \beta)=(\frac{\pi}{2},\frac{\pi}{2})$ in the square  $[0, \frac{\pi}{2}]^2$, such that for  any $(\alpha,\beta) \in \mathcal{N}$, the corresponding representation is discrete and  faithful.
		\end{enumerate} 
\end{thm}

\begin{figure}
	\begin{center}
		\begin{tikzpicture}
		\node at (0,0) {\includegraphics[width=5cm,height=5cm]{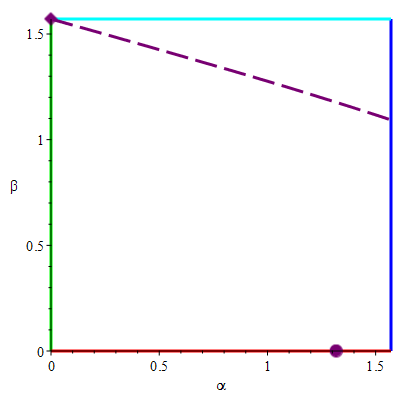}};
		\end{tikzpicture}
	\end{center}
	\caption{The moduli space $\mathcal{M}(0,\frac{2\pi}{3},\frac{4\pi}{3})$ of  representations of  $\mathbf{PSL}(2,\mathbb{Z})$ into $\mathbf{PU}(3,1)$. The cyan colored arc (that is,  $\beta=\frac{\pi}{2}$) and red colored arc (that is,  $\beta=0$)  correspond to the groups studied in \cite{FalbelKoseleff:2002} and  \cite{Falbelparker:2003} respectively. The  representations in the green colored arc (that is, $\alpha=0$) degenerate to ${\bf H}^{3}_{\mathbb R}$-geometry. There is only one point in the  green arc, that is, the purple diamond marked point, which corresponds to a discrete and faithful representation of $\mathbf{PSL}(2,\mathbb{Z})$.  The  representations in the blue colored arc (that is, $\alpha=\frac{\pi}{2}$) degenerate to ${\bf H}^{1}_{\mathbb C}$-geometry. So, in fact, there is only one representation in this arc. 
		The purple solid circle marked point  has an accidental parabolic element  as showed in  \cite{Falbelparker:2003}. The dash purple curve is the locus where the parabolic element $A_0$ has two eigenvalues.}
	\label{figure:modularmoduli}
\end{figure}


To our knowledge, Theorem \ref{thm:modulardiscrete} is the first  non-trivial  2-dimensional moduli space in   ${\bf H}^{3}_{\mathbb C}$-geometry such that its discreteness  has been studied.

See  Figure 	\ref{figure:modularmoduli} for the sub moduli space $\mathcal{M}(0,\frac{2\pi}{3},\frac{4\pi}{3})$:
	\begin{itemize}
	\item
	There is a  dash purple curve in Figure \ref{figure:modularmoduli}, which is the locus where the parabolic element $A_0$ has two eigenvalues (one has multiplicity three). One end point of this curve is $(\alpha,\beta)=(0, \frac{\pi}{2})$, and the other end point is  $(\alpha,\beta)=( \frac{\pi}{2}, \beta_0)$ with  $\cos(2 \beta_0)=-\frac{1}{\sqrt{3}}$. Due to technique reasons, our proof of discreteness in Theorem \ref{thm:modulardiscrete} does not  work  when  a representation lies in this curve (or near this curve), see Proposition \ref{prop:A0elliptopara}; 
	\item
	 In fact, we can prove the discreteness and faithfulness of representations in a bigger subset  in  $\mathcal{M}(0,\frac{2\pi}{3},\frac{4\pi}{3})$ which keeping away from this curve, see Corollary \ref{cor:modulardiscrete};
\item  We   believe that when $(\alpha,\beta)$ lies in this dash purple curve and  $\alpha$ is near to $\frac{\pi}{2}$, other methods will prove the discreteness and faithfulness of the corresponding representation.
	\end{itemize}


In Figure 	\ref{figure:modularmoduli}, there is a  purple diamond marked point with $(\alpha,\beta)=(0, \frac{\pi}{2})$ and   a purple solid circle marked point with $(\alpha,\beta)=(\arccos(\frac{1}{4}),0)$, both of them correspond to discrete and faithful representations of  $\mathbf{PSL}(2,\mathbb{Z})$ into $\mathbf{PU}(3,1)$.
It is reasonable to guess there should be a (topological) path $\mathcal{C}$ from  the purple diamond marked point to  the purple solid circle marked point, such that any  point in  $\mathcal{C}$  corresponds to an accidental  parabolic representation or geometrically infinite representation. See Subsection \ref{subsec:critical} for more details.


The proof of Theorem \ref{thm:modulardiscrete} via $\mathbb{C}$-spheres in ${\bf H}^{3}_{\mathbb C}$ is inspired by Falbel-Parker \cite{Falbelparker:2003}, but it is much more  involved. For example, the axis of the ellipto-parabolic element $A_0$ is very complicated, see Proposition \ref{prop:A0elliptopara}. On the other hand, the $\mathbb{C}$-spheres in the proof of Theorem \ref{thm:modulardiscrete} are very flexible, the ones chosen here seems  are  the simplest. For example the $\mathbb{C}$-sphere $S_{1}$ in Section \ref{sec:modulardiscrete} is $\iota_0$-invariant and  $\iota_1$-invariant. But in fact,  $\iota_0\iota_1$-invariance of $S_{1}$  is enough. 
The author even hopes that via  $\mathbb{C}$-spheres, one can prove discreteness and faithfulness of some groups in the moduli space with accidental parabolic elements (such as the case of  ${\bf H}^{2}_{\mathbb C}$-geometry in \cite{Falbelparker:2003} when $(\alpha,\beta)=(\arccos(\frac{1}{4}),0)$). 

\subsection{Questions on critical representations}  \label{subsec:critical}

For the sub moduli space  $\mathcal{M}(0,\frac{2\pi}{3},\frac{4\pi}{3})$, we define  $\mathcal{D}=\mathcal{D}(0,\frac{2\pi}{3},\frac{4\pi}{3})$ by $$\mathcal{D}= \left\{(\alpha,\beta) \in [0, \frac{\pi}{2}]^2 ~~~|~~~  \rho(\alpha,\beta) \text{~~~is ~~discrete and ~~ faithful} \right\}.$$
For a path from  a discrete and faithful representation to a non-discrete or non-faithful representation, there should be a point in this path such that some new  phenomenon occurs. Comparing to the well-understood  case of  ${\bf H}^{3}_{\mathbb R}$-geometry, accidental parabolicity or   geometrically infiniteness are two of the possible  new phenomena which can occur, see 	\cite{ASWY:2019}. Recall that for a representation $\rho$, if there is a loxodromic element $w \in \mathbf{PSL}(2,\mathbb{Z})$, such that $\rho(w) \in \mathbf{PU}(3,1)$ is a parabolic (or elliptic), then $\rho$ is called {\it accidental parabolic (or  accidental elliptic)}. 

So we  also define 
$\mathcal{C}=\mathcal{C}(0,\frac{2\pi}{3},\frac{4\pi}{3})$ by
 $$\mathcal{C}= \left\{(\alpha,\beta) \in \mathcal{D} ~~~|~~~  \rho(\alpha,\beta) \text{ is  accidental parabolic or  is geometrically infinite} \right\}.$$ 
 What in the author's mind is that a representation in $\mathcal{C}$ is discrete and faithful, but is critical in any reasonable sense. A priori, there may be new exotic phenomena in group representations into $\mathbf{PU}(3,1)$ the author does not know.
We will not give the definition of geometrical infiniteness here, the reader may refer to  \cite{Bowditch:1993}.  We have similar definitions of $\mathcal{D}$ and $\mathcal{C}$ for $\mathcal{M}(\frac{2\pi}{3},\frac{4\pi}{3},\frac{4\pi}{3})$.


We have proved that $\mathcal{D}$ contains 2-dimensional subset in Theorem \ref{thm:modulardiscrete} for the  moduli space  $\mathcal{M}(0,\frac{2\pi}{3},\frac{4\pi}{3})$.

\begin{question}\label{ques:critical} Study $\mathcal{D}$ and $\mathcal{C}$. More precisely,
	
	\begin{enumerate}
	
	\item [(a).] Is $\mathcal{D}$  connected?   Is $\mathcal{C}$  the whole frontier of $\mathcal{D}$?
	
			\item [(b).]  If (a) is true, then the most optimistic/simplest guess  is that  $\mathcal{C}$ is Jordan path asymptotic to $(0, \frac{\pi}{2})$ at one end, and 
			asymptotic to $(\arccos(\frac{1}{4}),0)$ at the other end;

			\item [(c).] If both (a) and  (b) are true, is $\mathcal{C}$ piece-wisely smooth?
		\end{enumerate}

\end{question}

In Question (a), by the  whole frontier of $\mathcal{D}$, we mean for any  $(\alpha_0, \beta_0) \in \mathcal{C}$,
is there  2-parameter of  $(\alpha, \beta) \in \mathcal{D}$ which converge to  $(\alpha_0, \beta_0)$?



There are some reasons that in  Question \ref {ques:critical} (b) we use the term  ``asymptotic" but not the term   ``connecting". Recall $A_{i}=\rho(a_{i})$ for $i=0,1,2$. Then both the words $A_0=A_2A_1$ and $A_2A_1A_2A^{-1}_1$
are Not responsible for the possible accidental parabolicity  phenomenon, as they were in ${\bf H}^{3}_{\mathbb R}$-geometry and ${\bf H}^{2}_{\mathbb C}$-geometry.  See Subsections \ref{subsection:012}  and  \ref{subsec:accidentalelliptic} respectively.
	 So even if the locus $\mathcal{C}$ is a Jordan path, we guess 
	 \begin{itemize}
	 	\item either it is fractal (the complicated guess);
	 	
	 	\item or it is  piece-wisely smooth (the optimistic guess). 
 	\end{itemize}
	 		 But even in the  optimistic guess,  $\mathcal{C}$  should be    locally infinite at least  near $(0, \frac{\pi}{2})$ and  $(\arccos(\frac{1}{4}),0)$. By this we mean that for example at any  neighborhood $\mathcal{N}$ of  $(\arccos(\frac{1}{4}),0)$ in $\mathcal{D}$, there should be infinitely many words $w_{i}$ in $\mathbf{PSL}(2,\mathbb{Z})$, such that $\rho(w_{i})$ is accidental parabolic at some point in $\mathcal{N}$.

\begin{question}\label{ques:approximation} 
	 It is also interesting even only a numerical  description/approximation  of  locus $\mathcal{C}$ in the moduli space $\mathcal{M}(0,\frac{2\pi}{3},\frac{4\pi}{3})$. 
	 
	 \end{question}
 
For similar approximation  questions in 3-dimensional hyperbolic geometry, see   \cite{EMartinS:2021}.
	 	In 	Figure \ref{figure:A2A1A2A1A2a1}, we have a rough approximation of the locus in $\mathcal{M}(0,\frac{2\pi}{3},\frac{4\pi}{3})$ where $(A_2A_1)^2A_2A^{-1}_1$ has a repeated eigenvalue, that is, the locus $\mathcal{H}((A_2A_1)^2A_2A^{-1}_1)=0$. Here $\mathcal{H}$ is the holy grail function, see Subsection \ref{subsection:holygrail}. If we  have good   approximations of the loci where $\mathcal{H}(\rho(w))=0$ for many $w \in \mathbf{PSL}(2,\mathbb{Z})$, then we shall have a  good   approximation of the locus  $\mathcal{C}$. But the locus $\mathcal{H}(\rho(w))=0$ is difficult to solve rigorous (or even approximate)  for  general word $w \in  \mathbf{PSL}(2,\mathbb{Z})$. 
	 	 So in the author's opinion, there are many exotic phenomena in the moduli spaces $\mathcal{M}(0,\frac{2\pi}{3},\frac{4\pi}{3})$ and $\mathcal{M}(\frac{2\pi}{3},\frac{4\pi}{3},\frac{4\pi}{3})$ need to understand.

\begin{question}\label{ques:122} For the moduli space $\mathcal{M}(\frac{2\pi}{3},\frac{4\pi}{3},\frac{4\pi}{3})$, we guess similar behavior occurs. 
	\end{question}

A first step should show there is a   2-dimensional subset of $\mathcal{M}(\frac{2\pi}{3},\frac{4\pi}{3},\frac{4\pi}{3})$ such that any representation in it 
is  discrete and faithful.


 {\bf Notation.} In this paper, inspired by \cite{Go}, by an $\mathbb{R}^3$-chain $\mathcal{L}$, we mean the boundary  of a totally geodesic $\mathbb{R}^3$-subspace  ${\bf H}^{3}_{\mathbb R} \hookrightarrow  {\bf H}^{3}_{\mathbb C}$. Topologically $\mathcal{L}$ is a $2$-sphere. Similarly, by a $\mathbb{C}^2$-chain $\mathcal{L}$, we mean the boundary  of a totally geodesic $\mathbb{C}^2$-subspace  ${\bf H}^{2}_{\mathbb C} \hookrightarrow  {\bf H}^{3}_{\mathbb C}$. Topologically $\mathcal{L}$ is a $3$-sphere. But by a  $\mathbb{C}$-sphere $S$, we mean $S$ is foliated by  1-parametrization of  $\mathbb{C}^2$-chains in  ${\bf H}^{3}_{\mathbb C}$ (with possibly that a $\mathbb{C}^2$-chain may degenerate to a point). So  topologically the $\mathbb{C}$-sphere $S$ is a $4$-sphere.


 {\bf The paper is organized as follows.} In Section \ref{sec:background}  we give well known background
 material on complex hyperbolic geometry. 
 In Section \ref{sec:modulispace}, we give the matrix representations of $\mathbf{PSL}(2,\mathbb{Z})$ into $\mathbf{PU}(3,1)$. We also prove Theorems \ref{thm:modularfmoduli} and \ref{thm:modularmoduli2dim} in Section 
  \ref{sec:modulispace}. Section \ref{sec:modulardiscrete} is devoted to the proof of Theorem  \ref{thm:modulardiscrete} via $\mathbb{C}$-spheres in   ${\bf H}^{3}_{\mathbb C}$.



\textbf{Acknowledgement}: The symbolic calculations in this paper were guided by massive 
numerical calculations on Maple. 


 \section{Background}\label{sec:background}

 The purpose of this section is to introduce briefly complex hyperbolic geometry. One can refer to Goldman's book \cite{Go} for more details.


\subsection{Complex hyperbolic space}  \label{subsec:chs}
Let ${\mathbb C}^{n,1}$  denote the vector space ${\mathbb C}^{n+1}$ equipped with the Hermitian
form  of signature $(n,1)$:
 $$\langle {\bf{z}}, {\bf{w}} \rangle= {\bf w}^* \cdot  H \cdot {\bf z},$$ where $ \bf{w}^*$ is the  Hermitian transpose  of  $\bf{w}$. We will take the Hermitian matrix $H$ be 
$$H_{s}=\begin{pmatrix}
0&0 & 1\\

0&I_{n-1} & 0\\
1&0 & 0\\
\end{pmatrix}, $$ 

or $$H_{b}=\begin{pmatrix}
I_{n} & 0\\
0 & -1\\
\end{pmatrix}, $$ and $I_{n-1}$ ($I_{n}$ resp. ) is the $(n-1) \times (n-1)$ ($n \times n$ resp. ) identity matrix.
Then
the Hermitian form divides ${\mathbb C}^{n,1}$ into three parts $V_{-}, V_{0}$ and $V_{+}$. Which are
\begin{eqnarray*}
  V_{-} &=& \{{\bf z}\in {\mathbb C}^{n+1}-\{0\} : \langle {\bf z}, {\bf z} \rangle <0 \}, \\
  V_{0} &=& \{{\bf z}\in {\mathbb C}^{n+1}-\{0\} : \langle {\bf z}, {\bf z} \rangle =0 \}, \\
  V_{+} &=& \{{\bf z}\in {\mathbb C}^{n+1}-\{0\} : \langle {\bf z}, {\bf z} \rangle >0 \}.
\end{eqnarray*}

Let $$[~~]: {\mathbb C}^{n+1}-\{0\}\longrightarrow {\mathbb C}{\mathbf P}^{n}$$ be  the canonical projection onto the  complex projective space. 
Then the {\it complex hyperbolic space} ${\bf H}^{n}_{\mathbb C}$ is the image of $V_{-}$ in ${\mathbb C}{\mathbf P}^{n}$
by the  map  $[~~ ]$. The  {\it ideal boundary} of ${\bf H}^{n}_{\mathbb C}$, or {\it  boundary at infinity}, is  the image of $V_{0}$ in
 ${\mathbb C}{\mathbf P}^{n}$, we denote it by $\partial {\bf H}^{n}_{\mathbb C}$.  In this paper, we will denote by   $$\mathbf{q}=(z_1,z_2, \cdots, z_{n+1})^{\top}$$ a vector in ${\mathbb C}^{n,1}$ (note that we use the boldface $\mathbf{q}$), and 
by $$q=[z_1,z_2, \cdots, z_{n+1}]^{\top}$$ the corresponding point in  ${\mathbb C}{\mathbf P}^{n}$. Here the 
superscript  $``\top"$  means the transpose of a vector.

In this paper, we only consider ${\bf H}^{3}_{\mathbb C}$, and its various sub-spaces. 
We will mainly use $H_{s}$ for the definition of  ${\bf H}^{n}_{\mathbb C}$ in this paper, it is also called the {\it Siegel model} of ${\bf H}^{n}_{\mathbb C}$. When we use  $H_{b}$, it is called the {\it ball model} of ${\bf H}^{n}_{\mathbb C}$.



There is a typical anti-holomorphic isometry $\iota$ of ${\bf H}^{n}_{\mathbb C}$. $\iota$ is given on the level of homogeneous coordinates by complex conjugate

\begin{equation}\label{antiholo}
\iota:\left[\begin{matrix} z_1 \\ z_2\\ \vdots \\ z_{n+1} \end{matrix}\right]
\longmapsto \left[\begin{matrix} \overline{z_1} \\
\overline{z_2} \\ \vdots \\\overline{z_{n+1}} \end{matrix}\right].
\end{equation}


\subsection{Siegel model of ${\bf H}^{n}_{\mathbb C}$}
When using the Hermitian matrix $H_{s}$, the standard lift $(z_1,z_2,\cdots, z_{n},1)^{\top}$ of $z=(z_1,z_2, \cdots, z_{n})\in \mathbb {C}^{n}$ is negative if and only if
$$z_1+|z_2|^2+\cdots +|z_{n}|^2+\overline{z}_1=2 \Re (z_1)+|z_2|^2+\cdots +|z_{n}|^2<0.$$  Thus ${\bf H}^{n}_{\mathbb C}$  is a paraboloid in ${\mathbb C}^{n}$, called the {\it Siegel domain}.
Its boundary $ \partial {\bf H}^{n}_{\mathbb C}$ satisfies
$$2{\rm Re}(z_1)+|z_2|^2+\cdots +|z_{n}|^2=0.$$
Therefore, the Siegel domain has an  analogue construction of the upper half space model for the  real hyperbolic space ${\bf H}^{n}_{\mathbb R}$.

Let $\mathfrak{N}=\mathbb{C}^{n-1}\times \mathbb{R}$ be the Heisenberg group with product
$$
[z,t]\cdot [w,s]=[z+w,t+s + 2 \cdot \Im\langle \langle  z, w\rangle\rangle],
$$
where $z=(z_1, z_2, \cdots, z_{n-1}) \in \mathbb{C}^{n-1}$, $w=(w_1, w_2, \cdots, w_{n-1}) \in \mathbb{C}^{n-1}$, and 
$$\langle \langle  z, w \rangle \rangle=w^{*}\cdot z$$ is the standard positive Hermitian form  on $\mathbb{C}^{n-1}$.
Then the boundary $\partial {\bf H}^{n}_{\mathbb C}$  of complex hyperbolic space can be identified to the union $\mathfrak{N} \cup \{\infty\}$, where $\infty$ is the point at infinity.
The \emph{standard lift} of $\infty$ and $q=[z_1,z_2, \cdots, z_{n-1},t]\in \mathfrak{N}$ in $\mathbb{C}^{n+1}$ are
\begin{equation}\label{eq:lift}
{\bf {\infty}}=\left[
\begin{array}{c}
1 \\
0 \\
\vdots  \\
0 \\
0 \\
\end{array}
\right]  \quad ~~ \text{and} ~~ \quad
{\bf{q}}=\left[
\begin{array}{c}
-(|z_1|^2+|z_2|^2+\cdots+|z_{n-1}|^2)+it \\
\sqrt{2}z_1 \\
\vdots  \\
\sqrt{2} z_{n-1} \\
1 \\
\end{array}
\right].
\end{equation}
So complex hyperbolic space with  its boundary ${\bf H}^{n}_{\mathbb C} \cup \partial {\bf H}^{n}_{\mathbb C}$ can be identified to $\mathfrak{N}\times{\mathbb{R}_{\geq 0}}\cup \infty$.
Any point $q=(z,t,u)\in{\mathfrak{N}}\times{\mathbb{R}_{\geq 0}}$ has the standard lift
$$
\quad
{\bf{q}}=\left[
\begin{array}{c}
-(|z_1|^2+|z_2|^2+\cdots+|z_{n-1}|^2)-u+it\\
\sqrt{2} z_1 \\
\vdots  \\
\sqrt{2} z_{n-1} \\
1 \\
\end{array}
\right].
$$
Here $(z_1,z_2,\cdots, z_{n-1},t,u)$ is called the \emph{horospherical coordinates} of $\overline{{\bf H}^{n}_{\mathbb C}} ={\bf H}^{n}_{\mathbb{C}} \cup \partial {\bf H}^{n}_{\mathbb{C}}$. The natural projection $$ \Pi _{V}: \mathfrak{N}=\mathbb{C}^{n-1} \times \mathbb{R} \rightarrow \mathbb{C}^{n-1}$$ is called the {\it vertical projection}.

\subsection{Totally geodesic submanifolds, reflections, rotations and Lagrangian inversions}

There are two kinds of totally geodesic submanifolds  in ${\bf H}^{n}_{\mathbb C}$:

\begin{itemize}
	
	\item Given any point $x \in {\bf H}^{n}_{\mathbb C}$, and a complex linear subspace $F$ of dimension $k$ in the tangent space $T_{x}{\bf H}^{n}_{\mathbb C}$, there is a unique complete holomorphic totally geodesic     submanifold contains $x$ and is tangent  to $F$. Such a holomorphic submanifold is called a \emph{$\mathbb{C}^{k}$-subspace} of ${\bf H}^{n}_{\mathbb C}$.  A $\mathbb{C}^{k}$-subspace is the intersection of a complex $k$-dimensional projective subspace in  $\mathbb{C}{\bf P}^{n}$ with ${\bf H}^{n}_{\mathbb C}$, and it is isometric to ${\bf H}^{k}_{\mathbb C}$.  A $\mathbb{C}^{2}$-subspace is also called a \emph{complex hyperbolic  plane}.  A $\mathbb{C}^{1}$-subspace is also called a \emph{complex geodesic} or a {$\mathbb{C}$-line}.

	
	\item  Corresponding to the compatible real structures on $\mathbb{C}^{n,1}$ are the real forms of ${\bf H}^{n}_{\mathbb C}$; that is, the maximal totally real totally geodesic subspaces of ${\bf H}^{n}_{\mathbb C}$, which has real dimension $n$.  A maximal totally real totally geodesic subspace of ${\bf H}^{n}_{\mathbb C}$  is the  fixed-point set  of an  anti-holomorphic isometry of ${\bf H}^{n}_{\mathbb C}$.  
	For the usual real structure, this submanifold is the real hyperbolic space ${\bf H}^{n}_{\mathbb R}$ with curvature $-\frac{1}{4}$. Any  totally geodesic subspace of a maximal totally real totally geodesic subspace is a  totally real totally geodesic subspace,  which  is the real hyperbolic space ${\bf H}^{k}_{\mathbb R}$ for some $k$, and it is  called a \emph{$\mathbb{R}^{k}$-subspace} of ${\bf H}^{n}_{\mathbb C}$.

\end{itemize}


Since the Riemannian sectional curvature of the  complex hyperbolic space
is non-constant, there are no totally geodesic hyperplanes in  ${\bf H}^{n}_{\mathbb C}$ when $n \geq 2$.



Let $\mathcal{L}$ be a  $\mathbb{C}^{n-1}$-subspace in ${\bf H}^{n}_{\mathbb C}$,  a {\it polar vector} of $\mathcal{L}$ is the unique vector (up to scaling) perpendicular to this $\mathbb{C}^{n-1}$-subspace with
respect to the Hermitian form. A polar vector of a $\mathbb{C}^{n-1}$-subspace belongs to  $V_{+}$ and each vector in $V_{+}$ corresponds to a $\mathbb{C}^{n-1}$-subspace of ${\bf H}^{n}_{\mathbb C}$.

Moreover, let $\mathcal{L}$ be a $\mathbb{C}^{n-1}$-subspace of ${\bf H}^{n}_{\mathbb C}$ with polar vector
${\bf n}\in V_{+}$,
then the order two  {\it complex reflection} fixing $\mathcal{L}$  is given by
$$
\iota_{\bf n}({\bf z}) = -{\bf z}+2\frac{\langle {\bf z}, {\bf n} \rangle}{\langle {\bf n},{\bf n}\rangle}{\bf n}.
$$
The $\mathbb{C}^{n-1}$-subspace $\mathcal{L}$, or its boundary at infinity, is also called the \emph{mirror} of $\iota_{\bf n}$. 

Similarly, let $\mathcal{L}$ be a  complex line in  ${\bf H}^{n}_{\mathbb C}$. For any $k$ at least $2$, there is an order $k$ holomorphic isometry $\iota$ with fixed point set exactly  $\mathcal{L}$ and it rotates $\frac{2 \pi}{k}$ in the normal bundle of  $\mathcal{L}$.  The isometry $\iota$ is called the $\frac{2 \pi}{k}$-rotation along  $\mathcal{L}$. Using ball model of ${\bf H}^{n}_{\mathbb C}$, a typical $\frac{2 \pi}{k}$-rotation is given by 
$$ \iota ([z_1,z_2, \cdots, z_n,1]^{\top})=[z_1, {\rm e}^{\frac{2 \pi \rm{i}}{k}}z_2, \cdots, {\rm e}^{\frac{2 \pi \rm{i}}{k}} z_n,1]^{\top}.$$

Now let  $\mathcal{L}$ be a totally geodesic     ${\bf H}^{n}_{\mathbb R} \hookrightarrow {\bf H}^{n}_{\mathbb C}$. There is an anti-holomorphic isometry $\iota$ of ${\bf H}^{n}_{\mathbb C}$ with fixed point set exactly $\mathcal{L}$. The isometry $\iota$ is called the  {\it 
Lagrangian inversion} or {\it $\mathbb{R}$-reflection} along $\mathcal{L}$. We have given an example of Lagrangian inversion in  (\ref{antiholo}). Other Lagrangian inversions can be obtained from this one by conjugacy 
with a  holomorphic isometry of ${\bf H}^{n}_{\mathbb C}$.

 Inspired by \cite{Go}, by an { \it $\mathbb{R}^k$-chain $\mathcal{L}$}, we mean the boundary  of a totally geodesic $\mathbb{R}^k$-subspace  ${\bf H}^{k}_{\mathbb R} \hookrightarrow  {\bf H}^{n}_{\mathbb C}$. Topologically $\mathcal{L}$ is a $(k-1)$-sphere. Similarly, by a { \it $\mathbb{C}^k$-chain $\mathcal{L}$}, we mean the boundary  of a totally geodesic $\mathbb{C}^k$-subspace  ${\bf H}^{k}_{\mathbb C} \hookrightarrow  {\bf H}^{n}_{\mathbb C}$. Topologically $\mathcal{L}$ is a $(2k-1)$-sphere. Using Siegel model of  ${\bf H}^{n}_{\mathbb C}$, that is in the Heisenberg group $\mathfrak{N}$,  a $\mathbb{C}^{n-1}$-chain $\mathcal{L}$  is \emph{vertical} if $\infty \in \mathcal{L}$. Otherwise, it is \emph{finite}. Let  $\iota$ be the order two $\mathbb{C}$-reflection with mirror a finite $\mathbb{C}^{n-1}$-chain $\mathcal{L}$, then $\iota(q _{\infty})$ is the \emph{center} of  $\mathcal{L}$.


\begin{thm} (Page 128 of   \cite{Go}) \label{thm:cform}In the Heisenberg group $\mathfrak{N}=\mathbb{C}^{n-1} \times \mathbb{R}$:
	\begin{itemize}
		\item If the  $\mathbb{C}^{n-1}$-chain $\mathcal{L}$ is vertical, then  $\Pi _{V}(\mathcal{L}- \infty)$ is a $\mathbb{C}$-affine subspace of  $\mathbb{C}^{n-1}$ with complex dimension $n-2$;

		\item If the  $\mathbb{C}^{n-1}$-chain $\mathcal{L}$ is finite, then  $\Pi _{V}$ maps $\mathcal{L}$ bijectively onto a Euclidean sphere in   $\mathbb{C}^{n-1}$ of real dimension $2n-3$  with radius  $R$.

	\end{itemize}
\end{thm}

By direct calculations, we have the following propositions. 
\begin{prop} \label{prop:center} If a finite  $\mathbb{C}^{2}$-chain $\mathcal{L}$ in  the Heisenberg group $\mathbb{C}^{2} \times \mathbb{R}$   has center $$(a+b {\rm i},~~c+d {\rm i},~~f)$$ and radius  $R$. Here all $a, b,c,d,f$ are real numbers. Then the polar vector of $\mathcal{L}$ is $$(R^2-(a^2+b^2+c^2+d^2)+f \rm{i},~~\sqrt{2}(a+b {\rm i}),~~\sqrt{2}( c+d {\rm i}),~~1)^{\top}.$$
\end{prop}

\begin{prop} \label{prop:cplane} If a finite  $\mathbb{C}^{2}$-chain $\mathcal{L}$ in  the Heisenberg group $\mathbb{C}^{2} \times \mathbb{R}$   has center $$(a+b {\rm i},~~c+d {\rm i},~~f)$$ and radius  $R$, then $\mathcal{L}$ is the intersection of the cylinder by  equation  $$ (x_1-a)^2+(y_1-b)^2+(x_2-c)^2+(y_2-d)^2=R^2$$
	and the hyperplane by equation 
	$$t+2ay_1-2bx_1+2cy_2-2dx_2-f=0 $$
	in $\mathbb{C}^{2} \times \mathbb{R}$ with coordinates $(x_1+y_1 {\rm i}, ~~x_2+y_2 {\rm i},~~t)$.
\end{prop}

\begin{prop} \label{prop:nonintersection} Let $\mathcal{L}_{j}$ be   finite $\mathbb{C}^{2}$-chains  in  the Heisenberg group $\mathbb{C}^{2} \times \mathbb{R}$  with   polar vectors  $$n_{j}=(R_{j}^2-(a_{j}^2+b_{j}^2+c_{j}^2+d_{j}^2)+f_{j} {\rm i},~~\sqrt{2}(a_{j}+b_{j} {\rm i}),~~\sqrt{2}( c_{j}+d_{j} {\rm i}),~~1)^{ \top}$$ for $j=1,2$. Here all $a_j, b_j,c_j,d_j,f_j,R_j$ are real numbers. Then the  $\mathbb{C}$-planes $\mathcal{P}_{j}$ with ideal boundary $\mathcal{L}_{j}$ for $j=1,2$ are disjoint if and only if $$|\langle n_1, n_2\rangle|^2- \langle n_1, n_1\rangle\langle n_2, n_2\rangle <0.$$
\end{prop}

A point $p$ in $\mathbb{C}^{2} \times \mathbb{R}$  is a  { \it degenerating   finite $\mathbb{C}^2$-chain},
 we view it as a finite $\mathbb{C}^2$-chain of  radius zero. The point $\infty$ at infinity of $\mathfrak{N}=\mathbb{C}^{2} \times \mathbb{R}$ is also a  { \it degenerating  $\mathbb{C}^2$-chain}.
 
 We can generalize the definition of $\mathbb{C}$-spheres in ${\bf H}^{2}_{\mathbb C}$-geometry in  \cite{FalbelZocca:1999, FalbelKoseleff:2002, Falbelparker:2003} to the following one. In particular, boundary of a bisector in Subsection \ref{subsection:bisector} is a very special case of a  $\mathbb{C}$-sphere.

\begin{defn} \label{prop:nonintersection} With possibly that a $\mathbb{C}^2$-chain may degenerate to a point, a    \emph{$\mathbb{C}$-sphere}  in  the Heisenberg group $\mathbb{C}^{2} \times \mathbb{R}$ is a topological $4$-sphere   with  a foliation by $\mathbb{C}^2$-chains.
\end{defn}

\subsection{The isometries} The complex hyperbolic space is a K\"{a}hler manifold of constant holomorphic sectional curvature $-1$.
We denote by $\mathbf{U}(n,1)$ the Lie group of $\langle \cdot,\cdot\rangle$ preserving complex linear
transformations and by $\mathbf{PU}(n,1)$ the group modulo scalar matrices. The group of holomorphic
isometries of ${\bf H}^{n}_{\mathbb C}$ is exactly $\mathbf{PU}(n,1)$. It is sometimes convenient to work with
$\mathbf{SU}(n,1)$.
The full isometry group of ${\bf H}^{n}_{\mathbb C}$ is
$$\widehat{\mathbf{PU}(n,1)}=\langle \mathbf{PU}(n,1),\iota\rangle,$$
where $\iota$ is the antiholomotphic isometry in (\ref{antiholo}).


Elements of $\mathbf{SU}(n,1)$ fall into three types, according to the number and types of  fixed points of the corresponding
isometry. Namely,
\begin{itemize}
	
	\item
	
	an isometry is {\it loxodromic} if it has exactly two fixed points 
	on $\partial {\bf H}^{n}_{\mathbb C}$;
	
	\item  an isometry is  {\it parabolic} if it has exactly one fixed point
	on $\partial {\bf H}^{n}_{\mathbb C}$;
	
	\item   an isometry is {\it elliptic}  when it has (at least) one fixed point inside ${\bf H}^{n}_{\mathbb C}$. 
	
\end{itemize}

An element $A\in \mathbf{SU}(n,1)$ is called {\it regular} whenever it has  distinct eigenvalues.
An elliptic $A\in \mathbf{SU}(n,1)$ is called  {\it special elliptic} if
it has a repeated eigenvalue. Complex reflection along a ${\bf H}^{n-1}_{\mathbb C} \hookrightarrow {\bf H}^{n}_{\mathbb C}$ is an example of  a special elliptic element in  $\mathbf{SU}(n,1)$.



For parabolic elements, we need a closer look on them. 

\begin{prop} (Proposition 3.4.1 of 	\cite{ChenGreenberg:1974}) \label{prop:elliptopara}Let  $A \in \mathbf{SU}(3,1)$ be a  parabolic isometry  of ${\bf H}^{3}_{\mathbb C}$.  
	\begin{enumerate}
		\item $A$ is unipotent if and only if   $A$ has  $1$ as its only eigenvalue;
		
		\item $A$ is ellipto-parabolic   if  it is not  unipotent. 
		
		\begin{itemize}
		
		\item Either $A$ has a unique invariant complex geodesic $\mathcal{L}$ upon which $A$ acts as a unipotent  automorphism  of  ${\bf H}^{1}_{\mathbb C}$. Moreover $A$ acts by  nontrivial unitary automorphism on the normal bundle of $\mathcal{L}$;
			\item Or  $A$ has a unique invariant  complex hyperbolic plane $\mathcal{P}$ upon which $A$ acts as a unipotent  automorphism  of   ${\bf H}^{2}_{\mathbb C}$. Moreover $A$ acts by  nontrivial unitary automorphism on the normal bundle of   $\mathcal{P}$.
			
		\end{itemize}
	\end{enumerate}
\end{prop}

The complex geodesic $\mathcal{L}$, or  complex hyperbolic plane $\mathcal{P}$, in Proposition  \ref{prop:elliptopara} is  called the \emph{axis} of the  ellipto-parabolic $A$.

\subsection{The holy grail function} \label{subsection:holygrail}


Gongopadhyay-Parker-Parsad  \cite{GongopadhyayPP}  generalized Goldman's function in ${\bf H}^{2}_{\mathbb C}$-geometry  to ${\bf H}^{3}_{\mathbb C}$-geometry, which is important to test accidental parabolicity and accidental ellipticity.

For a  matrix $A \in \mathbf{SU}(3,1)$, we denote by $$\tau(A)=tr(A)$$ the trace of $A$, and $$\sigma(A)=\frac{tr^2(A)-tr(A^2)}{2} \in \mathbb{R}.$$ Then  the characteristic polynomial of $A$ is $$\chi_{A}(X)=X^4- \tau X^3 + \sigma X^2 - \overline{\tau} X+1.$$
For $A \in \mathbf{SU}(3,1)$, consider the function
\begin{equation}\label{HGfunction}
\mathcal{H}(A)=\mathcal{H}(\tau, \sigma)=4(\frac{\sigma^2}{3}-|\tau|^2+4)^3-27(\frac{2 \sigma^3}{27}-\frac{|\tau|^2 \sigma}{3}-\frac{8 \sigma}{3}+(\tau^2+\overline{\tau}^2))^2,
\end{equation}
which is exactly the resultant of the characteristic polynomial of $A$, and it is called the  { \it holy grail function} of $A$.
Gongopadhyay-Parker-Parsad \cite{GongopadhyayPP} classified the dynamical action of $A \in \mathbf{SU}(3,1)$ on $\overline{{\bf H}^{3}_{\mathbb C}}$.
\begin{thm} (Gongopadhyay-Parker-Parsad  \cite{GongopadhyayPP}) \label{thm:holy}For $A \in \mathbf{SU}(3,1)$:
	\begin{itemize}
		\item $A$ is regular elliptic if and only if  $\mathcal{H}(A) > 0$;
		
		\item $A$ is regular loxodromic  if and only if  $\mathcal{H}(A)< 0$;

		\item $A$ has a repeated eigenvalue if and only if   $\mathcal{H}(A)=0$. 
		
			\begin{itemize}
			\item	If $\mathcal{H}(A)=0$ and $A$ is diagnalisable, then $A$ is either elliptic or loxodromic;
			\item  If $\mathcal{H}(A)=0$ and $A$ is not diagnalisable, then $A$ is parabolic.
			\end{itemize}
	\end{itemize}
\end{thm}


 \subsection{Bisectors and  spinal spheres}\label{subsection:bisector}
 
Boundaries of bisectors are   special cases of  $\mathbb{C}$-spheres. Moreover, we will use portions of  bisectors to construct $\mathbb{C}$-spheres in Section \ref{sec:modulardiscrete}. 
So in this subsection, we will describe a convenient set of
 coordinates for bisectors, deduced from  the slice decomposition of a bisector.

 \begin{defn} Given two distinct points $q_0$  and $q_1$ in ${\bf H}^{n}_{\mathbb C}$ with the same norm (e.g. one could
 	take lifts $\mathbf{q_0},\mathbf{q_1}$ of them such that $\langle \mathbf{q_0},\mathbf{q_0}\rangle=\langle \mathbf{q_1},\mathbf{q_1}\rangle= -1$). The \emph{bisector} $\mathcal{B}(q_0,q_1)$ is the projectivization  of the set of negative   vector $\mathbf{x}$ in ${\mathbb C}^{n, 1}$
 	with
 	$$|\langle \mathbf{x}, \mathbf{q_0}\rangle|=|\langle  \mathbf{x}, \mathbf{q_1}\rangle|.$$
 \end{defn}

 The  {\it spinal sphere} of the bisector $\mathcal{B}(p_0,p_1)$ is the intersection of $\partial {\bf H}^{n}_{\mathbb C}$ with the closure of $\mathcal{B}(p_0,p_1)$ in $\overline{{\bf H}^{n}_{\mathbb C}}= {\bf H}^{n}_{\mathbb C}\cup \partial { {\bf H}^{n}_{\mathbb C}}$. The bisector $\mathcal{B}(p_0,p_1)$ is a topological $(2n-1)$-ball, and its spinal sphere is a $(2n-2)$-sphere.
 The  {\it complex spine} of $\mathcal{B}(p_0,p_1)$ is the complex line through the two points $p_0$ and $p_1$. The {\it real spine} of $\mathcal{B}(p_0,p_1)$
 is the intersection of the complex spine with the bisector itself, which is a (real) geodesic. The real spine is the locus of points inside the complex spine which are equidistant from $p_0$ and $p_1$.

 Bisectors are not totally geodesic, but they have a very nice foliation by two different families of totally geodesic submanifolds. Mostow \cite{Mostow:1980}  showed that a bisector is the preimage of the real
 spine under the orthogonal projection onto the complex spine. The fibres of this projection are complex planes ${\bf H}^{n-1}_{\mathbb C} \hookrightarrow {\bf H}^{n}_{\mathbb C}$ called the {\it complex slices} of the bisector. Goldman \cite{Go} showed that a bisector is
 the union of all  maximal totally real totally geodesic subspaces  containing the real spine. Such Lagrangian subspaces are called the {\it real slices} or {\it meridians} of the bisector.
 
\section{The moduli space of representations of $\mathbf{PSL}(2,\mathbb{Z})$  into $\mathbf{PU}(3,1)$} \label{sec:modulispace}


In this section we give the matrix representations of $\mathbf{PSL}(2,\mathbb{Z})$ into $\mathbf{PU}(3,1)$. The main interesting cases are  those when  the generator $\rho(a_2)=A_2$ is a $\pi$-rotation about a $\mathbb{C}$-line, and $\rho(a_1)=A_1$ is elliptic of $(0, \frac{2\pi}{3}, \frac{4 \pi}{3})$ type or $(\frac{2\pi}{3}, \frac{4\pi}{3}, \frac{4 \pi}{3})$ type.
In fact, we will show these representations are $\mathbb{R}$-decomposable. Our   matrix representations are inspired by the  matrix representations of $\mathbf{PSL}(2,\mathbb{Z}) \longrightarrow \mathbf{PU}(2,1)$ in \cite{Falbelparker:2003}.

\subsection{Order three elements in $\mathbf{PU}(3,1)$} \label{subsection:order3}

Let  $A$ be an elliptic element in $\mathbf{PU}(3,1)$, then $A$ acts with a fixed point in $\mathbf{H}^3_{\mathbb C}$. So if we use the ball model of  $\mathbf{H}^3_{\mathbb C}$, and assume  $$[0,0,0,1]^{\top} \in \mathbf{H}^3_{\mathbb C}$$ is a  fixed point of $A$, then $A$ acts on the tangent space of $(0,0,0,1)^{\top}$ as  an element of $\mathbf{U}(3)$ with eigenvalues $$\rm{e}^{2 \theta_1 {\rm i}},~~~ \rm{e}^{2 \theta_2 {\rm i}},~~~\rm{e}^{2 \theta_3 {\rm i}}.$$ Moreover we can write $A$ as the product of two Lagrangian inversions $\iota_1$ and $\iota_2$ along $\mathbb{R}^3$-subspaces $\mathcal{L}_1$ and $\mathcal{L}_2$ in $\mathbf{H}^3_{\mathbb C}$. The pair  $\mathcal{L}_1$ and $\mathcal{L}_2$ are said to have \emph{configuration $(\theta_1, \theta_2, \theta_3)$}, and $A$ is elliptic of \emph{$(2\theta_1, 2\theta_2, 2\theta_3)$ type}, see \cite{FalbelKoseleff:2002top,Go}.
For example, using the ball model of $\mathbf{H}^3_{\mathbb C}$, then the $\mathbb{R}^3$-subspaces $$\{[y_1,~~y_2,~~y_3,1]^{\top} \in \mathbf{H}^3_{\mathbb C}\ | \ y_j \in \mathbb{R}\}$$ and 
$$\{[{\rm e}^{ \theta_1 {\rm i}}y_1,~~{\rm e}^{ \theta_2 {\rm i}}y_2,~~{\rm e}^{ \theta_3 {\rm i} }y_3,1]^{\top} \in \mathbf{H}^3_{\mathbb C} \ |\ y_j \in \mathbb{R}\}$$ have  configuration $(\theta_1, \theta_2, \theta_3)$.

In this paper, we are interested in some special cases of elliptic elements:
\begin{itemize}
	\item if  $(\theta_1, \theta_2, \theta_3)= (0, \frac{\pi}{2},  \frac{\pi}{2})$, then  $\iota_1\iota_2$ is a $\pi$-rotation with fixed point set a $\mathbb{C}$-line in $\mathbf{H}^3_{\mathbb C}$;
	
	\item if  $(\theta_1, \theta_2, \theta_3)= (0, 0,  \frac{\pi}{2})$, then  $\iota_1\iota_2$ is a $\mathbb{C}$-reflection with fixed point  set a $\mathbb{C}$-plane in $\mathbf{H}^3_{\mathbb C}$;
	
	\item if  $(\theta_1, \theta_2, \theta_3)= (0, \frac{\pi}{3},  \frac{2\pi}{3})$, then  $\iota_1\iota_2$ has order three, the fixed point set of $\iota_1\iota_2$ is a $\mathbb{C}$-line. $\iota_1\iota_2$ also  stabilizes a $\mathbb{C}$-plane in $\mathbf{H}^3_{\mathbb C}$ which is perpendicular to  the invariant $\mathbb{C}$-line;
	
		\item if  $(\theta_1, \theta_2, \theta_3)= ( \frac{\pi}{3}, \frac{2\pi}{3},  \frac{2\pi}{3})$, then  $\iota_1\iota_2$ has order three, which stabilizes three different $\mathbb{C}$-planes in $\mathbf{H}^3_{\mathbb C}$.  Moreover, the fixed point set of $\iota_1\iota_2$ is only  a point in  $\mathbf{H}^3_{\mathbb C}$. 
	
\end{itemize}

Order three element $J$ which is neither  a $\mathbb{C}$-reflection nor a $\mathbb{C}$-rotation is important for this paper. For a given $J$, we should make it explicit  it is $(\frac{2\pi}{3}, \frac{2\pi}{3},\frac{4\pi}{3})$  or $(\frac{2\pi}{3}, \frac{4\pi}{3},\frac{2\pi}{3})$ type. Eigenvalue set does not distinguish the types, we need the detailed information  on eigenvectors.  So we write down the following well-known fact carefully. 

Using ball model, that is we use the Hermitian matrix  $H_{b}$
when defining the complex hyperbolic space  $\mathbf{H}^3_{\mathbb C}$.
Take 
$$E_{012}=\begin{pmatrix}
1& 0& 0&0\\
0&\rm{e}^{\frac{2\pi {\rm i}}{3}} & 0& 0\\
0&0&\rm{e}^{\frac{4\pi{\rm i}}{3}} &0\\
0 & 0&0&1\\
\end{pmatrix},$$
$$E_{112}=\rm{e}^{-\frac{2\pi{\rm i}}{3}} \cdot \begin{pmatrix}
\rm{e}^{\frac{2\pi {\rm i}}{3}}& 0& 0&0\\
0&\rm{e}^{\frac{2\pi {\rm i}}{3}} & 0& 0\\
0&0&\rm{e}^{\frac{4\pi {\rm i}}{3}} &0\\
0 & 0&0&1\\
\end{pmatrix}$$
and 
$$E_{122}=\rm{e}^{-\frac{5\pi {\rm i}}{6}} \cdot \begin{pmatrix}
\rm{e}^{\frac{2\pi {\rm i}}{3}}& 0& 0&0\\
0&\rm{e}^{\frac{4\pi{\rm i}}{3}} & 0& 0\\
0&0&\rm{e}^{\frac{4\pi{\rm i}}{3}} &0\\
0 & 0&0&1\\
\end{pmatrix}$$
be elliptic elements in $\mathbf{PU}(3,1)$ of $$(0, \frac{2\pi}{3},\frac{4\pi}{3}),~~(\frac{2\pi}{3}, \frac{2\pi}{3},\frac{4\pi}{3}) ~~~~ \text{and} ~~~~(\frac{2\pi}{3}, \frac{4\pi}{3},\frac{4\pi}{3})$$  type respectively. Note that  $E^2_{112}$ is $(\frac{2\pi}{3}, \frac{4\pi}{3},\frac{4\pi}{3})$  type, and $E^2_{122}$ is $(\frac{2\pi}{3}, \frac{2\pi}{3},\frac{4\pi}{3})$  type. 


The matrices $E_{012}$, $E_{112}$ and $- {\rm i} \cdot E_{122}$ all have the same set of eigenvalues, say $$\lambda_1=\frac{-1+\sqrt{3} {\rm i}}{2},~~~\lambda_2=\frac{-1-\sqrt{3} {\rm i}}{2},~~~\lambda_3=1,~~~\lambda_4=1.$$
Take $$A=(1,0,0,0)^{\top},$$
$$B=(0,1,0,0)^{\top},$$
$$C=(0,0,1,0)^{\top},$$
$$D=(0,0,0,1)^{\top}.$$
Then $D$ is a negative vector, and $A,B,C$ are positive vectors with respect to the Hermitian matrix  $H_{b}$. 

The matrices $E_{012}$, $E_{112}$ and $- {\rm i} \cdot E_{122}$ all have the same set of eigenvectors $A, B,C,D$ above. But the eigenvectors correspond to  different eigenvalues:
\begin{enumerate}
	
\item For the matrix $E_{012}$, the eigenvectors correspond to $\lambda_1,~~\lambda_2,~~\lambda_3,~~\lambda_4$ are $$B,~~C,~~D,~~A$$ respectively. One of the  two eigenvectors correspond to $\lambda_3=\lambda_4=1$ is positive, and the other one is negative. 
In particular,  the $\mathbb{C}$-line spanned by $D$ and $A$ is a totally geodesic  $\mathbf{H}^1_{\mathbb C} \hookrightarrow\mathbf{H}^3_{\mathbb C}$.  $E_{012}$ acts as  the identity map when restricting on this $\mathbb{C}$-line;

\item For the  matrix $E_{112}$, the eigenvectors correspond to $\lambda_1,~~\lambda_2,~~\lambda_3,~~\lambda_4$ are $$C,~~D,~~B,~~A$$ respectively.  Both of  the  two eigenvectors correspond to $\lambda_3=\lambda_4=1$ are positive. The eigenvector corresponds to $\lambda_2$ is negative;

\item For the  matrix   $- \rm{i} \cdot E_{122}$, the eigenvectors correspond to $\lambda_1,~~\lambda_2,~~\lambda_3,~~\lambda_4$ are $$D,~~A,~~C,~~B$$ respectively.  Both of  the  two eigenvectors correspond to $\lambda_3=\lambda_4=1$ are positive. The eigenvector corresponds to $\lambda_1$ is negative. 
\end{enumerate}

Note that the only difference between $E_{012}$ and $E_{122}$ is  the eigenvector corresponds $\lambda_1$  is positive or negative.

	\subsection{Proof of the first part of Theorem \ref{thm:modularfmoduli}} \label{subsection:firstpartmoduli}
	
In this subsection, we prove  the first part of Theorem \ref{thm:modularfmoduli}.
That is, we consider thirteen  components of  rigid $\mathbb{C}$-Fuchsian representations of  $\mathbf{PSL}(2,\mathbb{Z})$ into $ \mathbf{PU}(3,1)$.

It is easy to see:
\begin{itemize}
	\item
if 	$f \in \mathbf{PU}(3,1)$ is  $(\frac{2\pi}{3},\frac{2\pi}{3},\frac{4\pi}{3})$ type, then $f^2$ is  $(\frac{2\pi}{3},\frac{4\pi}{3},\frac{4\pi}{3})$ type;
	\item if  $f \in \mathbf{PU}(3,1)$ is  $(\frac{2\pi}{3},\frac{2\pi}{3},\frac{2\pi}{3})$ type, then $f^2$ is  $(\frac{4\pi}{3},\frac{4\pi}{3},\frac{4\pi}{3})$ type.	
\end{itemize} 
There are also some other similar  relations on  types of order three elements.

Let $\rho:  \mathbf{PSL}(2,\mathbb{Z}) \longrightarrow \mathbf{PU}(3,1)$ be a discrete, faithful, type-preserving representation such that $$\rho(a_2)= A_2, ~~\rho(a_1)= A_1.$$ So $A_2$ has order two, $A_1$ has order three and $A_2A_1$ is parabolic. 
Then $A_2A^2_1$ is conjugate to $(A_2A_1)^{-1}$, it  is also parabolic. Up to relations above, we may assume $A_2$ has one of the following types:
\begin{itemize}
	\item [R1.]  $(\pi,\pi,\pi)$ type;
		\item [R2.] $(0,\pi,\pi)$ type;
			\item [R3.] $(0,0,\pi)$ type.
\end{itemize}
And  $A_1$ has one of the following types:
\begin{itemize}
	\item [J1.]    $(\frac{2\pi}{3}, \frac{2\pi}{3},\frac{2\pi}{3})$ type;
	\item [J2.]  $(0, \frac{2\pi}{3},\frac{2\pi}{3})$ type;
		\item [J3.]  $(0, 0,\frac{2\pi}{3})$ type;
		\item [J4.]   $(0, \frac{2\pi}{3},\frac{4\pi}{3})$ type;
	\item  [J5.]   $(\frac{2\pi}{3}, \frac{4\pi}{3},\frac{4\pi}{3})$ type;
\end{itemize}
So in total the pair $\{A_2, A_1\}$ have fifteen types.

 Using the  ball model of  ${\bf H}^{3}_{\mathbb C}$, we can see
\begin{itemize}
	\item
  a complex reflection  at a point preserves every complex plane passing through the point;
  
  \item
  a complex reflection  at a point preserves every complex line passing through the point;
  \item A   complex reflection  along a complex plane preserves every complex line  orthogonal to this complex plane;
   \item A   complex reflection  along  a complex plane preserves every complex plane   orthogonal to this complex plane in a  complex line;
   \item   A   complex rotation  along  a complex line preserves every complex plane  orthogonal to this complex line; 
   \item   A   complex rotation  along a complex line preserves every complex line orthogonal to this complex line.
\end{itemize}

Now suppose  $A_2$ is  $(\pi, \pi,\pi)$ type, so $A_2$ is  a complex reflection  on a point  $p \in  {\bf H}^{3}_{\mathbb C}$.  Let  $\mathcal{P}$ be the totally geodesic complex subspace of ${\bf H}^{3}_{\mathbb C}$ spanned  by $p$,  $A_1(p)$ and  $A_1^{-1}(p)$.
Then  $\mathcal{P}$ is  a complex plane in  ${\bf H}^{3}_{\mathbb C}$, or in the degenerating case $\mathcal{P}$ is a complex line. $A_1$ also  preserves $\mathcal{P}$ invariant.  So the group $\langle A_2, A_1\rangle$ preserves $\mathcal{P}$ invariant. We  may assume $\langle A_2, A_1\rangle < \mathbf{PU}(2,1)$ and $A_2$ is a $\mathbb{C}$-reflection on the point $p \in \mathcal{P}$. 
So by Proposition 3.1 of \cite{Falbelparker:2003}, $\langle A_2, A_1\rangle$ in fact preserves  a complex line in $\mathcal{P}$ invariant. It is a rigid $\mathbb{C}$-Fuchsian representation.

If   $A_2$ is  $(0, 0,\pi)$ type, it  is a  complex reflection  along a complex plane $\mathcal{P}$. The polar of $\mathcal{P}$ is $n \in {\mathbb C}{\mathbf P}^{3} - \overline{\mathbf{H}^3_{\mathbb C}}$. If the three points $n$,  $A_1(n)$ and  $A_1^{-1}(n)$ span a  ${\mathbb C}{\mathbf P}^{2} \hookrightarrow  {\mathbb C}{\mathbf P}^{3}$.  Let  $\mathcal{Q}$ be the complex hyperbolic  plane which is the intersection of this complex projective plane  and $\mathbf{H}^3_{\mathbb C}$. $\mathcal{Q}$ is $A_1$-invariant.  Let $m$ be the polar of  $\mathcal{Q}$. Then the Hermitian product of $n$ and $m$ is zero. Which implies  $\mathcal{P}$ intersects with $ \mathcal{Q}$ orthogonally in a complex line. So $\mathcal{Q}$ is $A_2$-invariant, and $A_2$ acts on  $\mathcal{Q}$  as a   complex reflection  along the  complex line $\mathcal{P}\cap \mathcal{Q}$. Now  $\mathcal{Q}$ is $\langle A_2,A_1\rangle $ invariant. In other words,  we  may assume $\langle A_2, A_1 \rangle < \mathbf{PU}(2,1)$ and $A_2$ is a $\mathbb{C}$-reflection along a complex line  in $\mathcal{Q}$. 
So by Proposition 3.1 of \cite{Falbelparker:2003}, $\langle A_2, A_1 \rangle$ in fact preserves  a complex line in $\mathcal{Q}$ invariant. It is a rigid $\mathbb{C}$-Fuchsian representation.  If   the three points $n$,  $A_1(n)$ and  $A_1^{-1}(n)$  span a $\mathbf{C}\mathbb{P}^1$. Let  $\mathcal{Q}$ be the complex geodesic which is the intersection of this complex projective line  and $\mathbf{H}^3_{\mathbb C}$. $\mathcal{Q}$ is $A_1$-invariant. $\mathcal{P}$ intersects with $ \mathcal{Q}$ orthogonally at a point. So $\mathcal{Q}$ is $A_2$-invariant, and $A_2$ acts on  $\mathcal{Q}$  as a   complex reflection at the point $\mathcal{P}\cap \mathcal{Q}$. Now  $\mathcal{Q}$ is $\langle A_2,A_1\rangle $ invariant. In other words,  we  may assume $\langle A_2, A_1 \rangle < \mathbf{PU}(1,1)$. It is a rigid $\mathbb{C}$-Fuchsian representation.


 If  $A_2$ is  $(0, \pi,\pi)$ type and $A_1$ is   $(\frac{2\pi}{3}, \frac{2\pi}{3},\frac{2\pi}{3})$ type. So  $A_2$  is a  complex rotation  along  a complex line $\mathcal{L}$ and  $A_1$ is a  complex reflection  on a point $q$. Take the complex plane $\mathcal{P}$	which passes through $q$ and  orthogonal to $\mathcal{L}$.  $\langle A_2, A_1 \rangle$  preserves $\mathcal{P}$ invariant. Moreover, $A_2$ and $A_1$ are both complex reflections on  points in $\mathcal{P}$. So by Proposition 3.1 of \cite{Falbelparker:2003}, $\langle A_2, A_1 \rangle$ in fact preserves  a complex line in $\mathcal{P}$  invariant. It is a rigid $\mathbb{C}$-Fuchsian representation.

If  $A_2$ is  $(0, \pi,\pi)$ type and $A_1$ is  $(0, \frac{2\pi}{3},\frac{2\pi}{3})$ type. So  $A_2$  and $A_1$ are  complex rotations  along   complex lines $\mathcal{L}$ and  $\mathcal{L}'$  respectively. Then we may assume  $\mathcal{L}$ and  $\mathcal{L}'$ are not asymptotic. There is a unique  complex line $\mathcal{L}^{*}$ which is   orthogonal to $\mathcal{L}$ and  $\mathcal{L}'$. $\langle A_2, A_1 \rangle$ preserves $\mathcal{L}^{*}$ invariant. So it is a rigid $\mathbb{C}$-Fuchsian representation.

If  $A_2$ is  $(0, \pi,\pi)$ type and $A_1$ is   $(0, 0,\frac{2\pi}{3})$ type.  So  $A_2$ is a complex  rotation  along a complex line $\mathcal{L}$ and $A_1$ is a   complex reflection  along a complex plane  $\mathcal{P}$. Now $\langle A_1, A_2A_1A_2\rangle$ is an index two subgroup of $\langle A_2, A_1 \rangle$. Let $n$ be the polar of $\mathcal{P}$. The span of $n$ and $A_2(n)$ is a  ${\mathbb C}{\mathbf P}^{1}$. The intersection of this  complex projective line and $\mathbf{H}^3_{\mathbb C}$ is a $\mathbb{C}$-line. Which is invariant under $\langle A_1, A_2A_1A_2 \rangle$.  The 
  $\langle A_1, A_2A_1A_2 \rangle$-action on this  $\mathbb{C}$-line  with  $A_1 \cdot A_2A_1A_2$  parabolic is rigid. So $\langle A_2, A_1 \rangle$ is a  rigid $\mathbb{C}$-Fuchsian representation.

The two dimensional components we claimed in  Theorem \ref{thm:modularfmoduli}  when $A_2$ is $(0,0,\pi)$ type, and $A_1$ is  $(0, \frac{2\pi}{3},\frac{4\pi}{3})$ type
or  $(\frac{2\pi}{3}, \frac{4\pi}{3},\frac{4\pi}{3})$ type will be treated in later subsections. This ends the proof of the first part of  Theorem \ref{thm:modularfmoduli}.

\subsection{Preliminary on  infinite  $\mathbb{R}^3$-chains} \label{subsection:rform}
In this subsection, we show two properties of infinite  $\mathbb{R}^3$-chains in $\mathbb{C}^2 \times \mathbb{R}$, which are generalization of the well-known properties in $\mathbf{H}^2_{\mathbb C}$-geometry into $\mathbf{H}^3_{\mathbb C}$-geometry.

	Recall that the  \emph{presentation matrix}  of an anti-holomorphic  isometry $f$ of $\mathbf{H}^3_{\mathbb C}$ is a matrix  $M$ such that 
$$f((x,y,z,w)^{\top})=M \cdot (\overline{x}, \overline{y},\overline{z},\overline{w})^{\top}$$ 
for a lift  $(x,y,z,w)^{\top} \in \mathbb{C}^{3,1}$ of $p \in \mathbf{H}^3_{\mathbb C}$.
 The \emph{presentation matrix}  of a holomorphic  isometry $f$ of $\mathbf{H}^3_{\mathbb C}$ is just any matrix $M \in \mathbf{U}(3,1)$ such that 
$$f((x,y,z,w)^{\top})=M \cdot (x, y, z, w)^{\top}.$$  
Moreover for two  anti-holomorphic  isometries $f$ and $g$ with  presentation matrices $M$ and $N$, then $fg$ has presentation matrix $M \cdot \overline{N} \in \mathbf{U}(3,1)$, where $\overline{N}$ is the matrix obtained from $N$ by taking the conjugacy of each entry of $N$, see \cite{Will:2007}.

The following Proposition \ref{prop:infiniteRform}  should be compared with Page 36 of \cite{Parker:2010note}, where the general form of infinite $\mathbb{R}^2$-chains in $\mathbb{C} \times \mathbb{R}$ were given. 

\begin{prop}  \label{prop:infiniteRform}Let  $\mathcal{L}$ be an infinite  $\mathbb{R}^3$-chain passing through $$(x^0_1+y^0_1 \cdot {\rm i},~~x^0_2+y^0_2 \cdot {\rm i},~~ v^{0})\in \mathbb{C}^2 \times \mathbb{R}.$$ Then there are $\phi_1, \phi_2 \in [0, \pi)$ such that  $\mathcal{L}$ is given by
\begin{equation}\label{infiniteRformgeneral}
\left \{(x^0_1+y^0_1 \cdot {\rm i}+{\rm e}^{\phi_1 {\rm i}}\cdot r_1,~~x^0_2+y^0_2 \cdot {\rm i}+{\rm e}^{\phi_2 {\rm i}}\cdot r_2,~~v)~~|~~ r_1,r_2 \in \mathbb{R} \right \},
\end{equation}
where $$v=v^{0}+2[r_1y^0_1 \cos(\phi_1)-r_1x^{0}_1 \sin(\phi_1)+r_2y^0_2 \cos(\phi_2)-r_2x^{0}_2 \sin(\phi_2)].$$
\end{prop}
\begin{proof}First, $\mathcal{L}_{0}$ given by  $$\{(x_1,x_2,0)~~|~~ x_1,x_2 \in \mathbb{R}\}$$ is an infinite  $\mathbb{R}^3$-chain passing through $(0,0,0)\in \mathbb{C}^2 \times \mathbb{R}$. Let 
	$$A=\begin{pmatrix}
1& 0& 0&0\\
	0&a & b& 0\\
	0&-\lambda \overline{b}&\lambda \overline{a} &0\\
	0 & 0&0&1\\
	\end{pmatrix} $$
	be a Heisenberg rotation fixing $\infty$ and  $(0,0,0)\in \mathbb{C}^2 \times \mathbb{R}$, 
	with   $$|\lambda|=1,~~|a|^2+|b|^2=1.$$ 
	The  rotation $A$ translates  $\mathcal{L}_{0}$ into  $A(\mathcal{L}_0)$, which is given by $$\left \{(a\cdot x_1+b \cdot x_2,~~- \lambda \overline{b}\cdot x_1+\lambda \overline{a} \cdot x_2,~~0)~~|~~ x_1,x_2 \in \mathbb{R}\right\}.$$ 
The infinite  $\mathbb{R}^3$-chain	$A(\mathcal{L}_0)$ is  also given by
	 \begin{equation}\label{infiniteRformpaeezero}
	\left \{({\rm e}^{\phi_1 {\rm i}}\cdot r_1,~~{\rm e}^{\phi_2 {\rm i}}\cdot r_2,0)~~|~~ r_1,r_2 \in \mathbb{R} \right\}
	\end{equation}for $(\phi_1, \phi_2)$ depends on $a,b,\lambda$. This means that in the matrix $A$, we may take $$a={\rm e}^{\phi_1 {\rm i}},~~b=0,~~ \text{and}~~\lambda=\rm{e}^{\phi_1 {\rm i}+\phi_2 {\rm i}}.$$
	Another way to see this is $\mathbf{SU}(2)/ \mathbf{SO}(2)$ has real dimension two. 
	 In particular, for different pairs   $(\phi_1, \phi_2)$, (\ref{infiniteRformpaeezero}) gives the 2-dimensional parameterization of all  infinite  $\mathbb{R}^3$-chains passing through $(0,0,0)\in \mathbb{C}^2 \times \mathbb{R}$.

		Now applying the Heisenberg translation $N=N_{(x^0_1+y^0_1 \cdot {\rm i},~~x^0_2+y^0_2 \cdot {\rm i},~~ v^{0})}$, which is given by 
		$$\begin{pmatrix}
		1&  \sqrt{2} (-x^0_1+y^0_1 \cdot {\rm i})&  \sqrt{2} (-x^0_2+y^0_2 \cdot {\rm i})&-((x^0_1)^2 +(y^0_1)^2+(x^0_2)^2+(y^0_2)^2)+v^{0} {\rm i}\\
		0&1& 0& \sqrt{2} (x^0_1+y^0_1 \cdot {\rm i})\\
		0&0&1 & \sqrt{2} (x^0_2+y^0_2 \cdot {\rm i})\\
		0 & 0&0&1\\
		\end{pmatrix},$$
 we get 
  the general form in (\ref{infiniteRformgeneral}) of   infinite  $\mathbb{R}$-chains passing through  $(x^0_1+y^0_1 \cdot {\rm i},~~x^0_2+y^0_2 \cdot {\rm i}, ~~v^{0})$.
\end{proof}


Proposition  \ref{prop:asymptotic} bellow is an extension of $\frac{1}{6}$-part of  Proposition 3.1  of \cite{FalbelZocca:1999} into $\mathbf{H}^3_{\mathbb C}$-geometry. 

\begin{prop}  \label{prop:asymptotic}Let $\iota_1$ and $\iota_2$ be   Lagrangian inversions along  infinite  $\mathbb{R}^3$-chains  $\mathcal{L}_1$ and $\mathcal{L}_2$ in $\mathbb{C}^2 \times \mathbb{R}$. If $\mathcal{L}_1$ and $\mathcal{L}_2$ are disjoint, then $\iota_1\iota_2$ is parabolic. 
\end{prop}
\begin{proof}
	
Up to Heisenberg rotations and Heisenberg translations, we may assume $$\mathcal{L}_1=\{(x_1,x_2,0) \in \mathbb{C}^2 \times \mathbb{R}~~|~~ x_1,x_2 \in \mathbb{R}\}.$$
	 So the presentation matrix of $\iota_1$ is just the identity matrix  $I_4$.
	 
	 From  Proposition \ref{prop:infiniteRform} we may assume $\mathcal{L}_2$ is in the form of (\ref{infiniteRformgeneral}), and  $\mathcal{L}_2=NA(\mathcal{L}_1)$ for some Heisenberg rotation $A$ and Heisenberg translation $N$. Then  $$\iota_2= (N \circ A) \circ \iota_1  \circ (N \circ A)^{-1}$$
	as an anti-holomorphic isometry of $\mathbf{H}^3_{\mathbb C}$.
	 The presentation matrix $M$ of $\iota_2$ is $$M= (NA) \cdot I_4 \cdot  \overline{(NA)^{-1}}.$$
	 Where $\overline{(NA)^{-1}}$ is the (entices) conjugacy of $(NA)^{-1}$. 
	 
	 Then the presentation matrix of $\iota_2 \iota_1$  in $\mathbf{U}(3,1)$ is also $M$. Which is in the form of  $$\begin{pmatrix}
	 1& \sqrt{2}(x^0_1+y^0_1 {\rm i})-\sqrt{2}(x^0_1-y^0_1 {\rm i}) \rm{e}^{2\phi_1 {\rm i}}&  \sqrt{2}(x^0_2+y^0_2 {\rm i})-\sqrt{2}(x^0_2-y^0_2 \rm{i}) \rm{e}^{2\phi_2 {\rm i}} &m\\
	 0& \rm{e}^{2 \phi_1 {\rm i}} & 0& *\\
	 0&0& \rm{e}^{2 \phi_2 {\rm i}} &*\\
	 0 & 0&0&1\\
	 \end{pmatrix}.$$
	 	Where the imaginary part of $m$ is $$2v_0+2\sin(2\phi_1) ((x^0_1)^2-(y^0_1)^2)+2\sin(2\phi_2) ((x^0_2)^2-(y^0_2)^2)-4\cos(2\phi_1) x^0_1y^0_1-4\cos(2\phi_2)x^0_2y^0_2.$$
	 

There are three cases:	
	\begin{itemize}
		\item If $\phi_1, \phi_2 \in (0, \pi)$, then $\sin(\phi_1) \neq 0$, $\sin(\phi_2) \neq 0$. So $\mathcal{L}_1$ passes through a point $(z_1,z_2,t)$ in $\mathbb{C}^2 \times \mathbb{R}$ with $\Im(z_1)=\Im(z_2)=0$. Then we may assume $y^0_1=y^0_2=0$. The Heisenberg  translation  	$$N=\begin{pmatrix}
		1&  \sqrt{2} x^0_1&  \sqrt{2}x^0_2&-((x^0_1)^2 +(x^0_2)^2+v^{0} {\rm i}\\
		0&1& 0& -\sqrt{2} x^0_1\\
		0&0&1 & -\sqrt{2} x^0_2\\
		0 & 0&0&1\\
		\end{pmatrix}$$
		fixes $\mathcal{L}_1$. 
		In turn we may assume  $x^0_1=x^0_2=0$ 
		up to  the  Heisenberg translation  $N$. That is,  we may assume $\mathcal{L}_2$ passes through $$(0,0,0) \in \mathbb{C}^2 \times \mathbb{R}.$$ So $v^0 \neq 0$ as we assume $\mathcal{L}_1 \cap \mathcal{L}_2=\emptyset$. Then from the presentation matrix  $M$, $\iota_2 \iota_1$   is parabolic.

		\item If  $\phi_1\in (0, \pi)$ and $\phi_2 =0$. As above we may assume $x^0_1=y^0_1=0$. If  $y^0_2= 0$, we may assume  $x^0_2=0$ up to a  Heisenberg translation which fixes $\mathcal{L}_1$. So $v^0 \neq 0$, and $\iota_2 \iota_1$  is parabolic.
			If	$y^0_2 \neq 0$, then the presentation matrix $M$ of $\iota_2 \iota_1$  is in the form of 
			$$\begin{pmatrix}
			1& 0&  2\sqrt{2}y^0_2 {\rm i}  &m\\
			0& \rm{e}^{2 \phi_1 {\rm i}} & 0& 0\\
			0&0& 1 &2\sqrt{2}y^0_2 {\rm i}\\
			0 & 0&0&1\\
			\end{pmatrix}.$$
			So	 $\iota_2 \iota_1$  is the product of an elliptic element and a nontrivial unipotent element, it is 	ellipto-parabolic.

		\item If $\phi_1=\phi_2 =0$. 	Up to  a  Heisenberg translation which fixes $\mathcal{L}_1$, we may assume $x^0_1=x^0_2=0$. Then at least one of 	$y^0_1$ and 	$y^0_2$ is non-zero.  Then from the presentation matrix  $M$, $\iota_2 \iota_1$   is unipotent. 
		
	\end{itemize}

So in any cases, $\iota_1\iota_2$ is parabolic. 
\end{proof}

	\subsection{The moduli space $\mathcal{M}(0,\frac{2\pi}{3},\frac{4\pi}{3})$} \label{subsection:012}
	We now give  the matrix representations of $\mathbf{PSL}(2,\mathbb{Z})$ into $\mathbf{PU}(3,1)$, such that the generator $\rho(a_2)=A_2$ is a $\pi$-rotation about a $\mathbb{C}$-line, and $\rho(a_1)=A_1$ is elliptic of $(0, \frac{2\pi}{3}, \frac{4 \pi}{3})$ type.	In fact, we may take  $\iota_{i}$ be the  $\mathbb{R}$-reflection in $\mathbf{H}^3_{\mathbb C}$ with invariant $\mathbb{R}^3$-chain $\mathcal{R}_{i}$ for $i=0,1,2$, such that $$A_2=\iota_1\iota_{0},~~A_1=\iota_0\iota_{2},~~~\text{and}  ~~A_2A_1=\iota_1\iota_{2}~~ \text{is~~parabolic}.$$

	Recall the Hermitian matrix  $H_{s}$ in the definition of $\mathbf{H}^3_{\mathbb C}$.
	A point $(z_1,z_2,t)$ in the Heisenberg group $\mathfrak{N}=\mathbb{C}^2 \times \mathbb{R}$ and the point $\infty$ are lifts to 
	 $$(-(|z_1|^2+|z_2|^2)+{\rm i} t,~~\sqrt{2} z_1,~~\sqrt{2} z_2,~~1)^{\top}$$ and 
	  $$(1,0,0,0)^{\top}$$ in $\mathbb{C}^{3,1}$.
	  
	  	Conjugating if necessary, we may assume the fixed point of $A_2A_1=\iota_1\iota_{2}$ is $\infty$, and   $\mathcal{R}_{0}$ is a preferred  finite $\mathbb{R}^3$-chain.  As a  Lagrangian inversion of $\mathbb{C}^{3,1}$,  we take  $$\iota_{0}:(x,y,z,w)^{\top}\longrightarrow (\overline{w},\overline{y},\overline{z},\overline{x})^{\top}.$$
	 So as a map on the  Heisenberg group, we have 
	  \begin{flalign}
	 \nonumber &  \iota_{0}(x_1+y_1 {\rm i},~~x_2+y_2 {\rm i},~~t)=& \nonumber \\
	 &\left(\frac{-x_1+y_1 {\rm i}}{(x_1^2+y_1^2+x_2^2+y_2^2)+{\rm i}t},~~\frac{-x_2+y_2 {\rm i}}{(x_1^2+y_1^2+x_2^2+y_2^2)+{\rm i}t},~~\frac{t}{(x_1^2+y_1^2+x_2^2+y_2^2)^2+t^2}\right) .&\nonumber
	 \end{flalign}
	  The fixed point set of $\iota_{0}$ in $\mathfrak{N}$ is called  {\it the standard imaginary $\mathbb{R}^3$-chain}. 
	 	 The equations for  the standard imaginary $\mathbb{R}^3$-chain $\mathcal{R}_{0}$ in  $\mathfrak{N}$ are
	 	 \begin{flalign} \label{equation:standardimaginaryRchain}
	 	 	   \left\{ \begin{aligned}
	 	 (x_1^2+y_1^2+x_2^2+y_2^2)^2+t^2=1; \qquad& \  \\ 
	 	 x_1=-(x_1^2+y_1^2+x_2^2+y_2^2)x_1+ty_1; \qquad& \  \\
	 	 x_2=-(x_1^2+y_1^2+x_2^2+y_2^2)x_2+ty_2; \qquad& \  \\
	 	 y_1=(x_1^2+y_1^2+x_2^2+y_2^2)y_1+tx_1; \qquad& \
	 	 \\	y_2=(x_1^2+y_1^2+x_2^2+y_2^2)y_2+tx_2. \qquad& \  
	 	 \end{aligned}
	 	 \right.
	 	 	  \end{flalign}
	 	 	  We note that the fourth and the fifth equations in (\ref{equation:standardimaginaryRchain}) are consequents of the first three, and 	$\mathcal{R}_{0}$  is a 2-sphere.  	


	We also normalize such that  the $\mathbb{R}^3$-chain $\mathcal{R}_{1}$ is 
	 $$\left\{ (y_1 {\rm i}, x_2,0)~~|~~ y_1, x_2 \in \mathbb{R} \right\}$$ in the Heisenberg group. Let $\iota_1$ be the  Lagrangian inversion along  $\mathcal{R}_{1}$. So  $\mathcal{R}_{0}$  and  $\mathcal{R}_{1}$ have two common points $$(\pm \rm{i}, 0,0),$$
	 and $\iota_{0} \iota_1$ is $(0, \frac{\pi}{2},\frac{\pi}{2})$ type elliptic, that is,  $\iota_{0} \iota_1$ is a $\pi$-rotation along a $\mathbb{C}$-line.

	We note that   the presentation matrix of $\iota_0$ is 
	 	 \begin{equation}\label{matrix:M0}
	 	 M_0=\begin{pmatrix}
	 0& 0& 0&1\\
	 0&1 & 0& 0\\
	 0&0&1 &0\\
	 1 & 0&0&0\\
	 \end{pmatrix},
	 	 \end{equation}
	 and  the presentation matrix of $\iota_1$ is 
	 \begin{equation}\label{matrix:M1}
	 M_1=\begin{pmatrix}
	 1& 0& 0&0\\
	 0&-1 & 0& 0\\
	 0&0&1 &0\\
	 0 & 0&0&1\\
	 \end{pmatrix}.
	  \end{equation}

	 Let $p_1= \infty$, $p_2= \iota_0\iota_2(\infty)$ and  $p_3= (\iota_0\iota_2)^2(\infty)$  in the Heisenberg group. 
	 Then 
	$p_2 =(0,0,0)$, and 
	 $$p_3=\iota_0\iota_2 (0,0,0)=\iota_2\iota_0 (\infty)=\iota_2 (0,0,0).$$
	 So $p_3$ is fixed by $\iota_{0}$ and its lies in the standard imaginary $\mathbb{R}^3$-chain $\mathcal{R}_{0}$.
	  Let $\alpha \in [0, \frac{\pi}{2}]$ be the Cartan's angular invariant of  $p_1$, $p_2$ and  $p_3$, that is $$\alpha=\arg (-\langle \hat{p}_1,\hat{p}_2\rangle  \langle \hat{p}_2,\hat{p}_3\rangle \langle \hat{p}_3,\hat{p}_1\rangle )$$
	 for any lifts $\hat{p}_i$ of $p_{i}$ into $\mathbb{C}^{3,1}$. 
	 Then we have $$p_3=p_3(\alpha,\beta)=(\rm{i} \rm{e}^{-\frac{\alpha\rm{i}}{2}} \sqrt{\cos(\alpha)}\cos(\beta),~~\rm{i} \rm{e}^{-\frac{\alpha\rm{i}}{2}} \sqrt{\cos(\alpha)}\sin(\beta),~~\sin(\alpha))$$ for some $\beta$.
	 Up to symmetry, we may assume $\beta\in [0, \frac{\pi}{2}]$.
		When $\beta=0$, this is exactly the coordinates in \cite{Falbelparker:2003}.

	We now consider  Lagrangian inversion $\iota_{2}=\iota_{2}(\alpha,\beta)$ along the $\mathbb{R}^3$-chain $\mathcal{R}_{2}(\alpha,\beta)$, both of them   depend on $\alpha$ and $\beta$.  When $\beta=0$, we extend the matrix in   \cite{Falbelparker:2003} into the $\mathbf{H}^3_{\mathbb C}$-geometry. So  the presentation matrix of $\iota_2=\iota_{2}(\alpha,0)$ is 
	 	 \begin{equation}\label{matrix:M2}
	 	 M_2=M_2(\alpha,0) =\begin{pmatrix}
	 1& {\rm i} \rm{e}^{-\frac{\alpha\rm{i}}{2}} \sqrt{2\cos(\alpha)}& 0&-\rm{e}^{-\alpha\rm{i}}\\
	 0&\rm{e}^{-\alpha\rm{i}} & 0& {\rm i} \rm{e}^{-\frac{\alpha\rm{i}}{2}} \sqrt{2\cos(\alpha)}\\
	 0&0&-\rm{e}^{-\frac{\alpha\rm{i}}{3}}  &0\\
	 0 & 0&0&1\\
	 \end{pmatrix}.
	 	 \end{equation}
	 We note that the $(3,3)$-entry of the matrix $M_2$ is $-\rm{e}^{-\frac{\alpha\rm{i}}{3}}$ but not $1$, this is because the matrix in  \cite{Falbelparker:2003} does not has determinant 1. 
	  So when $\beta=0$, the $\mathbb{R}^3$-chain $\mathcal{R}_{2}=\mathcal{R}_{2}(\alpha,0)$ is given by 
	 $$\left\{ (x_1+y_1 {\rm i}, x_2+y_2 {\rm i},t):x_1, y_1, x_2,y_2,t \in \mathbb{R} \right\}$$ with  equations 
\begin{equation}	 \label{equation:R2alphazero}
	 \left\{ \begin{aligned}
	 0=2x_1\sin(\frac{\alpha}{2})+  2y_1\cos(\frac{\alpha}{2})-\sqrt{\cos(\alpha)};  \qquad &   \\ 
	 0=t-x_1 \cos(\frac{\alpha}{2})\sqrt{\cos(\alpha)}+  y_1 \sin(\frac{\alpha}{2}) \sqrt{\cos(\alpha)} -\frac{\sin(\alpha)}{2};
	  \qquad &   \\
	 -{\rm e}^{-\frac{\alpha {\rm i}}{3}}(x_2-y_2 {\rm i})=x_2+y_2 {\rm i}. \qquad&  
	 \end{aligned}
	 \right. \end{equation}

When $\beta=0$, the group  $\langle \iota_0,  \iota_1,  \iota_2 \rangle < \widehat{\mathbf{PU}(3,1)}$, but  it is essentially a subgroup of $\widehat{\mathbf{PU}(2,1)}$ since it preserves a $\mathbf{H}^2_{\mathbb C} \hookrightarrow\mathbf{H}^3_{\mathbb C}$ invariant.

For general $\beta$, we	 rotate the  $\mathbb{R}^3$-chain $\mathcal{R}_{2}(\alpha,0)$ to another  $\mathbb{R}^3$-chain $\mathcal{R}_{2}(\alpha, \beta)$. So we take 
	  \begin{equation}\label{matrix:U}
	 U=\left(\begin{matrix}
	 1 & 0& 0&0 \\ 0 & \cos(\beta) & -\sin(\beta) &0\\ 0 & \sin(\beta) & \cos(\beta) &0
	 \\ 0 & 0  & 0 &1 \end{matrix}\right).
	 \end{equation}
	 in $\mathbf{PU}(3,1)$. Note that  $$U(p_3(\alpha,0))=p_3(\alpha,\beta).$$
	 Then we take $$\iota_{2}(\alpha,\beta)=U \cdot \iota_{2} (\alpha,0) \cdot \overline{U^{-1}}$$  be the Lagrangian inversion
	 on the $\mathbb{R}^3$-chain  $$\mathcal{R}_{2}(\alpha,\beta)=U(\mathcal{R}_{2}(\alpha,0)).$$ Note that  $U$ has real entries, so  $U^{-1}= \overline{U^{-1}}$.
	 
	 For the convenience of later, we also write down precisely the presentation matrix of $\iota_{2}(\alpha,\beta)$, it is $M_2(\alpha,\beta)=UM_2(\alpha,0) U^{-1}=$
	 $$\begin{pmatrix}
	 1&  \rm{i} \rm{e}^{\frac{-\alpha {\rm i}}{2}} \sqrt{2\cos(\alpha)}\cos(\beta)&  \rm{i} \rm{e}^{\frac{-\alpha\rm{i}}{2}} \sqrt{2\cos(\alpha)}\sin(\beta)&-\rm{e}^{-\alpha\rm{i}}\\
	 0&\cos^2(\beta)\rm{e}^{-\alpha\rm{i}}-\sin^2(\beta)\rm{e}^{\frac{-\alpha\rm{i}}{3}} & \sin(2\beta)\cos(\frac{\alpha}{3})\rm{e}^{\frac{-2\alpha\rm{i}}{3}}& \rm{i} \rm{e}^{\frac{-\alpha\rm{i}}{2}} \sqrt{2\cos(\alpha)}\cos(\beta)\\
	 0&\sin(2\beta)\cos(\frac{\alpha}{3})\rm{e}^{-\frac{2\alpha\rm{i}}{3}}&\sin^2(\beta)\rm{e}^{-\alpha\rm{i}}-\cos^2(\beta)\rm{e}^{\frac{-\alpha\rm{i}}{3}}&\rm{i}\rm{e}^{\frac{-\alpha\rm{i}}{2}} \sqrt{2\cos(\alpha)}\sin(\beta)\\
	 0 & 0 &0&1\\
	 \end{pmatrix}.$$
	 
	 It is easy to see $U M_0 U^{-1}=M_0$,  and  $$U M_1 U^{-1}=\begin{pmatrix}
	 1& 0& 0&0\\
	 0&-\cos(2 \beta) & -\sin(2 \beta)& 0\\
	 0&-\sin(2 \beta)&\cos(2 \beta) &0\\
	 0 & 0&0&1\\
	 \end{pmatrix}.$$
	Which does not equal to $M_1$ except when  $\beta=0$. 
	 So, from the original group $$\langle \iota_0,  \iota_1,  \iota_2(\alpha, 0) \rangle$$ depends on $(\alpha,0)$,  we have the new group $$\langle \iota_0,  \iota_1,  U\iota_2(\alpha,0) U^{-1}\rangle=\langle \iota_0,  \iota_1,  \iota_2(\alpha,\beta) \rangle$$ depends on  $(\alpha,\beta)$, which is not conjugate to the  original one in $\widehat{\mathbf{PU}(3,1)}$ when $\beta\neq 0$.

	Now we have the presentation  matrix $A_0$ of $\iota_1\iota_2$, which  is $\rm{e}^{\frac{-\alpha\rm{i}}{3}}$ times 
 $$\begin{pmatrix}
1&  -\rm{i} \rm{e}^{\frac{\alpha\rm{i}}{2}} \sqrt{2\cos(\alpha)}\cos(\beta)&  -\rm{i} \rm{e}^{\frac{\alpha\rm{i}}{2}} \sqrt{2\cos(\alpha)}\sin(\beta)&-\rm{e}^{\alpha\rm{i}}\\
0&\sin^2(\beta)\rm{e}^{\frac{\alpha\rm{i}}{3}}-\cos^2(\beta)\rm{e}^{\alpha\rm{i}} & -\sin(2\beta)\cos(\frac{\alpha}{3})\rm{e}^{\frac{2\alpha\rm{i}}{3}}& \rm{i} \rm{e}^{\frac{\alpha\rm{i}}{2}} \sqrt{2\cos(\alpha)}\cos(\beta)\\
0&\sin(2\beta)\cos(\frac{\alpha}{3})\rm{e}^{\frac{2\alpha\rm{i}}{3}}&-\cos^2(\beta)\rm{e}^{\frac{\alpha\rm{i}}{3}}+\sin^2(\beta)\rm{e}^{\alpha\rm{i}}&-\rm{i}\rm{e}^{\frac{\alpha\rm{i}}{2}} \sqrt{2\cos(\alpha)}\sin(\beta)\\
0 & 0 &0&1\\
\end{pmatrix}.$$	
  The presentation matrix $A_1$ of $\iota_0\iota_2$, which is $\rm{e}^{\frac{-\alpha\rm{i}}{3}}$ times  $$ \begin{pmatrix}
  0& 0& 0&1\\
  0&-\sin^2(\beta)\rm{e}^{\frac{\alpha\rm{i}}{3}}+\cos^2(\beta)\rm{e}^{\alpha\rm{i}} & \sin(2\beta)\cos(\frac{\alpha}{3})\rm{e}^{\frac{2\alpha\rm{i}}{3}}& -\rm{i} \rm{e}^{\frac{\alpha\rm{i}}{2}} \sqrt{2\cos(\alpha)}\cos(\beta)\\
  0&\sin(2\beta)\cos(\frac{\alpha}{3})\rm{e}^{\frac{2\alpha\rm{i}}{3}}&-\cos^2(\beta)\rm{e}^{\frac{\alpha\rm{i}}{3}}+\sin^2(\beta)\rm{e}^{\alpha\rm{i}}&-\rm{i}\rm{e}^{\frac{\alpha\rm{i}}{2}} \sqrt{2\cos(\alpha)}\sin(\beta)\\
  1 & -\rm{i}\rm{e}^{\frac{\alpha\rm{i}}{2}} \sqrt{2\cos(\alpha)}\cos(\beta) &-\rm{i}\rm{e}^{\frac{\alpha\rm{i}}{2}} \sqrt{2\cos(\alpha)}\sin(\beta)&-\rm{e}^{\alpha\rm{i}}\\
  \end{pmatrix}.$$
 The presentation matrix of $\iota_1\iota_0$ is   $$A_2=\begin{pmatrix}
  0& 0& 0&1\\
  0&-1 & 0& 0\\
  0&0&1  &0\\
 1 & 0&0&0\\
  \end{pmatrix}.$$

  \begin{rem} In fact, together with $M_0$ and $M_1$,  any one of the matrices  $M_2(\alpha, \beta)$, $A_0$ or $A_1$  is enough to get the other two.
  But it tends to have some typo when writing down  these matrices. So the author wrote down all the matrices of $M_2(\alpha, \beta)$, $A_0$ and $A_1$ explicitly here as a double check.
  	 \end{rem}

  Then it can be showed directly that $$A^2_2=I_4,~~A^3_1=-I_{4},~~A_2A_0A^{-1}_1=I_4$$ and  $$A^{*}_{i}H_sA_{i}=H_s,~~\det(A_i)=1$$ for $i=0,1,2$. Moreover,  $A_0$ is parabolic with fixed point $\infty$ when $\alpha \neq 0$ by  Proposition \ref{prop:A0para} later. So we have a type-preserving representation of  $\mathbf{PSL}(2,\mathbb{Z})$ into $\mathbf{PU}(3,1)$ for $\alpha,\beta \in [0, \frac{\pi}{2}]$ with $\rho(a_i)=A_i$ for $i=0,1,2$. By Proposition \ref{prop:A1012} later, $A_1$ is $(0,\frac{2\pi}{3},\frac{4\pi}{3})$ type elliptic. So we denote  this  moduli space  by $\mathcal{M}(0,\frac{2\pi}{3},\frac{4\pi}{3})$, which  is  parameterized by the square $[0, \frac{\pi}{2}]^2$.

  When $\alpha=\beta=0$, the $\mathbb{R}^3$-chains $\mathcal{R}_{0}$ and  $\mathcal{R}_{2}$ have two common points $$(\frac{\rm{i}}{2}, \pm \frac{\sqrt{3} \rm{i}}{2},0). $$ Then $\iota_{0} \iota_{2}$ is a $(0, \frac{2\pi}{3}, \frac{4\pi}{3})$ type elliptic element in this case. In general, it is desirable to show that $\mathcal{R}_{1}$ and  $\mathcal{R}_{2}$ have exactly  two common points as a step to show $\iota_{0} \iota_{2}$ is $(0, \frac{2\pi}{3}, \frac{4\pi}{3})$ type. But we avoid this and show $\iota_{0} \iota_{2}$ is $(0, \frac{2\pi}{3}, \frac{4\pi}{3})$ type directly in  Proposition \ref{prop:A1012}.

  

 In fact,  the element $A_0$ is parabolic in general, but it is  elliptic in some special cases which are not interesting (to the author). See Item (4a) of Theorem \ref{thm:modularmoduli2dim}.
  \begin{prop}\label{prop:A0para} For any $\beta \in [0, \frac{\pi}{2}]$:
  	
  	\begin{itemize} 
  			\item When $\alpha \in (0, \frac{\pi}{2}]$,  $\mathcal{R}_{1} \cap \mathcal{R}_{2} = \emptyset$, so $A_0$ is parabolic;
  		
  			\item When $\alpha=0$ and  $\beta= \frac{\pi}{2}$,  then   $\mathcal{R}_{1} \cap \mathcal{R}_{2} = \emptyset$, so $A_0$ is parabolic;
  		\item When $\alpha=0$ and   $ \beta \in [0, \frac{\pi}{2})$,   $\mathcal{R}_{1} \cap \mathcal{R}_{2} \neq  \emptyset$, so $A_0$ is elliptic. 
  	\end{itemize} 
  	
  \end{prop}
  
  \begin{proof}Since $\mathcal{R}_{2}(\alpha,\beta)=U(\mathcal{R}_{2}(\alpha,0))$,  we only need to  study  the disjointness of  $U^{-1}(\mathcal{R}_{1})$ and $\mathcal{R}_{2}(\alpha,0)$.  The $\mathbb{R}^3$-chain $U^{-1}(\mathcal{R}_{1})$ is given by
  	\begin{equation} \label{equation:UR1}
  	\left\{ (\cos(\beta) r_1 {\rm i}+\sin(\beta) r_2, ~~-\sin(\beta) r_1  {\rm i}+\cos(\beta) r_2,0)~~|~~ r_1, r_2 \in \mathbb{R} \right\}
  		\end{equation}
  		 in the Heisenberg group. 
  	
  	First consider the case $\alpha=0$. If $$p=(x_1+y_1 {\rm i}, x_2+y_2 {\rm i},t) \in U^{-1}(\mathcal{R}_{1}) \cap \mathcal{R}_{2}(\alpha,0).$$
  	Then we have $t=0$ from (\ref{equation:UR1}). By  (\ref{equation:R2alphazero}), we have $x_1=x_2=0$ and $y_1=\frac{1}{2}$. So $r_2=0$ from  (\ref{equation:UR1}). We have $$p=(\frac{\rm{i}}{2}, y_2  {\rm i},0)=(\cos(\beta)r_1  {\rm i},~~-\sin(\beta)r_1  {\rm i},~~0).$$
  	So the intersection of two infinite $\mathbb{R}^3$-chains $U^{-1}(\mathcal{R}_{1}) \cap \mathcal{R}_{2}(\alpha,0)$ is a point if $\cos(\beta)\neq 0$, and then $A_0=\iota_1\iota_2$ is elliptic. 
  	 Moreover,   $U^{-1}(\mathcal{R}_{1}) \cap \mathcal{R}_{2}(\alpha,0)=\emptyset$ if  $\cos(\beta)= 0$, then by Proposition \ref{prop:asymptotic},  $A_0=\iota_1\iota_2$  is parabolic.

  	When  $\alpha \in (0, \frac{\pi}{2}]$. We assume \begin{equation} \nonumber
  	p=(a_1 {\rm e}^{\phi_1 \rm{i}},  ~~a_2   \rm{e}^{\phi_2\rm{i}},~~0) \in U^{-1}(\mathcal{R}_{1}) \cap \mathcal{R}_{2}(\alpha,0)
  	\end{equation}
  	with $a_1, a_2  \in \mathbb{R}$. Then we may assume  $\phi_1=\phi_2+\frac{\pi}{2}$ since $U^{-1}(\mathcal{R}_{1})$ is just a rotation of $\mathcal{R}_{1}$.  By the third equation in (\ref{equation:R2alphazero}), we have $$\phi_2= \pm\frac{\pi}{2} -\frac{\alpha}{6},$$ or $a_2=0$ and $\phi_2$ is not defined.

  	In the first case, we have  $$\phi_1= -\frac{\alpha}{6} ~~\text{or} ~~~-\pi -\frac{\alpha}{6}.$$
  	We now consider the first two equations in  (\ref{equation:R2alphazero}), they are now
  	\begin{equation}	 \label{equation:R2alphazerowithphi1}
  	\left\{ \begin{aligned}
  	0=2a_1\sin(\frac{\alpha}{2}+\phi_1)-\sqrt{\cos(\alpha)};  \qquad &   \\ 
  	0=-\sqrt{\cos(\alpha)}a_1 \cos(\frac{\alpha}{2}+\phi_1) -\frac{\sin(\alpha)}{2}.
  	\qquad &   \\ 
  	\end{aligned} 
  	\right.\end{equation}
  	Since $\sin(\alpha) >0$ now,  from the second equation in 	(\ref{equation:R2alphazerowithphi1}),  we have $$\sqrt{\cos(\alpha)}a_1 \cos(\frac{\alpha}{2}+\phi_1) <0,$$ then $\alpha \in (0, \frac{\pi}{2})$ and  $a_1 \cos(\frac{\alpha}{2}+\phi_1) <0$.
  	\begin{itemize}
  		\item If $\phi_1= -\frac{\alpha}{6}$,   we have $a_1<0$. Which contradicts to the first equation in   	(\ref{equation:R2alphazerowithphi1}) since  $\sin(\frac{\alpha}{2}+\phi_1) >0$;
  			\item If $\phi_1=- \pi -\frac{\alpha}{6}$,  we have $a_1>0$. Which also contradicts to the first equation in   
  		(\ref{equation:R2alphazerowithphi1}) since now $\sin(\frac{\alpha}{2}+\phi_1) <0$.
  			\end{itemize}
  	
  	If $a_2=0$, then $$p=(\cos(\beta)r_1  {\rm i}+sin(\beta) r_2,~~ -\sin(\beta)r_1  {\rm i}+\cos(\beta)r_2,~~0)=(x_1+y_1  {\rm i},0,0).$$
  	So $$-\sin(\beta)r_1=\cos(\beta)r_2=0.$$ At least one of $r_1$ and $r_2$ is zero.
  	\begin{itemize}
  		\item If $r_1=r_2=0$, then $p=(0,0,0)$. Which  does not lie in  $\mathcal{R}_{2}(\alpha,0)$;
  		
  		\item If $r_1=0$ but $r_2 \neq 0$, then $\cos(\beta)=0$ and $p=(r_2,0,0)$. Which does not lie in  $\mathcal{R}_{2}(\alpha,0)$ since $\alpha \in (0, \frac{\pi}{2}]$. 
  		
  		\item If $r_2=0$ but $r_1 \neq 0$, then $\sin(\beta)=0$ and $p=(r_1 \rm{i},0,0)$. Which does not lie in  $\mathcal{R}_{2}(\alpha,0)$ since $\alpha \in (0, \frac{\pi}{2}]$. 
  		
  		\end{itemize}
  	
  	 So when  $\alpha \in (0, \frac{\pi}{2}]$,  in either cases we have   $U^{-1}(\mathcal{R}_{1}) \cap \mathcal{R}_{2}(\alpha,0) = \emptyset$, and $A_0=\iota_1\iota_2$ is parabolic by Proposition \ref{prop:asymptotic}.
  \end{proof}

  \begin{remark}The element $A_0$ has an eigenvalue of 
  	multiplicity two (or 	multiplicity three in a curve with equation  (\ref{equation:A0twoeigenvalue}) in Proposition \ref{prop:A0elliptopara}). But the corresponding eigenvectors are positive and null when $\alpha=0$ but $\beta \neq \frac{\pi}{2}$, and  the corresponding eigenvectors are null and negative when $\alpha \in (0, \frac{\pi}{2}]$. This also interprets Proposition \ref{prop:A0para}.
  \end{remark}


   \begin{prop}\label{prop:A1012} For $\alpha,\beta \in [0, \frac{\pi}{2}]$,  $A_1$ is a $(0, \frac{2\pi}{3},\frac{4\pi}{3})$-type elliptic element. 
  \end{prop}
  
  \begin{proof}When $\alpha=\frac{\pi}{2}$ and $\beta=0$, the matrix $-A_1$ has eigenvalues $\lambda_i$ for $i=1,2,3,4$ as in Subsection \ref{subsection:order3}. The corresponding eigenvectors are 
  	$$(\frac{\sqrt{3}+\rm{i}}{2},0,0,1)^{\top},$$
  	$$(0,1,0,0)^{\top},$$
  	$$(\frac{-\sqrt{3}+\rm{i}}{2},0,0,1)^{\top},$$
  	$$(0,0,1,0)^{\top}.$$
  	They are positive, positive, negative and  positive vectors respectively. So $A_1$ is a $(0, \frac{2\pi}{3},\frac{4\pi}{3})$-type elliptic element.  We note that here we use $H_{s}$, but not $H_{b}$, in the definition of ${\bf H}^3_{\mathbb C}$, so the eigenvectors are different from $A,B,C,D$ in Subsection \ref{subsection:order3}. 
  	
  	The entries of $A_1$ are continuous on $\alpha$ and $\beta$. So the eigenvalues and eigenvectors vary  continuously on $\alpha, \beta$.
  	Since $A_1$ has order three, then the eigenvalues are constants independent of $\alpha, \beta$. The positivities/negativities of the corresponding eigenvectors do not change when varying  $\alpha$ and $\beta$. So by the argument in Subsection \ref{subsection:order3},  $A_1$ is always a $(0, \frac{2\pi}{3},\frac{4\pi}{3})$-type elliptic element. 
  \end{proof}

\subsection{Proof of Theorem \ref{thm:modularmoduli2dim} for the moduli space $\mathcal{M}(0,\frac{2\pi}{3},\frac{4\pi}{3})$} \label{subsec:proofmoduli012}

For each pair  $(\alpha, \beta)\in [0, \frac{\pi}{2}]^2$, we have a representation $\rho=\rho(\alpha, \beta)$ of   $\mathbf{PSL}(2,\mathbb{Z})$ into $\mathbf{PU}(3,1)$ in the moduli space $\mathcal{M}(0,\frac{2\pi}{3},\frac{4\pi}{3})$. In this subsection, we consider some  special representations in this moduli space.

 
  (1)  When $\beta=0$, each $A_i$ is in the form of $$\begin{pmatrix}
	*& *& 0&*\\
	*&* & 0& *\\
	0&0&*  &0\\
	* & *&0&*\\
	\end{pmatrix}.$$
	So $\langle A_0, A_1,A_2\rangle $ preserves a totally geodesic ${\bf H}^2_{\mathbb C} \hookrightarrow {\bf H}^3_{\mathbb C}$ invariant.
	We will view the representation as degenerating to a representation into  $\mathbf{PU}(2,1)$. In fact, in this case, when  deleting the third row and the third column of $A_i$ for $i=0,1,2$, we get matrices in $\mathbf{PU}(2,1)$. Which are exactly the matrices in Pages 66 and 67 of  \cite{Falbelparker:2003}. Now $A_2$ is a $\mathbb{C}$-reflection about a $\mathbb{C}$-line in  ${\bf H}^2_{\mathbb C}$. 
	The main result in \cite{Falbelparker:2003} is a representation is discrete and faithful if and only if $\alpha \in [\arccos(\frac{1}{4}), \frac{\pi}{2}]$.

(2)  When $\beta=\frac{\pi}{2}$, each  $A_i$ is in the form of $$\begin{pmatrix}
*& 0& *&*\\
0&* & 0& 0\\
*&0&*  &*\\
* & 0&*&*\\
\end{pmatrix}.$$
So $\langle A_0, A_1,A_2\rangle $ preserves a totally geodesic ${\bf H}^2_{\mathbb C} \hookrightarrow{\bf H}^3_{\mathbb C}$ invariant.
We will view the representation as degenerating to a representation into  $\mathbf{PU}(2,1)$. In fact, in this case, when  deleting the second row and the second column of $A_i$ for $i=0,1,2$, we get matrices in $\mathbf{PU}(2,1)$.  Now $A_2$ is  
$$\begin{pmatrix}
0& 0&1\\
0&1 & 0\\
1 & 0&0\\
\end{pmatrix}.$$
  Which is  a $\mathbb{C}$-reflection about the point $[-1,0,1]^{\top} \in {\bf H}^2_{\mathbb C}$ in Siegel model. These are groups studied in  \cite{FalbelKoseleff:2002}.  The relation between our angle $\alpha$ here and $\theta$ in  \cite{FalbelKoseleff:2002} is   $$\alpha=\pi- 6  \theta.$$
  \begin{itemize} 
  	\item When $\theta=\frac{\pi}{12}$, $\alpha=\frac{\pi}{2}$, the representation degenerates to ${\bf H}^1_{\mathbb C}$-geometry;
  	\item 
  	 When $\theta=\frac{\pi}{6}$, $\alpha=0$, the representation degenerates to ${\bf H}^2_{\mathbb R}$-geometry. 
   \end{itemize}
 
 
(3)  When $\alpha=0$, take $$V=\begin{pmatrix}
 1& 0& 0&0\\
 0& \rm{i} & 0& 0\\
0&0& \rm{i}  &0\\
 0 & 0&0&1\\
 \end{pmatrix}.$$
 Then $V^{*}H_s V=H_s$, so $V \in \mathbf{U}(3,1)$.  Moreover $VA_{i}V^{-1} \in \mathbf{PO}(3,1)$. In other words,  
  $\langle A_0, A_1,A_2\rangle $ preserves a totally geodesic ${\bf H}^3_{\mathbb R} \hookrightarrow{\bf H}^3_{\mathbb C}$ invariant.
 We  will view the representation as  degenerating to a representation into  $\mathbf{PO}(3,1)$. 
 
 We note that $$VA_0V^{-1}=\begin{pmatrix}
 1& -\sqrt{2} \cos(\beta)& -\sqrt{2} \sin(\beta)&-1\\
 0& -\cos(2\beta) & -\sin(2\beta)& -\sqrt{2} \cos(\beta)\\
 0& \sin(2\beta)& -\cos(2\beta)  &\sqrt{2} \sin(\beta)\\
 0 & 0&0&1\\
 \end{pmatrix}.$$
It has eigenvalues $$1, 1,~~-\cos(2\beta)-\sin(2\beta) \rm{i},~~-\cos(2\beta)+\sin(2\beta) \rm{i}.$$
  From this, it is easy to see $A_0$ is parabolic when $\beta=\frac{\pi}{2}$. 
  It is well-known that $\mathbf{PSL}(2,\mathbb{Z})$ is rigid in $\mathbf{PO}(3,1)$. Moreover, the representations in  the arc  $$\left \{\alpha=0, ~~\beta \in [0, \frac{\pi}{2}]\right\}$$   correspond to  a path in the so called $(2,3)$-slice in ${\bf H}^3_{\mathbb R}$-geometry. For example, see  \cite{Martin:2015}.
  So $(\alpha,\beta)=(0, \frac{\pi}{2})$ corresponds to the discrete and faithful  representation of  $\mathbf{PSL}(2,\mathbb{Z})$, and this point is the only point in this arc which corresponds to a discrete and faithful representation of  $\mathbf{PSL}(2,\mathbb{Z})$.


(4)  When $\alpha=\frac{\pi}{2}$, each  $A_i$ is in the form of $$\begin{pmatrix}
 *& 0& 0&*\\
 0&* & *& 0\\
 0&*&*  &0\\
 * & 0&0&*\\
 \end{pmatrix}.$$
 
 So $\langle A_0, A_1,A_2\rangle $ preserves a totally geodesic ${\bf H}^1_{\mathbb C} \hookrightarrow {\bf H}^3_{\mathbb C}$ invariant. 
 Moreover,  deleting the second and the third columns and rows of $A_i$ we get $2 \times 2$  matrices which are independent on $\beta$. We denote them also by $A_i$. 
Which are  
  $$A_1=\begin{pmatrix}
 0&1\\
  1& -\rm{i}
 \end{pmatrix}, ~~~ A_0
 =\begin{pmatrix}
 1&-\rm{i}\\
 0&1
 \end{pmatrix}.$$
So  the $\langle A_0, A_1,A_2\rangle $ action on this  ${\bf H}^1_{\mathbb C}$ is independent on $\beta$. It is easy to see the action is discrete. In other words, when $\alpha=\frac{\pi}{2}$, for any $\beta \in [0, \frac{\pi}{2}]$, the representation $\rho(\frac{\pi}{2}, \beta)$  corresponds to the unique discrete and faithful  representation of $\mathbf{PSL}(2,\mathbb{Z})$ into $\mathbf{PU}(1,1)$.

This ends the  proof of Theorem \ref{thm:modularmoduli2dim} for the moduli space $\mathcal{M}(0,\frac{2\pi}{3},\frac{4\pi}{3})$.

 When  $(\alpha,\beta) \in (0, \frac{\pi}{2})^2$, we have truely  ${\bf H}^3_{\mathbb C}$-geometry. At some point  $(\alpha,\beta) \in (0, \frac{\pi}{2})^2$, the representation is non-discrete or unfaithful.   But in some region in  $\mathcal{M}(0,\frac{2\pi}{3},\frac{4\pi}{3})$, we have discrete and faithful  representations of $\mathbf{PSL}(2,\mathbb{Z})$ into $\mathbf{PU}(3,1)$. This is what we will do in Section \ref{sec:modulardiscrete}.

\subsection{Accidental ellipticities in the moduli space $\mathcal{M}(0,\frac{2\pi}{3},\frac{4\pi}{3})$} \label{subsec:accidentalelliptic}

Recall the presentation of $\mathbf{PSL}(2,\mathbb{Z})$ in (\ref{psl2z}). For $\rho \in \mathcal{M}(0,\frac{2\pi}{3},\frac{4\pi}{3})$, we denote $A_{i}=\rho(a_i) \in \mathbf{PU}(3,1)$ for $i=0,1,2$.
	For $m_{i}, n_{i} \in \mathbb{Z}_{+}$,   we denote by $$w=w(m_1,n_1, \cdots, m_{k}, n_{k})$$  the element  $$(a_2a_1)^{m_1}(a_2a^{-1}_1)^{n_1}(a_2a_1)^{m_2}(a_2a^{-1}_1)^{n_2}\cdots (a_2a_1)^{m_k}(a_2a^{-1}_1)^{n_k}$$ in  $\mathbf{PSL}(2,\mathbb{Z})$. Any word in  $\mathbf{PSL}(2,\mathbb{Z})$ has a unique presentation in this form up to cyclic rotation of $$(m_1,n_1,m_2,n_2, \cdots, m_{k}, n_{k})$$ to $$(m_i,n_i,m_{i+1},n_{i+1}, \cdots, m_{i-1}, n_{i-1}).$$
	We denote by $$W=W(m_1,n_1, \cdots, m_{k}, n_{k})=\rho(w)$$ the word on $A_1,A_2$ for $\rho \in \mathcal{M}(0,\frac{2\pi}{3},\frac{4\pi}{3})$. 	The word $(\iota_0 \iota_1 \iota_2)^2$ in \cite{Falbelparker:2003} is the word $W(1,1)$ here when $\beta=0$.



When considering accidental ellipticities in  the moduli space $\mathcal{M}(0,\frac{2\pi}{3},\frac{4\pi}{3})$, the first word we should take care is  $W(1,1)=A_2A_1A_2A^{-1}_1$. 
	
	\begin{prop}\label{prop:W(1,1)} For any pair  $(\alpha,\beta) \in (0, \frac{\pi}{2})^2$,  the word $A_2A_1A_2A^{-1}_1$ is loxodromic.		
	\end{prop}

	\begin{proof}
	We note that the trace of $W(1,1)=A_2A_1A_2A^{-1}_1$ is 
$$5-4\cos(2\beta)\cos(\alpha)-2\sin^2(2\beta)\cos(\frac{2 \alpha}{3})-2\sin^2(2\beta).$$ 
In particular, it is always real.
And  $\sigma(A_2A_1A_2A^{-1}_1)$ is 
$$ 8-8 \cos(2\beta)\cos(\alpha)-24\cos^2(\beta)+\sin^2(2\beta)(2\cos(\frac{4 \alpha}{3})-4\cos(\frac{2 \alpha}{3})-6).$$

Then we can get $\mathcal{H}(A_2A_1A_2A^{-1}_1)$, which  is $$-8\sin^2(2\beta)(\cos(\frac{4 \alpha}{3})-1) \cdot X \cdot Y^2.$$
The terms $X$ and $Y$ in terms of $\alpha,\beta$ are easy to get via Maple, but we omit their explicit expressions.

 In  particular, when $\beta=0$,  $\beta=\frac{\pi}{2}$, or $\alpha=0$,  then $\mathcal{H}(A_2A_1A_2A^{-1}_1)=0$. We are interested  whether there are other pairs of $(\alpha,\beta)$ such that $\mathcal{H}(A_2A_1A_2A^{-1}_1)=0$. By Maple,  the maximum of $X$ when $(\alpha,\beta) \in [0, \frac{\pi}{2}]^2$ is zero at the pair $(\alpha,\beta)=(0, \frac{\pi}{6})$. 
The minimum of $Y$ is zero at the pairs  $(\alpha,\beta)=(\arccos(\frac{1}{4}),0)$ or $(0,\frac{\pi}{3})$. See Figure 	\ref{figure:A2A1A2a1hzero} for an illustration  of these facts. 
Then  by Theorem \ref{thm:holy}, for any pair  $(\alpha,\beta) \in (0, \frac{\pi}{2})^2$,  $A_2A_1A_2A^{-1}_1$ is loxodromic. 
		\end{proof}

	\begin{figure}
		\begin{center}
			\begin{tikzpicture}
			\node at (0,0) {\includegraphics[width=6cm,height=6cm]{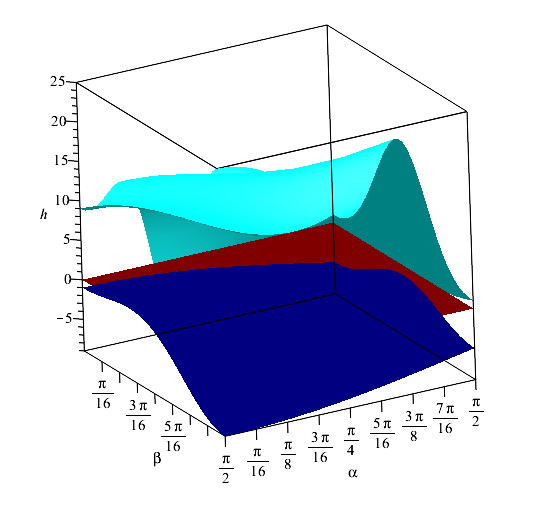}};
			\end{tikzpicture}
		\end{center}
		\caption{The blue surface (the below one) is the graphic of $X$ for $\mathcal{H}(A_2A_1A_2A^{-1}_1)$. The cyan surface (the top one) is the graphic of $Y$ for $\mathcal{H}(A_2A_1A_2A^{-1}_1)$. The red surface (the plane) is when $h=0$. There are exactly three tangency points between the  red surface and the union of the blue and cyan surfaces. This illustrates the fact that  $A_2A_1A_2A^{-1}_1$ is always loxodromic when $\alpha, \beta \in (0, \frac{\pi}{2})$.}
		\label{figure:A2A1A2a1hzero}
	\end{figure}

	Proposition \ref{prop:W(1,1)} implies that the word $A_2A_1A_2A^{-1}_1$ is NOT  responsible for the accidental parabolicities of group representations in $ (0, \frac{\pi}{2})^2$ (if there are). So the behaviors of the representations  near $(\arccos(\frac{1}{4}),0)$ are very mysterious.

	The author also notes that for $W(2,1)=(A_2A_1)^2A_2A^{-1}_1$, $\tau(W(2,1))$ and $\sigma(W(2,1))$ are not very complicated (each of them occupies two lines). Then we can get $\mathcal{H}(W(2,1))$, in turn we have Figure \ref{figure:A2A1A2A1A2a1}, which is a very rough  approximation of the locus  $\mathcal{H}((A_2A_1)^2A_2A^{-1}_1)=0$.  But for other word $W \in \rho(\mathbf{PSL}(2,\mathbb{Z}))$, the author has difficulties when   approximating  the locus  $\mathcal{H}(W)=0$.

	At some pair $(\alpha,\beta) \in (0, \frac{\pi}{2})^{2}$, some elements $\mathbf{PSL}(2,\mathbb{Z})$ are accidental elliptic in $\rho(\mathbf{PSL}(2,\mathbb{Z}))$. 
	For example, when $\alpha=\beta=0.2$, then $\mathcal{H}(W)>0$ for $W$ be any of $W(2,1)$, $W(4,1)$, $W(6,1)$,  $W(13,1)$,  $W(1,2,3,2)$,  $W(1,2,3,3)$,  $W(1,2,3,4)$, $W(1,2,3,1,2,1)$, $W(1,2,3,1,2,1)$ and  $W(1,2,3,1,2,2)$. For other pairs of $(\alpha,\beta)$, $W(k,1)$ is accidental elliptic for many  $k \in \{2,3, \cdots, 30\}$. 
	It seems to the author there are infinitely many words in  $\mathbf{PSL}(2,\mathbb{Z})$ which is  accidental elliptic at some pair $(\alpha,\beta) \in (0, \frac{\pi}{2})^{2}$.  The accidental ellipticities/parabolicities   in  the moduli space $\mathcal{M}(0,\frac{2\pi}{3},\frac{4\pi}{3})$ deserve further study, see Question \ref{ques:critical}.

	\begin{figure}
		\begin{center}
			\begin{tikzpicture}
			\node at (0,0) {\includegraphics[width=6.5cm,height=6cm]{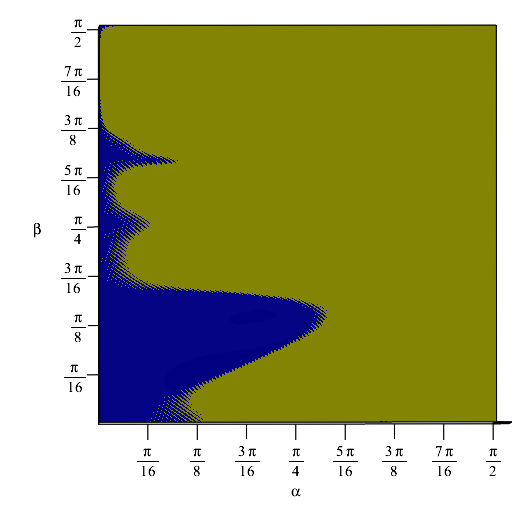}};
			\end{tikzpicture}
		\end{center}
		\caption{An approximation of the locus where $\mathcal{H}((A_2A_1)^2A_2A^{-1}_1)=0$, which is the intersection of the blue region and the yellow region. In the blue region,  $(A_2A_1)^2A_2A^{-1}_1$ is regular elliptic. In the yellow region, $(A_2A_1)^2A_2A^{-1}_1$ is regular loxodromic.}
		\label{figure:A2A1A2A1A2a1}
	\end{figure}


	

\subsection{The moduli space $\mathcal{M}(\frac{2\pi}{3},\frac{4\pi}{3},\frac{4\pi}{3})$ } \label{subsec:122}

We also  have 2-dimensional  representations of $\mathbf{PSL}(2,\mathbb{Z})$ into $\mathbf{PU}(3,1)$, such that the generator $A'_2=\rho(a_2)$ is a $\pi$-rotation about a $\mathbb{C}$-line, and $A'_1=\rho(a_1)$ is elliptic of $(\frac{2\pi}{3}, \frac{4\pi}{3}, \frac{4 \pi}{3})$ type.

	Let $\iota'_{i}$ be the  $\mathbb{R}$-reflection in $\mathbf{H}^3_{\mathbb C}$ with invariant $\mathbb{R}^3$-chain $\mathcal{R}'_{i}$ for $i=0,1,2$. We take $$A'_2=\iota'_1\iota'_{0}, ~~~~  A'_1=\iota'_0\iota'_{2}, ~~\text{and ~~then} ~~ A'_2A'_1=\iota'_1\iota'_{2} ~~\text {is parabolic}.$$  With the same  $\iota_{0}$,  $\iota_{1}$,  $\mathcal{R}_{0}$ and $\mathcal{R}_{1}$ as in Subsection \ref{subsection:012}. That is, we take $\iota'_{0}=\iota_{0}$ and   $\iota'_{1}=\iota_{1}$. We now rotate the $\mathbb{R}^3$-chain $\mathcal{R}_{2}$ in certain way  to get a new $\mathbb{R}^3$-chain $\mathcal{R}'_{2}$. This corresponds to conjugating  $\iota_{2}$ into $\iota'_{2}$. In turn we have   $\iota'_{0} \iota'_{2}$ is $(\frac{2\pi}{3}, \frac{4\pi}{3},\frac{4\pi}{3})$ type elliptic.

The Lagrangian inversion $\iota'_{2}=\iota'_{2}(\alpha,\beta)$ depends $\alpha$ and $\beta$.  When $\beta=0$, we 
take $$X=\begin{pmatrix}
	1& 0& 0&0\\
	0&1 & 0&0\\
	0&0& \rm{e}^{\frac{\pi\rm{i}}{3}} &0\\
	0 & 0&0&1\\
\end{pmatrix}.$$ 
Then $X^* H_s X=H_s$, so $X \in \mathbf{U}(3,1)$.
 Take $\mathcal{R}'_2(\alpha,0)=X (\mathcal{R}_2(\alpha,0))$, and $$\iota'_2(\alpha,0)= X \circ  \iota _2(\alpha,0) \circ X^{-1}$$ as an isometry of  $\mathbf{H}^3_{\mathbb C}$.
The presentation matrix of $\iota'_2(\alpha,0)$ is just $$M'_2(\alpha,0)=X \cdot   M_2(\alpha,0) \cdot \overline{X^{-1}}.$$
We normalize  the presentation matrix of $\iota'_2=\iota'_{2}(\alpha,0)$ such that $M_2(\alpha,0)$ and $M'_2(\alpha,0)$ have the same determinant, so we have 
$$M'_2(\alpha,0)=\rm{e}^{-\frac{\pi\rm{i}}{6}}\begin{pmatrix}
1& \rm{i} \rm{e}^{-\frac{\alpha\rm{i}}{2}} \sqrt{2\cos(\alpha)}& 0&-\rm{e}^{-\alpha\rm{i}}\\
0&\rm{e}^{-\alpha\rm{i}} & 0& \rm{i} \rm{e}^{-\frac{\alpha\rm{i}}{2}} \sqrt{2\cos(\alpha)}\\
0&0&-\rm{e}^{-\frac{\alpha\rm{i}}{3}} \rm{e}^{\frac{2\pi\rm{i}}{3}} &0\\
0 & 0&0&1\\
\end{pmatrix}.$$
The  $\mathbb{R}^3$-chain $\mathcal{R}'_{2}(\alpha,0)$ is given by 
$$\left\{ (x_1+y_1 {\rm i}, x_2+y_2 {\rm i},t)~~|~~~x_1, y_1, x_2,y_2,t \in \mathbb{R} \right\}$$ with  equations 
\begin{equation}  \label{equation:R2primebetizero}
\left\{ \begin{aligned}
0=2x_1\sin(\frac{\alpha}{2})+  2y_1\cos(\frac{\alpha}{2})-\sqrt{\cos(\alpha)};  \qquad &   \\ 
0=t-x_1 \cos(\frac{\alpha}{2})\sqrt{\cos(\alpha)}+  y_1\sin(\frac{\alpha}{2})\sqrt{\cos(\alpha)} -\frac{\sin(\alpha)}{2};
\qquad &   \\
-{\rm e}^{-\frac{\alpha {\rm i}}{3}}{\rm e}^{\frac{2\pi{\rm i}}{3}}(x_2-y_2 {\rm i})=x_2+y_2 {\rm i}. \qquad&  
\end{aligned}
\right.\end{equation}





For general $\beta$, we take $$\iota'_{2}(\alpha,\beta)=U \cdot \iota'_{2} (\alpha,0) \cdot U^{-1}, $$ where $U$ is in (\ref{matrix:U}). 
	Now we have the presentation  matrix of $\iota'_1\iota'_2$ is $A'_0$, which is   $\rm{e}^{\frac{-\alpha\rm{i}}{3}}$ times
$$\begin{pmatrix}
\rm{e}^{\frac{\pi\rm{i}}{6}}& \rm{e}^{(\frac{\alpha}{2}-\frac{\pi}{3}) \rm{i}} \sqrt{2\cos(\alpha)}\cos(\beta)& \rm{e}^{(\frac{\alpha}{2}-\frac{\pi}{3}) \rm{i}} \sqrt{2\cos(\alpha)}\sin(\beta)&-\rm{e}^{(\alpha+\frac{\pi}{6})\rm{i}}\\
0&-\rm{i}\sin^2(\beta)\rm{e}^{\frac{\alpha\rm{i}}{3}}-\cos^2(\beta)\rm{e}^{(\alpha+\frac{\pi}{6})\rm{i}}& \frac{\sin(2\beta)}{2}(\rm{i}\rm{e}^{\frac{\alpha \rm{i}}{3}}-\rm{e}^{(\alpha+\frac{\pi }{6})\rm{i}})& -\rm{e}^{(\frac{\alpha}{2}-\frac{\pi}{3}) \rm{i}} \sqrt{2\cos(\alpha)}\cos(\beta)\\
0&\frac{\sin(2\beta)}{2}(\rm{e}^{(\alpha+\frac{\pi }{6})\rm{i}}-\rm{i}\rm{e}^{\frac{\alpha \rm{i}}{3}})&\rm{i}\cos^2(\beta)\rm{e}^{\frac{\alpha\rm{i}}{3}}+\sin^2(\beta)\rm{e}^{(\alpha+\frac{\pi}{6})\rm{i}}& \rm{e}^{(\frac{\alpha}{2}-\frac{\pi}{3}) \rm{i}} \sqrt{2\cos(\alpha)}\sin(\beta)\\
0 & 0 &0&\rm{e}^{\frac{\pi\rm{i}}{6}}\\
\end{pmatrix}.$$
 The presentation matrix of $\iota'_0\iota'_2$ is $A'_1$, which is   $\rm{e}^{\frac{-\alpha\rm{i}}{3}}$ times $$ \begin{pmatrix}
0& 0& 0&\rm{e}^{\frac{\pi\rm{i}}{6}}\\
0&\rm{i}\sin^2(\beta)\rm{e}^{\frac{\alpha\rm{i}}{3}}+\cos^2(\beta)\rm{e}^{(\alpha+\frac{\pi}{6})\rm{i}}&\frac{\sin(2\beta)}{2}(\rm{e}^{(\alpha+\frac{\pi }{6})\rm{i}}-\rm{i}\rm{e}^{\frac{\alpha \rm{i}}{3}})&   \rm{e}^{(\frac{\alpha}{2}-\frac{\pi}{3}) \rm{i}} \sqrt{2\cos(\alpha)}\cos(\beta)\\
0&\frac{\sin(2\beta)}{2}(\rm{e}^{(\alpha+\frac{\pi }{6})\rm{i}}-\rm{i}\rm{e}^{\frac{\alpha \rm{i}}{3}})& \rm{i}\cos^2(\beta)\rm{e}^{\frac{\alpha\rm{i}}{3}}+\sin^2(\beta)\rm{e}^{(\alpha+\frac{\pi}{6})\rm{i}}&\rm{e}^{(\frac{\alpha}{2}-\frac{\pi}{3}) \rm{i}} \sqrt{2\cos(\alpha)}\sin(\beta)\\
\rm{e}^{\frac{\pi\rm{i}}{6}} & \rm{e}^{(\frac{\alpha}{2}-\frac{\pi}{3}) \rm{i}} \sqrt{2\cos(\alpha)}\cos(\beta)&\rm{e}^{(\frac{\alpha}{2}-\frac{\pi}{3}) \rm{i}} \sqrt{2\cos(\alpha)}\sin(\beta)&-\rm{e}^{(\alpha+\frac{\pi}{6})\rm{i}}\\
\end{pmatrix}.$$
The presentation matrix of $\iota'_1\iota'_0$ is just   $A'_2=A_2$. 

Then it can be showed directly  that $$ ~~(A'_1)^2=I_4,~~ (A'_1)^3=- {\rm i} \cdot I_{4}, ~~A'_2A'_0(A'_1)^{-1}=I_4$$ and  $$(A'_{i})^{*}H_sA'_{i}=H_s,~~\det(A'_i)=1$$ for $i=0,1,2$. Moreover
$A'_0$ is parabolic with fixed point $\infty$ when $\alpha \neq 0$ by Proposition \ref{prop:A0primepara}. So we have a type-preserving representation of  $\mathbf{PSL}(2,\mathbb{Z})$ into $\mathbf{PU}(3,1)$ for $\alpha,\beta \in [0, \frac{\pi}{2}]$. From Proposition \ref{prop:typeA1122} later, $A'_1$ is $(\frac{2\pi}{3},\frac{4\pi}{3},\frac{4\pi}{3})$-type elliptic. We denote this moduli space  by $\mathcal{M}(\frac{2\pi}{3},\frac{4\pi}{3},\frac{4\pi}{3} )$, then it  is  parameterized by  $[0, \frac{\pi}{2}]^2$.

Proposition \ref{prop:A0primepara} is a little different from Proposition \ref{prop:A0para} for the moduli space  $\mathcal{M}(0,\frac{2\pi}{3},\frac{4\pi}{3} )$.
\begin{prop}\label{prop:A0primepara} For the moduli space $\mathcal{M}(\frac{2\pi}{3},\frac{4\pi}{3},\frac{4\pi}{3} )$:
		\begin{itemize} 
			\item When $\alpha \in (0, \frac{\pi}{2}]$, for any $\beta \in [0, \frac{\pi}{2}]$, then  $\mathcal{R}'_{1} \cap \mathcal{R}'_{2} = \emptyset$, so $A_0$ is parabolic;
			\item When $\alpha=0$ and  $\beta\in (0, \frac{\pi}{2}]$, then   $\mathcal{R}'_{1} \cap \mathcal{R}'_{2} = \emptyset$, so $A_0$ is parabolic;
			\item When $\alpha=0$ and   $ \beta =0$,   then $\mathcal{R}_{1} \cap \mathcal{R}_{2} \neq  \emptyset$, so $A_0$ elliptic. 
		\end{itemize} 
\end{prop}

\begin{proof}Since $\mathcal{R}'_{2}(\alpha,\beta)=U(\mathcal{R}'_{2}(\alpha,0))$,  we only need to study  the disjointness of  $U^{-1}(\mathcal{R}_{1})$ and $\mathcal{R}'_{2}(\alpha,0)$. The $\mathbb{R}^3$-chain  $U^{-1}(\mathcal{R}_{1})$ is given by
	\begin{equation} \label{equation:UR1new}
	\left\{ (\cos(\beta)r_1 {\rm i}+sin(\beta) r_2, -\sin(\beta)r_1 {\rm i}+\cos(\beta)r_2,0)~~|~~ r_1, r_2 \in \mathbb{R} \right\}
	\end{equation}
	in the Heisenberg group.

		First consider the case $\alpha=0$. If $$p=(x_1+y_1 {\rm i}, x_2+y_2 {\rm i},~~t) \in U^{-1}(\mathcal{R}_{1}) \cap \mathcal{R}'_{2}(\alpha,0).$$
		Then we have $t=0$ from (\ref{equation:UR1new}).  By  (\ref{equation:R2primebetizero}), we have $x_1=0$ and $y_1=\frac{1}{2}$. In turn  $$p=(\frac{{\rm i}}{2}, r \cdot (\cos(\frac{\pi}{6})-\sin(\frac{\pi}{6}) {\rm i}),0)= (\cos(\beta)r_1 {\rm i}+sin(\beta) r_2, -\sin(\beta)r_1 {\rm i}+\cos(\beta)r_2,0)$$
	 for some $r \in \mathbb{R}$. 
	 \begin{itemize}
	 	
	 	\item If $\beta=0$, then $(\frac{\rm{i}}{2},0,0)$ is the common point of $U^{-1}(\mathcal{R}_{1})$ and $\mathcal{R}'_{2}(\alpha,0)$, so $A_0$ is elliptic;
	 	\item If $\beta \in (0, \frac{\pi}{2}]$, then $r_2=0$ and $U^{-1}(\mathcal{R}_{1}) \cap \mathcal{R}'_{2}(\alpha,0)=\emptyset$. So $A_0$ is parabolic.
		 \end{itemize}
		
		When  $\alpha \in (0, \frac{\pi}{2}]$. We assume $$p=(a_1  {\rm e}^{\phi_1\rm{i}}, ~~a_2 \rm{e}^{\phi_2\rm{i}},~~0) \in U^{-1}(\mathcal{R}_{1}) \cap \mathcal{R}'_{2}(\alpha,0)$$
		with $a_1, a_2  \in \mathbb{R}$. Then we may assume  $\phi_1=\phi_2+\frac{\pi}{2}$ since $U^{-1}(\mathcal{R}_{1})$ is just a rotation of $\mathcal{R}_{1}$.  By the third equation in (\ref{equation:R2primebetizero}), we have $$\phi_2= \frac{5\pi}{6} -\frac{\alpha}{6}~~ \text{or} ~~ -\frac{\pi}{6} -\frac{\alpha}{6}$$ if  $a_2\neq 0$. When  $a_2=0$, $\phi_2$ is not defined. 	
		
	If $a_2\neq 0$, we have  $$\phi_1= \frac{4\pi}{3}  -\frac{\alpha}{6} ~~\text{or} ~~~ \frac{\pi}{3} -\frac{\alpha}{6}.$$
		We now consider the first two equations in  (\ref{equation:R2primebetizero}), they are now
		\begin{equation}	 \label{equation:R2primealphazerowithphi1}
		\left\{ \begin{aligned}
		0=2a_1\sin(\frac{\alpha}{2}+\phi_1)-\sqrt{\cos(\alpha)};  \qquad &   \\ 
		0=-\sqrt{\cos(\alpha)}a_1 \cos(\frac{\alpha}{2}+\phi_1) -\frac{\sin(\alpha)}{2}.
		\qquad &   \\ 
		\end{aligned} 
		\right.\end{equation}
		Since $\sin(\alpha) >0$,  from the second equation in 	(\ref{equation:R2primealphazerowithphi1}),  we have $$\sqrt{\cos(\alpha)}a_1 \cos(\frac{\alpha}{2}+\phi_1) <0,$$  then $\alpha \in (0, \frac{\pi}{2})$ and  $a_1 \cos(\frac{\alpha}{2}+\phi_1) <0$.
		\begin{itemize}
			\item If $\phi_1= \frac{4\pi}{3}  -\frac{\alpha}{6}$,   we have $a_1>0$ and $\sin(\frac{\alpha}{2}+\phi_1) <0$. Which contradicts to the first equation in   
			(\ref{equation:R2primealphazerowithphi1});
			\item If $\phi_1=\frac{\pi}{3}  -\frac{\alpha}{6}$,  we have $a_1<0$ and $\sin(\frac{\alpha}{2}+\phi_1) >0$. Which also contradicts to the first equation in   
			(\ref{equation:R2primealphazerowithphi1}).
		\end{itemize}  	
	
		If $a_2=0$, then $$p=(\cos(\beta)r_1 {\rm i}+\sin(\beta) r_2, ~~-\sin(\beta)r_1 {\rm i}+\cos(\beta)r_2,~~0)=(x_1+y_1 {\rm i},~~0,0).$$
		So $$-\sin(\beta)r_1=\cos(\beta)r_2=0.$$ At least one of $r_1$ and $r_2$ is zero.
		\begin{itemize}
			\item If $r_1=r_2=0$, then $p=(0,0,0)$, which does not lie in  $\mathcal{R}'_{2}(\alpha,0)$;
			\item If $r_1=0$ but $r_2 \neq 0$, then $\cos(\beta)=0$ and  $p=(r_2,0,0)$. Which does  not lie in  $\mathcal{R}'_{2}(\alpha,0)$ since $\alpha \in (0, \frac{\pi}{2}]$;
			\item If $r_2=0$ but $r_1 \neq 0$, then $\sin(\beta)=0$ and $p=(r_1 \rm{i},0,0)$. Which does not lie in  $\mathcal{R}'_{2}(\alpha,0)$ since $\alpha \in (0, \frac{\pi}{2}]$. 
		\end{itemize}
		So when  $\alpha \in (0, \frac{\pi}{2}]$,  in either cases we have   $U^{-1}(\mathcal{R}'_{1}) \cap \mathcal{R}'_{2}(\alpha,0) = \emptyset$, and $A'_0$ is parabolic  by Proposition \ref{prop:asymptotic}.
\end{proof}

\begin{prop}\label{prop:R0primeR2primenointersection} When $\alpha=\beta=0$, then $\mathbb{R}^3$-chains $\mathcal{R}'_{0}$ and  $\mathcal{R}'_{2}$ do not intersect.
	
\end{prop}

\begin{proof}Recall $\mathcal{R}'_0=\mathcal{R}_0$. We assume $$p=(x_1+y_1 {\rm i}, x_2+y_2 {\rm i},t) \in \mathcal{R}'_{0} \cap \mathcal{R}'_{2}(0,0).$$
	By 	(\ref{equation:R2primebetizero}), we have $y_1=\frac{1}{2}$,  $t=x_1$, so $x_1=0$
	by the second   equation of  (\ref{equation:standardimaginaryRchain}). And then $x_2=0$  by the third equation of  (\ref{equation:standardimaginaryRchain}). 
	On the other hand, we have  $x_2+y_2 \rm{i}=r \cdot (\cos(\frac{\pi}{6})-\sin(\frac{\pi}{6}) \rm{i})$ for some $r \in \mathbb{R}$ by the third equation of (\ref{equation:R2primebetizero}). 
	So $r=0$, and $p=(\frac{\rm{i}}{2}, 0,0)$,  but it does not lie in $\mathcal{R}'_0$.
	
\end{proof}

From Proposition \ref{prop:R0primeR2primenointersection}, if  $\iota'_0 \iota'_2$ has order three, then it is not the type of $(0,\frac{a\pi}{3},\frac{b\pi}{3})$ for any $a, b \in \{0,2,4\}$. But we still need to show its type carefully. 
\begin{prop}\label{prop:typeA1122}For $(\alpha, \beta) \in [0, \frac{\pi}{2}]^2$, $A'_1$ is a $(\frac{2\pi}{3}, \frac{4\pi}{3},\frac{4\pi}{3})$-type elliptic element.
\end{prop}
\begin{proof}
	When $\alpha=\frac{\pi}{2}$ and $\beta=0$, the matrix $-\rm{i} A'_1$ has eigenvalues $\lambda_i$ for $i=1,2,3,4$ as in Subsection \ref{subsection:order3}. The eigenvectors are 
	$$(\frac{-\sqrt{3}+\rm{i}}{2},0,0,1)^{\top},$$
	$$(\frac{\sqrt{3}+\rm{i}}{2},0,0,1)^{\top},$$
	$$(0,1,0,0)^{\top}$$
	and 
	$$(0,0,1,0)^{\top}.$$
	They are negative, positive, positive and  positive vectors respectively with respect to the Hermitian matrix  $H_s$. So $-\rm{i}A'_1$ is $(\frac{2\pi}{3}, \frac{4\pi}{3},\frac{4\pi}{3})$-type elliptic element, but not  $(\frac{2\pi}{3}, \frac{2\pi}{3},\frac{4\pi}{3})$-type elliptic element, as the argument  in Subsection  \ref{subsection:order3}. 
Then as the proof of Proposition 	\ref{prop:A1012}, for general  $(\alpha, \beta) \in [0, \frac{\pi}{2}]^2$, $A'_1$ is a  $(\frac{2\pi}{3}, \frac{4\pi}{3},\frac{4\pi}{3})$-type elliptic element.
\end{proof}

\subsection{Proof of Theorem \ref{thm:modularmoduli2dim} for the moduli space $\mathcal{M}(\frac{2\pi}{3},\frac{4\pi}{3},\frac{4\pi}{3})$} \label{subsec:proofmoduli122}


In this subsection, we consider some  special representations in the moduli space  $\mathcal{M}(\frac{2\pi}{3},\frac{4\pi}{3},\frac{4\pi}{3})$.

(1) When $\beta=0$, each  $A'_i$ is in the form of $$\begin{pmatrix}
	*& *& 0&*\\
	*&* & 0& *\\
	0&0&*  &0\\
	* & *&0&*\\
	\end{pmatrix}.$$
	So $\langle A'_0, A'_1,A'_2\rangle $ preserves a totally geodesic ${\bf H}^2_{\mathbb C} \hookrightarrow {\bf H}^3_{\mathbb C}$ invariant. 	Moreover,  when  deleting the third row and the third column of $A'_i$ for $i=0,1,2$, we get matrices in $\mathbf{PU}(2,1)$. Which  equal to the matrices obtained from  $A_i$ in Subsection  \ref{subsection:012} by deleting the third row and the third column of them for $i=0,1,2$ (up to scaling matrices).
	So $\langle A'_0, A'_1,A'_2\rangle $ is essentially  the group  studied in   \cite{Falbelparker:2003}.

(2).  When $\beta=\frac{\pi}{2}$, arguments as above, we can see that $\langle A'_0, A'_1,A'_2\rangle $ is essentially  the group  studied in   \cite{FalbelKoseleff:2002}. 

(3). When $\alpha=\frac{\pi}{2}$, each  $A'_i$ is in the form of $$\begin{pmatrix}
*& 0& 0&*\\
0&* & *& 0\\
0&*&*  &0\\
* & 0&0&*\\
\end{pmatrix}.$$
So $\langle A'_0, A'_1,A'_2\rangle $ preserves a totally geodesic ${\bf H}^1_{\mathbb C} \hookrightarrow {\bf H}^3_{\mathbb C}$ invariant.
Moreover,   when  deleting   the second and the third columns and rows of $A'_i$ and $A_i$,  we get the same matrix for $i=0,1,2$.
So $\langle A'_0, A'_1,A'_2\rangle $ is the same as the  group  $\langle A_0, A_1,A_2\rangle$ in Subsection  \ref{subsection:012}
when acting on this  totally geodesic ${\bf H}^1_{\mathbb C} \hookrightarrow {\bf H}^3_{\mathbb C}$.

(4). When $\alpha=0$, the trace of $A'_0$ is 
$$\frac{3\sqrt{3}}{2}+\frac{\rm{i}}{2}-\sqrt{3}\cos^2(\beta)+\rm{i}\cos^2(\beta).$$
 Which does not lie in $\mathbb{R} \cup \rm{i}\mathbb{R}$. So there is no matrix $V$ with $V^{*} H_s V=H_s$ such that  $VA_{i}V^{-1} \in \mathbf{PO}(3,1)$ for $i=0,1,2$. In other words,  
$\langle A_0, A_1,A_2\rangle $ does not preserves  a totally geodesic ${\bf H}^3_{\mathbb R} \hookrightarrow {\bf H}^3_{\mathbb C}$ invariant. The author notes that  when 
$$(\alpha,\beta)=(0,0) ~~ \text{or} ~~(0, \frac{\pi}{2}),$$ then $\rho(\mathbf{PSL}(2,\mathbb{Z}))$
is a subgroup of $\mathbf{U}(2,1) \times \mathbf{U}(1)$. So they both stabilize a $${\bf H}^2_{\mathbb R}\hookrightarrow {\bf H}^2_{\mathbb C}\hookrightarrow {\bf H}^3_{\mathbb C}$$ invariant. When deleting
\begin{itemize}
	
	\item the second row and the second column of $\rho(a_i)=A'_i$ for $i=0,1,2$ if $(\alpha,\beta)=(0,0)$;
\item the third row and the third column of $\rho(a_i)=A'_i$ for $i=0,1,2$ if  $(\alpha,\beta)=(0,\frac{\pi}{2})$, 
\end{itemize}
we get subgroups of  $\mathbf{U}(2,1)$ with real entries. 
This interprets   why when $(\alpha,\beta)=(0,0) ~~ \text{or} ~~(0, \frac{\pi}{2})$, we have groups which  have  ${\bf H}^2_{\mathbb R}$-geometry essentially.

\section{The proof of Theorem \ref{thm:modulardiscrete} via $\mathbb{C}$-spheres} \label{sec:modulardiscrete}

In  this section, we prove Theorem \ref{thm:modulardiscrete}  for the moduli space $\mathcal{M}(0,\frac{2\pi}{3},\frac{4\pi}{3})$.
 We will construct a  region  $D \subset\partial {\bf H}^3_{\mathbb C}$ co-bounded by  $\mathbb{C}$-spheres. Then by the  Poincar\'e polyhedron theorem,   $D \subset\partial {\bf H}^3_{\mathbb C}$ is  a fundamental domain of the group. We refer to \cite{FalbelZocca:1999} for the precise statement of this
version of Poincar\'e polyhedron theorem we need.  

For each $\rho \in \mathcal{M}(0,\frac{2\pi}{3},\frac{4\pi}{3})$, $\rho(\mathbf{PSL}(2,\mathbb{Z}))$ is generated by $A_1,A_2$.  Consider $$c_1=A_2=\iota_1\iota_0,~~ c_0=A_1A_2A_1^{-1},~~c_2=A_1^{-1}A_2A_1.$$ Each $c_i$ is a $\pi$-rotation along a $\mathbb{C}$-line in ${\bf H}^{3}_{\mathbb C}$, and   $\langle c_0,c_1,c_2\rangle$ is an index three subgroup of $\langle A_1,A_2\rangle$. We will show  $\langle c_0,c_1,c_2\rangle$  is a discrete group for a 2-dimensional subset $\mathcal{N}$ of $\mathcal{M}(0,\frac{2\pi}{3},\frac{4\pi}{3})$, which in turn implies $\langle A_1,A_2\rangle$ is discrete.



For any   $(\alpha, \beta) \in [0,\frac{\pi}{2}]^2$, the corresponding representation in $\mathcal{M}(0,\frac{2\pi}{3},\frac{4\pi}{3})$ is denoted by $\rho=\rho(\alpha,\beta)$. The  main technical result in this section is 
\begin{thm} \label{thm:cspheres} There is a neighborhood  $\mathcal{N}$ of  $(\alpha, \beta)=(\frac{\pi}{2},0)$ or $(\frac{\pi}{2},\frac{\pi}{2})$ in $[0, \frac{\pi}{2}]^2$. For any representation  $\rho$ corresponds to a point in $\mathcal{N}$ with $c_0,c_1,c_2 \in \rho(\mathbf{PSL}(2,\mathbb{Z}))$, there are embedded  $\mathbb{C}$-spheres $S_0$, $S_1$ and $S_2$ in $\mathbb{C}^2 \times \mathbb{R}$ such that
	\begin{itemize}
		
\item  $c_{j}(S_{j})=S_{j}$;
\item  $S_{j}$ and   $S_{k}$ are disjoint except for the point of tangency at a parabolic element fixed point when $j \neq k$;
\item $S_0$, $S_1$ and $S_2$ are the faces of a fundamental domain $D$ for $\langle c_0,c_1,c_2\rangle$;
\item  $\langle c_0,c_1,c_2\rangle$ is a discrete subgroup of $\mathbf{PU}(3,1)$ which is isomorphic to $*^3 \mathbb{Z}_2$ abstractly.

			\end{itemize}	
		
In particular,  $\rho$  	is a discrete and faithful representation of  $\mathbf{PSL}(2,\mathbb{Z})$  into $\mathbf{PU}(3,1)$ when $(\alpha, \beta) \in \mathcal{N}$.
	
\end{thm}

The neighborhood $\mathcal{N}$ of $(\alpha,\beta)=(\frac{\pi}{2},0)$ or $(\frac{\pi}{2},\frac{\pi}{2})$ in $[0,\frac{\pi}{2}]^2$ shall satisfy several conditions, 
so we will determine  $\mathcal{N}$ only in the end of the proof.  

By the $\mathbb{Z}_3$-symmetry of the group  $\langle c_0,c_1,c_2\rangle$, Theorem  \ref{thm:cspheres} can be obtained by the combination of Propositions  \ref {prop:S1} and \ref{prop:S1S2disjoint}. 
Recall that a  $\mathbb{C}^2$-chain is the boundary of a totally geodesic complex hyperbolic plane in ${\bf H}^3_{\mathbb C}$, so it is topologically a  3-sphere. 

\begin{prop} \label{prop:S1}  For $(\alpha, \beta) \in \mathcal{N}$, the $\mathbb{C}$-sphere $S_1$ can be decomposed into three parts $S_{1,-}$,  $S_{1,0}$ and  $S_{1,+}$ along $\mathbb{C}^2$-chains. Where $S_{1,\pm}$ are topologically 4-balls, and $S_{1,0}$ is topologically  $\mathbb{S}^3\times [-1,1]$. Moreover, $S_{1,0}$ is $c_1$-invariant, $c_1$ exchanges  $S_{1,+}$ and $S_{1,-}$.
	\end{prop}

\begin{prop} \label{prop:S1S2disjoint} For $(\alpha, \beta) \in \mathcal{N}$, let  $S_2=A^{-1}_1(S_1)$. Then 	the  $\mathbb{C}$-spheres $S_1$ and   $S_2$ are disjoint except for point of tangency at the parabolic  fixed point of $A_0$. 
	
\end{prop}

See Figure 	\ref{figure:s0s1s2} for a  schematic picture of the $\mathbb{C}$-spheres $S_0$, $S_1$ and $S_2$.  Each of $S_{i}$ is decomposed into three parts  $S_{i,*}$ for $*=-,0,+$. 

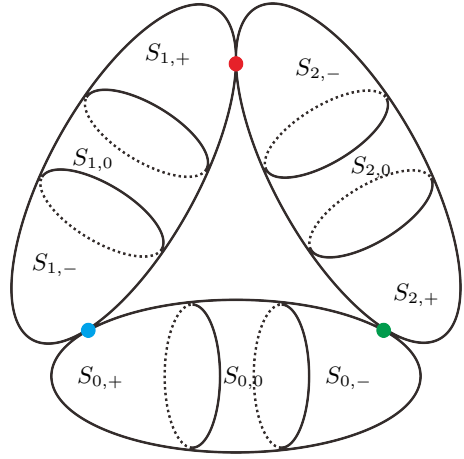
\begin{figure}
	\begin{center}
		\begin{tikzpicture}
		\node at (0,0) {\includegraphics[width=6cm,height=6cm]{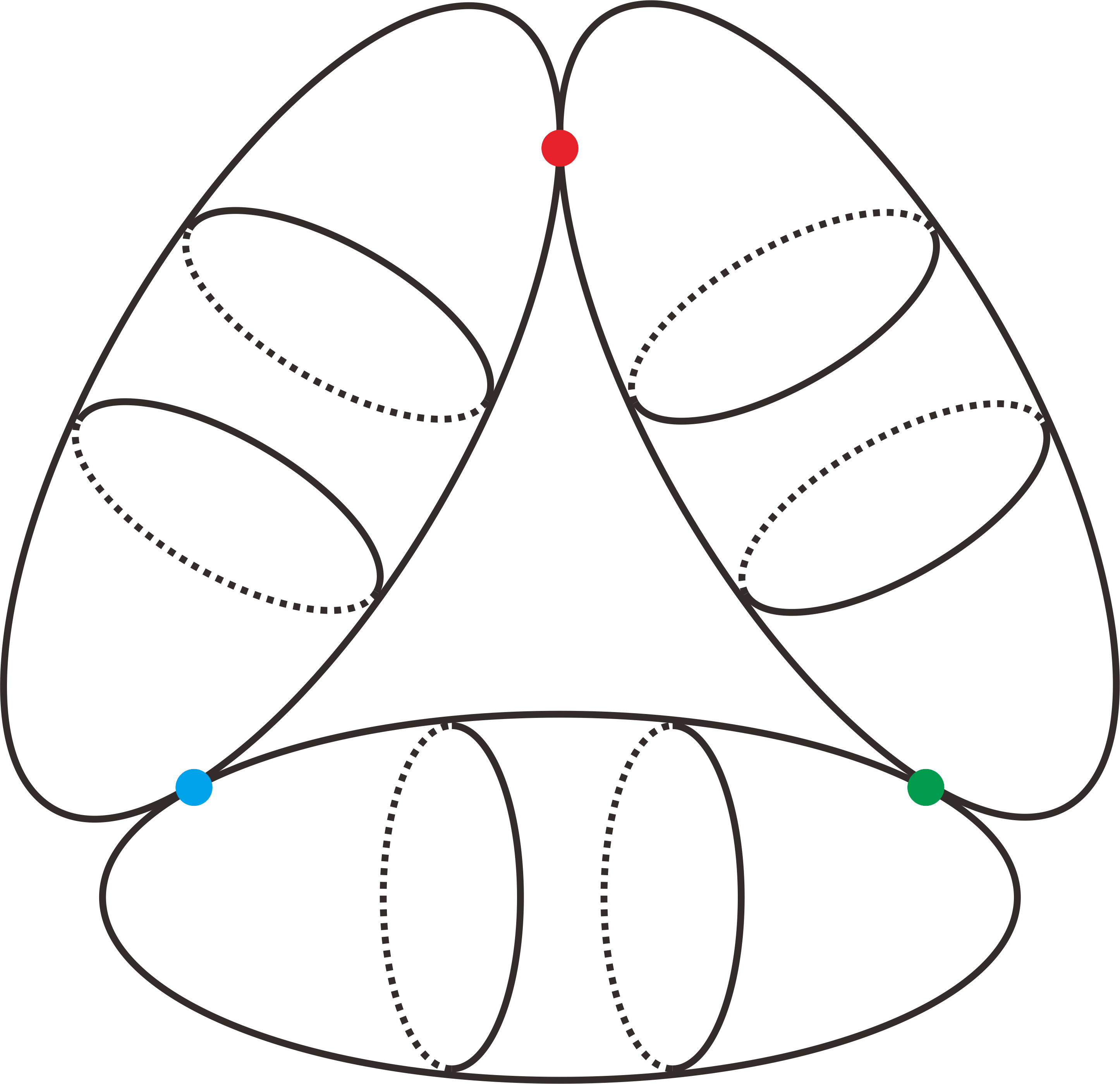}};
		
		\node at (-1.8,-2){\small $S_{0,+}$};
		\node at (0.1,-2){\small $S_{0,0}$};
		\node at (1.5,-2){\small $S_{0,-}$};
		
		\node at (-0.9,2.3){\small $S_{1,+}$};
		\node at (-1.9,0.9){\small $S_{1,0}$};
		\node at (-2.4,-0.5){\small $S_{1,-}$};
		
		\node at (2.4,-0.9){\small $S_{2,+}$};
		\node at (1.8,0.8){\small $S_{2,0}$};
		\node at (1.1,2.1){\small $S_{2,-}$};

		\end{tikzpicture}
	\end{center}
	\caption{A schematic picture of the $\mathbb{C}$-spheres $S_0$, $S_1$ and $S_2$. The red point is the  tangency point between  $S_1$ and $S_2$, which is the fixed point of $A_0=\iota_1 \iota_2$. The green point is the  tangency point between  $S_2$ and $S_0$, which is the fixed point of $\iota_2 \iota_0\iota_1\iota_2\iota_0 \iota_2$; The cyan point is the  tangency point between  $S_0$ and $S_1$, which is the fixed point of $\iota_0 \iota_2\iota_1 \iota_0$. The region outside the three spheres is the fundamental domain $D$.}
	\label{figure:s0s1s2}
\end{figure}

 Recall that a parabolic element $f$ acting on ${\bf H}^3_{\mathbb C}$ can be unipotent or ellipto-parabolic. In the later case, there is a $f$-invariant ${\bf H}^1_{\mathbb C} \hookrightarrow {\bf H}^3_{\mathbb C}$ or ${\bf H}^2_{\mathbb C} \hookrightarrow {\bf H}^3_{\mathbb C}$, which is the  axis of $f$. Proposition \ref{prop:A0elliptopara} bellow is a refinement of Proposition \ref{prop:A0para}, which  is crucial for the construction of the $\mathbb{C}$-sphere $S_1$. The author also notes the fact that  $z_{A_0,2}=\sqrt{2} x_{A_0,2}$ in  Proposition \ref{prop:A0elliptopara} is always real is very lucky. Without this fact, the construction of the $\mathbb{C}$-sphere $S_1$ would be  more difficult.

 
 
\begin{prop}  \label{prop:A0elliptopara}For  $\alpha \in (0, \frac{\pi}{2}]$ and $ \beta \in [0, \frac{\pi}{2}]$,  $A_0$ is ellipto-parabolic. Moreover, 
	\begin{itemize}
		\item If  \begin{equation}\label{equation:A0twoeigenvalue}
		\cos(2\beta)(\cos(\frac{2 \alpha}{3})+1)+\cos(\alpha)+\cos(\frac{\alpha}{3})=0,
		\end{equation}
		$A_0$ fixes a ${\bf H}^2_{\mathbb C} \hookrightarrow {\bf H}^3_{\mathbb C}$ invariant. The invariant  complex hyperbolic plane has polar vector 
		\begin{equation} \label{equation:A0cplane}
		(c, ~~~\rm{i} \sqrt{2},~~~\frac{\sqrt{-2\cos^2(\frac{2 \alpha}{3})+\cos(\frac{2 \alpha}{3})+1}}{\sin(\frac{\alpha}{3})(2\cos(\frac{\alpha}{3})-1)},~~~0)^{\top}.	\end{equation}
		Where $c$ is a  real number   depends on $\alpha$. It is not difficult to get the  explicit expression of $c$  (it occupies 2 lines),  but we omit it;
		\item If 
		 \begin{equation}\label{equation:A0threeeigenvalue}
		 \cos(2\beta)(\cos(\frac{2 \alpha}{3})+1)+\cos(\alpha)+\cos(\frac{\alpha}{3})\neq 0,
		  \end{equation}
		 $A_0$ fixes a ${\bf H}^1_{\mathbb C} \hookrightarrow {\bf H}^3_{\mathbb C}$ invariant. The invariant  complex hyperbolic line
		 has  $\mathbb{C}$-chain  
		\begin{equation}\label{equation:A0line}
		\left\{({\rm i} \sqrt{2} y_{1,A_0}, ~~~\sqrt{2} x_{2,A_0},~~t) \in \mathbb{C}^2 \times \mathbb{R}\right\}.
			\end{equation}
	Where 	$$y_{1,A_0}=\frac{ \sqrt{\cos(\alpha)} \cos(\beta) (\cos(\frac{\alpha}{2})+\cos(\frac{\alpha}{6}))}{\cos(2\beta)(\cos(\frac{2\alpha}{3})+1)+\cos(\alpha)+\cos(\frac{\alpha}{3})}$$
		and $$x_{2,A_0}=\frac{{\rm i} y_{1,A_0}\sin(2\beta) \rm{e}^{\frac{ \alpha \rm{i}}{3}}\cos(\frac{\alpha}{3})-\rm{i}\sqrt{\cos(\alpha)}\sin(\beta) \rm{e}^{\frac{\alpha \rm{i}}{6}} }{(1-\sin^2(\beta) \rm{e}^{\alpha \rm{i}})\rm{e}^{\frac{-\alpha \rm{i}}{3}}+\cos^2(\beta)}$$	are constants depend on  $\alpha,\beta$. Moreover, $x_{2,A_0} \in \mathbb{R}$.

\end{itemize}

\end{prop}
		

		

\begin{proof}
	By a   calculation, the eigenvalues of $A_0$ are $\rm{e}^{-\frac{\alpha\rm{i}}{3}}$ (a repeated eigenvalue) and 	\begin{equation} \label{equation:A0eigenvalue}
	-\rm{e}^{\frac{\alpha\rm{i}}{3}}[\cos(\frac{\alpha}{3})\cos(2\beta) \pm \rm{i} \sqrt{1-cos^2(2\beta)
		cos^2(\frac{\alpha}{3})}].	\end{equation}
	
	So if $\cos(2\beta)(\cos(\frac{2 \alpha}{3})+1)+\cos(\alpha)+\cos(\frac{\alpha}{3})=0$. Then  $A_0$ has  eigenvalues $\rm{e}^{-\frac{\alpha\rm{i}}{3}}$ (a multiplicity three eigenvalue) and $\rm{e}^{\alpha\rm{i}}$. Moreover, by a calculation, the eigenvector corresponds to  $\rm{e}^{\alpha\rm{i}}$ is a positive vector in the form of (\ref{equation:A0cplane}).
	
	 If $\cos(2\beta)(\cos(\frac{2 \alpha}{3})+1)+\cos(\alpha)+\cos(\frac{\alpha}{3})\neq 0$. Then  $A_0$ has three  eigenvalues, one of them has multiplicity two.
	Moreover, be directly calculation we   have  $$ {\rm e}^{\frac{\alpha \rm{i}}{3}} \cdot A_0  \cdot (-(y_{1,A_0}^2+x_{2,A_0}^2)+ {\rm i} t,~~ {\rm i} \sqrt{2} y_{1,A_0}, ~~\sqrt{2} x_{2,A_0},~~1)^{\top}$$ is  $$(-(y_{1,A_0}^2+x_{2,A_0}^2)+{\rm i}t', ~~{\rm i} \sqrt{2} y_{1,A_0}, ~~\sqrt{2} x_{2,A_0},~~1)^{\top}$$
	for some $t'$ depends on  $t, \alpha$ and $\beta$.
	In particular,  the $\mathbb{C}$-chain  in  (\ref{equation:A0line}) is $A_0$-invariant.
	
We note that the real part of $(\cos^2(\beta) \rm{e}^{\alpha \rm{i}}-\rm{e}^{\alpha \rm{i}}+1)\rm{e}^{\frac{-\alpha \rm{i}}{3}}+\cos^2(\beta)$ for $x_{2,A_0}$ is $$-\cos(\frac{2\alpha}{3})\sin^2(\beta)+\cos^2(\beta)+\cos(\frac{\alpha}{3}).$$ It is  always positive except when  $\alpha=0$ and $\beta=\frac{\pi}{2}$. So $x_{2,A_0}$ is always well-defined when (\ref{equation:A0threeeigenvalue}) holds. 
\end{proof}


\begin{rem}
	\begin{itemize}
		
	\item 
	 $A_0$ above  is type (iv) and  (i) parabolic element in  Proposition 5.2  of 	\cite{GongopadhyayPP} respectively;
	\item 
	 When $\beta=0$, we have $$y_{1,A_0}=\frac{\sqrt{\cos(\alpha)}}{2\cos(\frac{\alpha}{2})}, ~~~\text{and }~~x_{2,A_0}=0.$$ Which is exactly  the equation of the $\mathbb{C}$-chain in Page 74 of \cite{Falbelparker:2003};
	 
	 \item   $x_{2,A_0}$ is in fact always real, but its expression without  $y_{1,A_0}$ is more complicated (more than seven lines), so we express it in terms of  $y_{1,A_0}$;
	   
	   \item The point  $(\alpha,\beta)=(\frac{\pi}{2}, \beta_0)$ lies in 
	    (\ref{equation:A0twoeigenvalue}) has  $\cos(2 \beta_0)=-\frac{1}{\sqrt{3}}$.
	      When   $\alpha=\frac{\pi}{2}$ but $\beta \neq \beta_0$, then $y_{1,A_0}=x_{2,A_0}=0$. This fact is very important for our arguments later. 
	       
	\end{itemize}
	\end{rem}

		
		

\begin{prop} \label{prop:CchaincernterinR1} Let  $\mathcal{L}$ be a $\mathbb{C}^2$-chain with radius $R$ and  center $(b \rm{i},c,0)$. We assume the center $(b \rm{i},c,0)$ lies  in the $\mathbb{R}^3$-chain   $\mathcal{R}_1$, in particular  $b, c  \in \mathbb{R}$. Then $\mathcal{L}$ is invariant under the Lagrangian inversion $\iota_1$.

\end{prop}

\begin{proof}The polar vector of $\mathcal{L}$ is  $$n=[R^2-(b^2 +c^2),~~ \sqrt{2}b \rm{i},~~ \sqrt{2}c,~~1]^{\top} \in  {\mathbb C}{\mathbf P}^{3}-  \overline{{\bf H}^3_{\mathbb C}}.$$ From the presentation matrix $M_1$ of $\iota_1$ in (\ref{matrix:M1}) of Section \ref{sec:modulispace}, then $\iota_1(n)=n$. So $\mathcal{L}$ is invariant under  $\iota_1$. 
\end{proof}

We now consider the $\iota_0$-action on $\mathcal{L}$ (equivalently, $\iota_0$-action   on the  polar vector  $n$). A  calculation with the presentation matrix $M_0$ of $\iota_0$  in (\ref{matrix:M0}) of  Section \ref{sec:modulispace} shows that $\iota_0(\mathcal{L})$ is  a $\mathbb{C}^2$-chain with radius $\frac{R}{|R^2-(b^2+c^2)|}$ and  center $$(\frac{-b \rm{i}}{R^2-(b^2+c^2)},~~\frac{c}{R^2-(b^2+c^2)},0) \in \mathcal{R}_1.$$
We shall take  a path $\sigma$ in the 3-space $$\left \{(R, b,c) \in \mathbb{R}_{>0} \times \mathbb{R}^2 \right \},$$ which is invariant under the involution $$(R, b,c) \longrightarrow (\frac{R}{|R^2-(b^2+c^2)|},~~\frac{-b }{R^2-(b^2+c^2)},~~\frac{c}{R^2-(b^2+c^2)}).$$
From  the path $\sigma$ we get a $\mathbb{C}$-sphere (which is a union of $\mathbb{C}^2$-chains) in $\mathbb{C}^2\times \mathbb{R}$ which is invariant under $\iota_0\iota_{1}$. 

For any pair $(\alpha, \beta)   \in [0, \frac{\pi}{2}]^2$  satisfies (\ref{equation:A0threeeigenvalue}),  let $\mathcal{L}_{\lambda}$ be the finite  $\mathbb{C}^2$-chain with polar vector 
$$p_{\lambda}=( \lambda, ~~{\rm i} \sqrt{2}  y_{1,A_0},~~ \sqrt{2}  x_{2,A_0}, ~~1)^{\top}.$$
Take  \begin{equation} \label{equation:S1plus}
S_{1,+}=\cup_{\lambda \in [2, \infty]} \mathcal{L}_{\lambda}.
\end{equation}
Here $\mathcal{L}_{\infty}$ is just the point $\infty$. 
Take  
\begin{equation} \label{equation:S2minus}
	S_{2,-}=\iota_2(S_{1,+}).
	\end{equation} 
	See Figure 	\ref{figure:s0s1s2} for a   schematic picture of  these kissing topological 4-balls.


	First note that the real numbers   $y_{1,A_0}$ and $x_{2,A_0}$ are continuous on $(\alpha, \beta)$ when $(\alpha, \beta)$  satisfies (\ref{equation:A0threeeigenvalue}). Moreover when 
 $\alpha=\frac{\pi}{2}$ and $\beta \neq \beta_0$, $y_{1,A_0}=x_{2,A_0}=0$. So we may take a neighborhood $\mathcal{N}$ of $(\frac{\pi}{2},0)$ or $(\frac{\pi}{2},\frac{\pi}{2})$ in $[0,\frac{\pi}{2}]^2$, such that  any pair $(\alpha, \beta)   \in \mathcal{N}$  satisfies (\ref{equation:A0threeeigenvalue}),  moreover we may assume 
  $$y^2_{1,A_0}+x^2_{2,A_0}<1$$ when  $(\alpha, \beta) \in \mathcal{N}$.

\begin{prop} \label{prop:S1plus} For any pair $(\alpha, \beta) \in \mathcal{N}$, both $S_{1,+}$ and  $S_{2,-}$ are  topologically   4-balls in $\mathbb{C}^2 \times \mathbb{R} \cup \infty$. Moreover,  $S_{1,+} \cap S_{2,-}$ is exactly the point $ \infty$.
\end{prop}
	\begin{proof}
		Consider $\mathbb{C}^{2} \times \mathbb{R}$ with coordinates $(x_1+y_1 {\rm i}, x_2+y_2 {\rm i},t)$. 
		By Proposition  \ref{prop:cplane}, for  various $\lambda \in [2,\infty)$,  the  $\mathbb{C}^2$-chain $\mathcal{L}_{\lambda}$ in  $S_{1,+}$ is exactly   the intersection of the fixed hyperplane with equation 
		\begin{equation} \label{equation:S1plushyperplane}
		t+2 \cdot y_{1,A_0}\cdot y_1-2 \cdot x_{2,A_0} \cdot x_2=0 
		\end{equation}
		and   the cylinder with equation $$x_1^2+(y_1- y_{1,A_0})^2+(x_2- x_{2,A_0})^2+y_2^2=R^2,$$
		here $$R^2-(y^2_{1,A_0}+x^2_{1,A_0})=\lambda.$$  In particular,  $\mathcal{L}_{\lambda}$ and  $\mathcal{L}_{\lambda'}$ are disjoint for different $\lambda, \lambda' \in [2, \infty)$. Then $S_{1,+}$ is the compactification of $\mathbb{S}^3 \times [0, \infty)$ by the infinite point $\mathcal{L}_{\infty}=\infty$. 
		So $S_{1,+}$ is topologically a  4-ball. 
		
		Moreover, since $(\sqrt{2}  y_{1,A_0} {\rm i},~~ \sqrt{2}  x_{2,A_0},~~ 0)$ in the Heisenberg group  is fixed by $\iota_{1}$. From the action of  $A^{-1}_{0}=\iota_{2} \iota_{1}$, consider the  $\mathbb{C}^2$-chain $\iota_2(\mathcal{L}_{\lambda})$ in  $\iota_{2}(S_{1,+})=S_{2,-}$, its   polar vector also lies in the axis of $A_0$. So the polar vector  of $\iota_2(\mathcal{L}_{\lambda})$ is 
		$$(\lambda', ~~{\rm i}\sqrt{2}  y_{1,A_0},~~ \sqrt{2}  x_{2,A_0}, ~~1)^{\top}$$
		for some $\lambda'$ depends on $\alpha, \beta$ and $\lambda$. We assume  $\iota_2(\mathcal{L}_{\lambda})$  has radius $R'$, then  $$\lambda'=(R')^2-( y^1_{2,A_0}+ x^2_{2,A_0})$$ by Proposition \ref{prop:center}. So  $\iota_2(\mathcal{L}_{\lambda})$
		is exactly  the intersection of the fixed hyperplane with equation 
			\begin{equation} \label{equation:S2minushyperplane}
			t+2 \cdot y_{1,A_0}\cdot y_1-2 \cdot x_{2,A_0} \cdot x_2-t^{*}=0
				\end{equation} 
and  the  cylinder $$x_1^2+(y_1- y_{1,A_0})^2+(x_2- x_{2,A_0})^2+y_2^2=(R')^2. $$ 
		Where from the presentation matrix of $\iota_{2}$,   
			$$t^{*}=2\sqrt{\cos(\alpha)}[-\sin(\frac{\alpha}{2})\cos(\beta) y_{1,A_{0}}+\cos(\frac{\alpha}{2})\sin(\beta) x_{2,A_{0}}]+\sin(\alpha).
			$$ 
			The constant $t^{*}$ depends continuous on $\alpha$ and $\beta$. But when $\alpha=\frac{\pi}{2}$, $t^{*}=1$. So we may take the small neighborhood $\mathcal{N}$ of $(\frac{\pi}{2},0)$ or $(\frac{\pi}{2},\frac{\pi}{2})$, such that  $t^{*} \geq \frac{1}{2}$ when $\alpha,\beta \in \mathcal{N}$.
	By  (\ref{equation:S1plushyperplane})  and (\ref{equation:S2minushyperplane}), except the point $\infty$,  $\iota_2(S_{1,+})$ and $S_{1,+}$ lie in disjoint parallel  hyperplanes, so they are disjoint. 
		 
\end{proof}

Take $$S_{1,-}=\iota_{0}(S_{1,+}).$$
 From Proposition \ref{prop:S1plus}, $S_{1,-}$ is a (finite) 4-ball in  $\mathbb{C}^2 \times \mathbb{R}$.  Proposition \ref{prop:S1minus}  shows that $S_{1,-}$ and $S_{1,+}$ are disjoint.


\begin{prop}\label{prop:S1minus} For any pair $(\alpha, \beta) \in \mathcal{N}$, and $\lambda \geq 2$,  $\mathcal{L}_{\lambda}$ and $\iota_0(\mathcal{L}_{\lambda})$ are disjoint finite  $\mathbb{C}^2$-chains  in $\mathbb{C}^2 \times \mathbb{R}$. Moreover, the complex hyperbolic planes  associated to them  in ${\bf H}^{3}_{\mathbb C}$ are hyper-parallel. 
\end{prop}

\begin{proof}
	For $\lambda \geq 2$, we normalize the polar vector of $\mathcal{L}_{\lambda}$ into 
	\begin{equation}\label{polarS1+}
	p_\lambda=\frac{1}{\sqrt{2 (\lambda+y^2_{1,A_0}+x^2_{2,A_0})}}\left(
	\begin{array}{c}
	\lambda 
	\\   [2 ex]
	{\rm i} y_{1,A_0}\sqrt{2} \\   [2 ex]
	x_{2,A_0}\sqrt{2} \\
	1
	\\
	\end{array}
	\right),
	\end{equation} with $\langle p_{\lambda},p_{\lambda}\rangle=1$. Then for $0 < \lambda \leq \frac{1}{2}$, we take
	\begin{equation}\label{polarS1-}
	p_\lambda=\frac{1}{\sqrt{2 \lambda (1+\lambda y^2_{1,A_0}+\lambda 
			x^2_{2,A_0})}} \left(
	\begin{array}{c}
	\lambda 
	\\   [0 ex]
	-{\rm i} y_{1,A_0}\sqrt{2} \lambda\\   [0 ex]
	x_{2,A_0}\sqrt{2} \lambda \\
	1
	\\
	\end{array}
	\right),
	\end{equation}
	with $\langle p_{\lambda},p_{\lambda}\rangle=1$.
	It can be checked directly that $$[\iota_{0}(p_{\lambda})]=[p_{\frac{1}{\lambda}}] \in \mathbb{C}{\bf P}^{3}- \overline{{\bf H}^{3}_{\mathbb C}}$$ when $\lambda \geq 2$.
	For $\lambda \geq 2$, $$ \langle p_\lambda,   p_{\lambda^{-1}}\rangle=\displaystyle{\frac{\lambda^2+1-2y^2_{1,A_0}+2x^2_{2,A_0}}{  2(\lambda+y^2_{1,A_0}+x^2_{2,A_0})}}.$$
When 	$\alpha =\frac{\pi}{2}$, and $ \beta \neq \beta_0$,  $y_{1,A_0}=x_{2,A_0}=0$, and $ \langle p_\lambda,  p_{\lambda^{-1}}\rangle =\frac{\lambda^2+1}{2 \lambda} \geq \frac{5}{4}$.
So we may take a neighborhood $\mathcal{N}$ of $(\frac{\pi}{2}, 0)$ or $(\frac{\pi}{2}, \frac{\pi}{2})$ such that   $$\langle p_\lambda,  p_{\lambda^{-1}}\rangle  \geq \frac{9}{8}>1$$ when  $(\alpha, \beta) \in \mathcal{N}$. This implies the complex hyperbolic planes  in ${\bf H}^{3}_{\mathbb C}$  with polars $ p_{\lambda}$ and  $p_{\lambda^{-1}}$  are hyper-parallel for any $\lambda \geq 2$.
	
\end{proof}


 \begin{prop}\label{prop:paracsliceofbisector} (Exercise 5.1.9 and Theorem 5.2.4 of 	\cite{Go}) Given two hyper-parallel  complex hyperbolic planes  $H_1$ and $H_2$ in ${\bf H}^{3}_{\mathbb C}$  with polars $u_1$ and $u_2$ respectively. We may assume $$\langle u_1, u_1\rangle =1,~~~\langle u_2, u_2\rangle =1,~~~ \langle u_1, u_2\rangle =a >1.$$
	For $s\in (0, \infty)$, let $$n(s)=\frac{s (u_1-(a-\sqrt{a^2-1})u_2)-s^{-1} (u_1-(a+\sqrt{a^2+1})u_2)}{2 \sqrt{a^2-1}}.$$ Then there is a  bisector $B$ which contains  $H_1$ and $H_2$  as slices, such that
	 $n(s)$ is a polar of  a complex slice of  the bisector $B$. Moreover,
	\begin{itemize}

			\item $\langle n(s), n(s) \rangle =1$ for  $s\in (0, \infty)$;
		
		\item $n(1)=u_1$;
		
		\item $n(a+\sqrt{a^2-1})=u_2$.
			\end{itemize}  So the portion in the bisector $B$ co-bound by  $H_1$ and $H_2$ can be parameterzied by polars $n(s)$ with $s \in [1, a+\sqrt{a^2-1}]$. 
\end{prop}

\begin{rem} There is a typo in  Exercise 5.19 of Page 155 in \cite{Go},  it should be $\frac{t Q_{-}-t^{-1}Q_{+}}{2}$ instead of $\frac{t Q_{-}+t^{-1}Q_{+}}{2}$ over there. 
\end{rem}

Now let $$a:=\langle p_2,  p_{\frac{1}{2}}\rangle=
\displaystyle{\frac{5-2y^2_{1,A_0}+2x^2_{2,A_0}}{  2(2+y^2_{1,A_0}+x^2_{2,A_0})}}.$$
When $\alpha=\frac{\pi}{2}$ and $\beta\neq \beta_0$, $y_{1,A_0}=x_{2,A_0}=0$, and  $a=\frac{5}{4}$ in turn. So we may assume there is a neighborhood $\mathcal{N}$ of $(\frac{\pi}{2}, 0)$ or $(\frac{\pi}{2}, \frac{\pi}{2})$ such that $a \geq \frac{9}{8}$ when $(\alpha,\beta) \in \mathcal{N}$.
Take a  homeomorphism $$\phi:[\frac{1}{2},2]\longrightarrow[1, a+\sqrt{a^2-1}]$$
such that $\phi(\frac{1}{2})=1$ and  $\phi(2)=a+\sqrt{a^2-1}$. Here we take the linear  homeomorphism  $$\phi(x)=\frac{2(a-1+\sqrt{a^2-1})}{3} \cdot x-\frac{a-1+\sqrt{a^2-1}}{3}+1.$$
For
$\lambda \in [\frac{1}{2}, 2]$, let $p_\lambda$ be 
\begin{equation}\label{polarS1zero}
\displaystyle{\frac{\phi(\lambda) ( p_{2}  -(a-\sqrt{a^2-1})p_{\frac{1}{2}})-\phi(\lambda)^{-1} ( p_{2} -(a+\sqrt{a^2-1})p_{\frac{1}{2}} )}{2 \sqrt{a^2-1}}}, 
\end{equation}
which is a positive vector 
by Proposition \ref{prop:paracsliceofbisector}.
In particular, when $\lambda=\frac{1}{2}$, $p_{\frac{1}{2}}$ defined in (\ref{polarS1zero})  is the same as $p_{\frac{1}{2}}$ defined in (\ref{polarS1+}). Moreover, $p_{2}$ defined in (\ref{polarS1zero}) is the same as $p_{2}$ defined early in (\ref{polarS1-}).

From Propositions \ref{prop:S1minus}, \ref{prop:paracsliceofbisector} above,  and Theorem 5.2.4 of 	\cite{Go}, there is a unique bisector $B$ in ${\bf H}^{3}_{\mathbb C}$, such that $\mathcal{L}_{2}$ and $\iota_0(\mathcal{L}_{2})$ are $\mathbb{C}^2$-chains of two  complex slices of $B$.  $\partial B$ is a union of $\mathbb{C}^2$-chains by Subsection \ref{subsection:bisector}.
We denote by $S_{1,0}$ the part of $ \partial B$ co-bound by $\mathcal{L}_{2}$ and $\iota_0(\mathcal{L}_{2})=\mathcal{L}_{\frac{1}{2}}$.  The $\mathbb{C}^2$-chains of  $S_{1,0}$ are parameterized  in (\ref{polarS1zero}).  $S_{1,0}$  is topologically $\mathbb{S}^{3} \times [-1,1]$ with one common 3-sphere ($\mathbb{C}^2$-chain) with $S_{1,+}$ and $ S_{1,-}$ respectively. 
Let $$S_1=S_{1,+} \cup S_{1,0} \cup S_{1,-}$$ 
be a $\mathbb{C}$-sphere in $\mathbb{C}^2 \times \mathbb{R} \cup \infty$.


 \begin{prop}\label{prop:S1embed} The $\mathbb{C}$-sphere $S_1$ is an embedded $4$-sphere in $\mathbb{C}^2 \times \mathbb{R} \cup \infty$. $\iota_0(S_{1,0})=S_{1,0}$, $\iota_0$ exchanges $S_{1,+}$ and $S_{1,-}$. So  $S_1$ is  $c_1=A_2=\iota_1\iota_0$ invariant.
 
 \end{prop}
 \begin{proof} Since all of these polars are normalized with norm one,  we need only to show that $|\langle p_\lambda, p_\mu \rangle| >1$ for $\lambda \in [0, \frac{1}{2}] \cup [2, \infty]$ and $\mu \in [\frac{1}{2},2]$. By a direct calculation, for $\lambda \in [2, \infty]$ and  $\mu \in [\frac{1}{2},2]$, we have $\langle p_\lambda, p_\mu \rangle$ is 
 	 	\begin{flalign}\label{equation:S1sphere} &   \frac{(2+2y^2_{1,A_0}+2x^2_{2,A_0}+\lambda)\cdot (\phi(\mu)-\frac{1}{\phi(\mu)})}{4\sqrt{(\lambda+y^2_{1,A_0}+x^2_{2,A_0})(2+y^2_{1,A_0}+x^2_{2,A_0})(a^2-1)}} & \nonumber \\
 	&+\frac{(1-2y^2_{1,A_0}+2x^2_{2,A_0}+2\lambda) \cdot (\frac{1}{\phi(\mu)}(a+\sqrt{a^2-1})-\phi(\mu)(a-\sqrt{a^2-1}))}{4\sqrt{(\lambda+y^2_{1,A_0}+x^2_{2,A_0})(1+y^2_{1,A_0}+x^2_{2,A_0})(a^2-1)}}.&\nonumber
 	\end{flalign}
 Take a neighborhood $\mathcal{N}$ of  $(\frac{\pi}{2}, 0)$ or $(\frac{\pi}{2}, \frac{\pi}{2})$, we may assume  $a \in [\frac{9}{8}, \frac{3}{2}]$ when $(\alpha,\beta)\in \mathcal{N}$.
 	Then  $\langle p_\lambda, p_\mu \rangle$ is increasing on  $\lambda \in [2, \infty]$, and  decreasing on   $\mu \in [\frac{1}{2},2]$ when $(\alpha,\beta) \in \mathcal{N}$.
 	Moreover, when $\lambda=\mu=2$, it can be checked directly  $\langle p_\lambda, p_\mu \rangle=1$.
 Then we have $S_{1,0}$ and $S_{1,+}$ intersect only at the common $\mathbb{C}^2$-chain $\mathcal{L}_{2}$. 
 
Moreover, since $S_{1,0}$ is $\iota_0$-invariant, and $S_{1,-}=\iota_{0}(S_{1,+})$. The parts $S_{1,0}$ and $S_{1,-}$ intersect only at the common slice $\mathcal{L}_{\frac{1}{2}}$. 

This ends the proof of Proposition \ref{prop:S1embed}. Which in turn implies Proposition \ref{prop:S1}.
 \end{proof}



Let $$S_2=A^{-1}_1(S_1),~~~ \text{and}~~S_0=A_1(S_1).$$ Both of them are embedded 4-spheres in $\mathbb{C}^2 \times \mathbb{R} \cup \infty$. In particular, let $$S_{2,*}=A^{-1}_1(S_{1,*}), ~~ \text{and}~~ S_{0,*}=A_1(S_{1,*})$$ for $*=+, -, 0$. Note that  $$S_{2,-}= \iota_{2} (S_{1,+})=\iota_{2}\iota_{0} (S_{1,-}),$$ and  $\infty$ is a tangency point of $S_{1,+}$ and $S_{2,-}$.

Since there is a  $\mathbb{Z}_3$-symmetry of the $\mathbb{C}$-spheres  $$\{S_0, ~~S_1,~~S_2\},$$ and  there is a  $\mathbb{Z}_2$-symmetry for each of $S_i$. The following proposition implies that $S_{i}$ and $S_{j}$ are disjoint except the  tangency points. For example, $S_{1,0}$ is invariant under $\iota_{0}$, and $$\iota_{0}(S_{2,+} \cap S_{1,0})=S_{0,-} \cap S_{1,0}.$$ So  $S_{2,+} \cap S_{1,0}=\emptyset$ implies $S_{0,-}\cap S_{1,0}=\emptyset$, and by $\mathbb{Z}_3$-symmetry, $$S_{1,-} \cap S_{2,0}=S_{2,-} \cap S_{0,0}=\emptyset.$$

\begin{prop}\label{prop:disjoint} There is a neighborhood $\mathcal{N}$ of $(\frac{\pi}{2},0)$ or $(\frac{\pi}{2},\frac{\pi}{2})$ in $[0,\frac{\pi}{2}]^2$, such that we  have the following disjoint properties of portions of  $\mathbb{C}$-spheres:	
	\begin{enumerate}
		
		\item \label{disjoint2+1-}  $S_{2,+}$ and $S_{1,-}$ are disjoint except the  tangency point $\infty$;
		
		\item \label{disjoint2+10}  $S_{2,+}$ and $S_{1,0}$ are disjoint;
		
			\item  \label{disjoint2+1+}  $S_{2,+}$ and $S_{1,+}$ are disjoint;
			
				\item \label{disjoint2010}   $S_{2,0}$ and $S_{1,0}$ are disjoint;
				
					\item \label{disjoint201+}   $S_{2,0}$ and $S_{1,+}$ are disjoint;
					
						\item \label{disjoint2-1+}  $S_{2,-}$ and $S_{1,+}$ are disjoint;

		\end{enumerate}
\end{prop}
\begin{proof}To prove Item (\ref{disjoint2+1-}) of Proposition \ref{prop:disjoint}, we consider $p_{\lambda}$ for $\lambda\in [0,\frac{1}{2}]$ in (\ref{polarS1-}), which is the polar of a $\mathbb{C}^2$-chain in $S_{1,-}$, and $q_{\mu}=A^{-1}_1 (p_{\mu})$ for $\mu \in [2,\infty]$ as in (\ref{polarS1+}), 
	which is the polar of a $\mathbb{C}^2$-chain in $S_{2,+}$. We note that $p_{\lambda}$ and $q_{\mu}$ also depend on $(\alpha,\beta)$.
	Now $\langle p_{\lambda}, q _{\mu} \rangle $ is too complicated to write down. But when $\alpha=\frac{\pi}{2}$ and $\beta\neq \beta_0$, then  $$|\langle p_{\lambda}, q _{\mu} \rangle|^2= \frac{\lambda^2 \mu^2+2\lambda \mu+ \mu^2+1}{4 \lambda \mu}.$$
	Which is strictly  bigger than 1  when $\lambda\in [0,\frac{1}{2}]$ and  $\mu \in [2,\infty]$.  Recall that we have normalized the polars  such that  $p_{\lambda}$ and $q_{\mu}$ have norm 1. We can take a neighborhood $\mathcal{N}$ of $(\frac{\pi}{2},0)$ or $(\frac{\pi}{2},\frac{\pi}{2})$ in $[0,\frac{\pi}{2}]^2$ such that $\langle p_{\lambda}, q _{\mu} \rangle $   is strictly  bigger than 1 when $(\alpha,\beta) \in \mathcal{N}$. This implies the  complex hyperbolic planes  with polars  $p_{\lambda}$ and $q_{\mu}$  are hyper-parallel  when $(\alpha,\beta) \in \mathcal{N}$.  So the $\mathbb{C}^2$-chains correspond to them are disjoint.

	To prove Item (\ref{disjoint2+10}) of Proposition \ref{prop:disjoint}, we consider $p_{\lambda}$ for $\lambda\in [\frac{1}{2},2]$ in (\ref{polarS1zero}), which is the polar of a $\mathbb{C}^2$-chain in $S_{1,0}$, and $q_{\mu}=A^{-1}_1 (p_{\mu})$ for $\mu \in [2,\infty]$  in (\ref{polarS1+}), 
	which is the polar of a $\mathbb{C}^2$-chain in $S_{2,+}$.
	When $\alpha=\frac{\pi}{2}$ and $\beta\neq \beta_0$, then  $$|\langle p_{\lambda}, q _{\mu} \rangle|^2= \frac{81+(4\lambda^4+16\lambda^3+24\lambda^2+16\lambda+85) \mu^2+36(1+\lambda)^2\mu}{72 \mu(1+\lambda)^2}.$$
	Which is at least 2  when $\lambda\in [\frac{1}{2},2]$ and  $\mu \in [2,\infty]$.   We can take a  neighborhood $\mathcal{N}$ of $(\frac{\pi}{2},0)$ or $(\frac{\pi}{2},\frac{\pi}{2})$ in $[0,\frac{\pi}{2}]^2$ such that $\langle p_{\lambda}, q _{\mu} \rangle $   is strictly  bigger than 1 when $(\alpha,\beta) \in \mathcal{N}$. This implies the  complex hyperbolic planes  with polars  $p_{\lambda}$ and $q_{\mu}$  are hyper-parallel  when $(\alpha,\beta) \in \mathcal{N}$.

		To prove Item (\ref{disjoint2+1+}) of Proposition \ref{prop:disjoint}, we consider $p_{\lambda}$ for $\lambda\in [2, \infty]$ in (\ref{polarS1+}), which is the polar of a $\mathbb{C}^2$-chain in $S_{1,+}$, and $q_{\mu}=A^{-1}_1 (p_{\mu})$ for $\mu \in [2,\infty]$  in (\ref{polarS1+}), 
	which is the polar of a $\mathbb{C}^2$-chain in $S_{2,+}$.
	When $\alpha=\frac{\pi}{2}$ and $\beta\neq \beta_0$, then   $$|\langle p_{\lambda}, q _{\mu} \rangle|^2= \frac{\lambda^2 \mu^2+2\lambda \mu+ \mu^2+1}{4 \lambda \mu}.$$
	Which is at least $\frac{25}{16}$  when $\lambda\in [2, \infty]$ and  $\mu \in [2,\infty]$. 
	We can take a neighborhood $\mathcal{N}$ of $(\frac{\pi}{2},0)$ or $(\frac{\pi}{2},\frac{\pi}{2})$ in $[0,\frac{\pi}{2}]^2$ such that $\langle p_{\lambda}, q _{\mu} \rangle $   is strictly  bigger than 1 when $(\alpha,\beta) \in \mathcal{N}$. This implies the  complex hyperbolic planes  with polars  $p_{\lambda}$ and $q_{\mu}$  are hyper-parallel  when $(\alpha,\beta) \in \mathcal{N}$. 
	
		To prove Item (\ref{disjoint2010}) of Proposition \ref{prop:disjoint}, we consider $q_{\mu}=A^{-1}_1(p_{\mu})$ for $\mu\in [\frac{1}{2}, 2]$ in (\ref{polarS1zero}), which is the polar of a $\mathbb{C}^2$-chain in $S_{2,0}$, and $p_{\lambda}$ for $\lambda \in [\frac{1}{2},2]$  in (\ref{polarS1zero}), 
	which is the polar of a $\mathbb{C}^2$-chain in $S_{1,0}$.
When $\alpha=\frac{\pi}{2}$ and $\beta\neq \beta_0$, then   $|\langle p_{\lambda}, q _{\mu} \rangle|^2$ is
	\begin{flalign}\nonumber &  \frac{16\lambda^4\mu^4+64\lambda^4\mu^3+64\lambda^3\mu^4+96\lambda^4\mu^2+256\lambda^3\mu^3+96\lambda^2\mu^4+64\lambda^4\mu+384\lambda^3\mu^2}{1296(1+\lambda)^2(\mu+1)^2}+& \nonumber \\
	\nonumber & \frac{384\lambda^2\mu^3+64\lambda\mu^4+16\lambda^4+256\lambda^3\mu+1224\lambda^2 \mu^2+256\lambda\mu^3+340\mu^4+64\lambda^3}{1296(1+\lambda)^2(\mu+1)^2} +& \nonumber \\
&\frac{1680\lambda^2\mu+1680\lambda\mu^2+1360\mu^3+744\lambda^2+2848\lambda\mu+2688\mu^2+1360\lambda+2656\mu+7549}{1296(1+\lambda)^2(\mu+1)^2}.&\nonumber
\end{flalign}
	Which is strictly bigger than 1   when $\lambda, 
\mu \in [\frac{1}{2}, 2]$.
We can take a neighborhood $\mathcal{N}$ of $(\frac{\pi}{2},0)$ or $(\frac{\pi}{2},\frac{\pi}{2})$ in $[0,\frac{\pi}{2}]^2$ such that $\langle p_{\lambda}, q _{\mu} \rangle $   is strictly  bigger than 1 when $(\alpha,\beta) \in \mathcal{N}$. This implies the  complex hyperbolic planes  with polars  $p_{\lambda}$ and $q_{\mu}$  are hyper-parallel  when $(\alpha,\beta) \in \mathcal{N}$.


	To prove Item (\ref{disjoint201+}) of Proposition \ref{prop:disjoint}, we consider $q_{\mu}=A^{-1}_1(p_{\mu})$ for $\mu\in [\frac{1}{2}, 2]$ in (\ref{polarS1zero}), which is the polar of a $\mathbb{C}^2$-chain in $S_{2,0}$, and $p_{\lambda}$ for $\lambda \in [2,\infty]$ in (\ref{polarS1+}), 
	which is the polar of a $\mathbb{C}^2$-chain in $S_{1,+}$.
When $\alpha=\frac{\pi}{2}$ and $\beta\neq \beta_0$, then $|\langle p_{\lambda}, q _{\mu} \rangle|^2$ is 
\begin{equation}\nonumber
\frac{4(\lambda^2+1)\mu^4+16(\lambda^2+1)\mu^3+6(4\lambda^2+6
		\lambda+4)\mu^2+(16\lambda^2+72\lambda+16)\mu+4\lambda^2+36\lambda+85}{72\lambda(\mu+1)^2}.
\end{equation}
	Which is strictly bigger than 1   when $
	\mu \in [\frac{1}{2}, 2]$ and  $\lambda \in [2,\infty]$. We can take neighborhood $\mathcal{N}$ of $(\frac{\pi}{2},0)$ or $(\frac{\pi}{2},\frac{\pi}{2})$ in $[0,\frac{\pi}{2}]^2$ such that $\langle p_{\lambda}, q _{\mu} \rangle $   is strictly  bigger than 1 when $(\alpha,\beta) \in \mathcal{N}$. This implies the  complex hyperbolic planes  with polars  $p_{\lambda}$ and $q_{\mu}$  are hyper-parallel  when $(\alpha,\beta) \in \mathcal{N}$. 
	
We have proved  Item (\ref{disjoint2-1+}) in  Proposition \ref{prop:S1plus}.

\end{proof}

From Propositions \ref{prop:disjoint}, we end the proof of  Proposition \ref{prop:S1S2disjoint}. There is only one component of $$\partial {\bf H}^3_{\mathbb C}-(S_0\cup S_1 \cup S_2)$$ which is adjacent to all of $S_i$ for $i=0,1,2$, we denote the closure of this component by $D$. By  Poincar\'e polyhedron theorem, $D$ is a fundamental domain of $\langle \iota_0, \iota_2,\iota_2 \rangle $, we get Theorem \ref{thm:cspheres}. So we end   the proof of Theorem \ref{thm:modulardiscrete} via $\mathbb{C}$-spheres.

 We can show a result which is  a little stronger than  Theorem \ref{thm:modulardiscrete}.
 
\begin{cor}\label{cor:modulardiscrete} For the moduli space  $\mathcal{M}(0,\frac{2\pi}{3},\frac{4\pi}{3})$,  for any pair  $(\alpha,\beta)=( \frac{\pi}{2}, \beta)$, and $\beta \neq \beta_0$ with  $\cos(2 \beta_0)=-\frac{1}{\sqrt{3}}$, we have a neighborhood $\mathcal{N}$ of it in  $[0, \frac{\pi}{2}]^2$, such that the corresponding representation of any point in $\mathcal{N}$ is discrete and faithful.
\end{cor}

\begin{proof}This is the same as the proof of Theorem \ref{thm:cspheres} (and Theorem \ref{thm:modulardiscrete}). For  any pair  $(\alpha,\beta)=( \frac{\pi}{2}, \beta)$, and $\beta \neq \beta_0$ with  $\cos(2 \beta_0)=-\frac{1}{\sqrt{3}}$, we have $y_{1,A_0}=x_{2,A_0}=0$. So we can take  a neighborhood $\mathcal{N}$ of $(\frac{\pi}{2}, \beta)$ such that $y_{1,A_0}$ and $x_{2,A_0}$ are small enough. We then construct $\mathbb{C}$-spheres $S_0,S_1,S_2$ as  the proof of Theorem \ref{thm:cspheres}. Since $y_{1,A_0}$ and $x_{2,A_0}$ are small enough,   properties as in  Propositions \ref{prop:S1embed} and  \ref{prop:disjoint} hold.  
	 
\end{proof}


\bibliographystyle{amsplain}

\end{document}